\documentclass[onefignum,onetabnum]{siamart190516}

\usepackage{lipsum}
\usepackage{amsfonts}
\usepackage{graphicx}
\usepackage{epstopdf}
\usepackage{algorithmic}
\ifpdf
  \DeclareGraphicsExtensions{.eps,.pdf,.png,.jpg}
\else
  \DeclareGraphicsExtensions{.eps}
\fi

\newsiamremark{remark}{Remark}
\newsiamremark{hypothesis}{Hypothesis}
\crefname{hypothesis}{Hypothesis}{Hypotheses}
\newsiamthm{claim}{Claim}

\headers{First-Kind Boundary Integral Equations for the Dirac Operator}{E. Schulz and R. Hiptmair}

\title{First-Kind Boundary Integral Equations for the Dirac Operator in 3D Lipschitz domains\thanks{First version submitted on December 30, 2020.
\funding{The work of Erick Schulz was supported by SNF as part of the grant 200021\_184848/1.}}}

\author{Erick Schulz\thanks{Seminar in Applied Mathematics, ETH, Zurich, Switzerland 
  (\email{erick.schulz@sam.math.ethz.ch}, \email{ralf.hiptmair@sam.math.ethz.ch}).}
\and Ralf Hiptmair\footnotemark[2]}

\usepackage{amsopn}

\ifpdf
\hypersetup{
  pdftitle={First-Kind Boundary Integral Equations for the Dirac Operator in 3D Lipschitz domains},
  pdfauthor={E. Schulz and R. Hiptmair}
}
\fi

\usepackage{cite}
\usepackage{amssymb}
\usepackage{amsfonts}
\usepackage{amsmath}
\usepackage{algorithmic}
\usepackage{cancel}
\usepackage{graphicx}
\usepackage{textcomp}
\def\BibTeX{{\rm B\kern-.05em{\sc i\kern-.025em b}\kern-.08em
		T\kern-.1667em\lower.7ex\hbox{E}\kern-.125emX}}
\usepackage{bm}
\usepackage{dutchcal}
\usepackage{mathrsfs}
\usepackage{hyperref}
\usepackage{cleveref}
\usepackage{commath}
\usepackage{empheq}
\usepackage{enumerate}
\usepackage[all]{xy}
\usepackage{pgf,tikz,pgfplots}
\pgfplotsset{compat=1.15}
\usetikzlibrary{arrows}
\usepackage[most]{tcolorbox}
\newcommand{\highlight}[1]{%
	\colorbox{black!10!white}{$\displaystyle#1$}}
\newtcolorbox{greyFrame}{
	colframe=black!20!white,
}

\definecolor{alizarin}{rgb}{0.82, 0.1, 0.26}
\definecolor{ao}{rgb}{0.0, 0.5, 0.0}

\newcommand\bl[1]{\mathbf{#1}}
\newcommand\bra[1]{\left(#1\right)}
\newcommand{\id}{\mathrm{Id}}
\newcommand\starequal{\stackrel{\mathclap{\normalfont\mbox{($*$)}}}{=}}
\newcommand\ip[2]{\langle #1, #2 \rangle}

\newtheorem*{warning}{Warning}

\makeatletter
\DeclareFontFamily{OMX}{MnSymbolE}{}
\DeclareSymbolFont{MnLargeSymbols}{OMX}{MnSymbolE}{m}{n}
\SetSymbolFont{MnLargeSymbols}{bold}{OMX}{MnSymbolE}{b}{n}
\DeclareFontShape{OMX}{MnSymbolE}{m}{n}{
	<-6>  MnSymbolE5
	<6-7>  MnSymbolE6
	<7-8>  MnSymbolE7
	<8-9>  MnSymbolE8
	<9-10> MnSymbolE9
	<10-12> MnSymbolE10
	<12->   MnSymbolE12
}{}
\DeclareFontShape{OMX}{MnSymbolE}{b}{n}{
	<-6>  MnSymbolE-Bold5
	<6-7>  MnSymbolE-Bold6
	<7-8>  MnSymbolE-Bold7
	<8-9>  MnSymbolE-Bold8
	<9-10> MnSymbolE-Bold9
	<10-12> MnSymbolE-Bold10
	<12->   MnSymbolE-Bold12
}{}
\let\llangle\@undefined
\let\rrangle\@undefined
\DeclareMathDelimiter{\llangle}{\mathopen}%
{MnLargeSymbols}{'164}{MnLargeSymbols}{'164}
\DeclareMathDelimiter{\rrangle}{\mathclose}%
{MnLargeSymbols}{'171}{MnLargeSymbols}{'171}
\makeatother

\def\Xint#1{\mathchoice
	{\XXint\displaystyle\textstyle{#1}}%
	{\XXint\textstyle\scriptstyle{#1}}%
	{\XXint\scriptstyle\scriptscriptstyle{#1}}%
	{\XXint\scriptscriptstyle\scriptscriptstyle{#1}}%
	\!\int}
\def\XXint#1#2#3{{\setbox0=\hbox{$#1{#2#3}{\int}$ }
		\vcenter{\hbox{$#2#3$ }}\kern-.6\wd0}}

\def\dashint{\Xint-}

\newcommand{\surfgrad}{\nabla_{\Gamma}}
\newcommand{\surfcurl}{\text{curl}_{\Gamma}}
\newcommand{\surfcurlbl}{\bl{curl}_{\Gamma}}

\newcommand{\LR}{\mathcal{L}_{\mathsf{R}}}
\newcommand{\LT}{\mathcal{L}_{\mathsf{T}}}
\newcommand{\extd}{\bl{d}}
\newcommand{\extdel}{\bm{\delta}}
\newcommand{\traceT}{\gamma_{\mathsf{T}}}
\newcommand{\traceR}{\gamma_{\mathsf{R}}}

\newcommand{\Dirac}{\mathsf{D}}
\newcommand{\vecU}{\vec{\bl{U}}}
\newcommand{\vecV}{\vec{\bl{V}}}

\newcommand{\vecb}{\vec{\bl{b}}}

\newcommand{\veca}{\vec{\bl{a}}}
\newcommand{\vecd}{\vec{\bl{d}}}

\newcommand{\vecc}{\vec{\bl{c}}}
\newcommand{\biT}{\mathcal{B}_{\mathsf{T}}}

\newcommand{\curl}{\mathbf{curl}}

\newcommand{\Hhalf}{H^{\frac{1}{2}}(\Gamma)}
\newcommand{\Hhalfcurl}{\bl{H}^{-\frac{1}{2}}(\text{curl}_{\Gamma},\Gamma)}

\newcommand{\HT}{\mathcal{H}_{\mathsf{T}}}

\newcommand{\HR}{\mathcal{H}_{\mathsf{R}}}

\newcommand{\Hcurl}{\bl{H}(\bl{curl},\Omega)}

\newcommand{\Hd}{\mathbf{H}(\bl{d},\Omega)}
\newcommand{\Hdnot}{\mathbf{H}_0(\bl{d},\Omega)}
\newcommand{\Hdelta}{\mathbf{H}(\bm{\delta},\Omega)}

\newcommand{\HD}{\bl{H}(\mathsf{D},\Omega)}

\newcommand{\intOmega}[1]{\int_{\Omega}#1\dif\bl{x}}

\begin{document}

\maketitle

% REQUIRED
\begin{abstract}
We develop novel first-kind boundary integral equations for Euclidean Dirac operators in 3D Lipschitz domains. They comprise square-integrable potentials and involve only weakly singular kernels. Generalized G{\aa}rding inequalities are derived and we establish that the obtained boundary integral operators are Fredholm of index zero. Their finite dimensional nullspaces are characterized and we show that their dimensions are equal to the number of topological invariants of the domain's boundary, in other words, to the sum of its Betti numbers. This is explained by the fundamental discovery that the associated bilinear forms agree with those induced by the 2D Dirac operators for surface de Rham Hilbert complexes whose underlying inner-products are the non-local inner products defined through the classical single-layer boundary integral operators for the Laplacian. Decay conditions for well-posedness in natural energy spaces of the Dirac system in unbounded exterior domains are also presented. 
\end{abstract}

% REQUIRED

% REQUIRED

\section{Introduction}
We develop first-kind boundary integral equations for the Hodge-Dirac operator in 3-dimensional Euclidean space
\begin{equation}\label{eq: Dirac decomposition}
    \mathsf{D}:=\bl{d}+\bm{\delta}:\bl{H}(\bl{d},\Omega^{\mp})\cap \bl{H}(\bm{\delta},\Omega^{\mp})\rightarrow L^2(\Omega^{\mp})^8,
\end{equation}
involving the exterior derivative and codifferential
\begin{align}
\bl{d}&:=\begin{pmatrix}
0                    & \bl{0}^\top                  & \bl{0}^\top                 & 0       \\
\highlight{\nabla} & \bl{0}_{3\times 3}                           & \bl{0}_{3\times 3}       &\bl{0}       \\
\bl{0}                   &  \highlight{\mathbf{curl}}& \bl{0}_{3\times 3}                       &\bl{0}    \\
0                    & \bl{0}^\top                            & \highlight{\text{div}} & 0
\end{pmatrix}
&\text{and} & & \bm{\delta}&:=
\begin{pmatrix}
0                    & \highlight{-\,\text{div}}                 & \bl{0}^\top                 & 0       \\
\bl{0} & \bl{0}_{3\times 3}                           & \highlight{\mathbf{curl}}      &\bl{0}       \\
\bl{0}                   &  \bl{0}_{3\times 3}  & \bl{0}_{3\times 3}                       &\highlight{-\nabla}    \\
0                    & \bl{0}^\top                            & \bl{0}^\top  & 0
\end{pmatrix}.
\end{align}
We are concerned with the partial differential equations $\mathsf{D}\vec{\bl{U}}=\vec{\bl{F}}$, which  in components $\vecU=\bra{U_0,\bl{U}_1,\bl{U}_2,U_3}^\top$ and $\vec{\bl{F}}=\bra{F_0,\bl{F}_1,\bl{F}_2,F_3}^\top$ read
\begin{equation}\label{eq: Dirac PDE explicit}
\begin{aligned}
-\,\text{div}\,\bl{U}_1 &= F_0,\\
\nabla\,U_0 +\mathbf{curl}\,\bl{U}_2 &= \bl{F}_1, \\
-\,\nabla U_3 +  \mathbf{curl}\,\bl{U}_1 & = \bl{F}_2,\\
\text{div}\,\bl{U}_2 &= F_3. 
\end{aligned}
\end{equation}

We will consider both interior and exterior boundary value problems, and assume that \eqref{eq: Dirac PDE explicit} is either posed on a bounded domain $\Omega^-$ having a Lipschitz boundary $\Gamma:=\partial\Omega^-$, or on the unbounded complement $\Omega^+:=\mathbb{R}^3\backslash\overline{\Omega^{-}}$. In the latter case, suitable decay conditions at infinity will be needed. Throughout, $\Omega\in\{\Omega^-,\Omega^+\}$.

\subsection{Related work}
	Current work discussing Dirac operators from the point of view of Hodge theory offers solutions to boundary value problems for \eqref{eq: Dirac PDE explicit} and related eigenvalue problems based on domain variational formulations \cite{leopardi2016abstract,christiansen2018eigenmode}.
	
	The operator matrix in \eqref{eq: Dirac decomposition} appears under a change of variables in the works of M.~Taskinen, S.~V{\"a}nsk{\"a} and  P.~Yl\"{a}-Oijala \cite{taskinen2006current,taskinen2007current, taskinen2007Picard} as R.~Picard's extended Maxwell operator. It was originally assembled by R.~Picard by combining the first-order Maxwell operator with the principal part of the equations of linear acoustics \cite{picard1984low,picard1985structural,leis2013initial}. In \cite{taskinen2006current,taskinen2007current, taskinen2007Picard}, Helmholtz-like boundary value problems for Picard's operator are studied with a focus on \emph{second-kind} boundary integral equations.

	Eigenvalue problems related to acoustic and electromagnetic scattering, that is transmission problems for the so-called perturbed Dirac operator, have also guided the study of second-kind boundary integral equations in the literature of harmonic and hypercomplex analysis. Important contributions were made in that direction by E.~Marmolejo-Olea, I.~Mitrea, M.~Mitrea, Q.~Shi \cite{marmolejo2012transmission}, A.~Axelsson, A.~Ros\'en and J.~Helsing  \cite{axelsson2006transmission,helsing2019dirac,rosengeometric}. There, the Dirac operator enters larger systems of equations that encompass or correspond to Maxwell's equations \cite{marmolejo2012transmission,helsing2019dirac}. An extensive body of work, created by these authors together with R. Grognard and J.~Hogan \cite{axelsson2001harmonic}, S.~Keith \cite{axelsson2006quadratic}, A.~McIntosh and S.~Monniaux \cite{mcintosh1999clifford,mcintosh2016hodge}, is devoted to the harmonic analysis of Dirac operators in $L_p$ spaces \cite{axelsson2012hodge,marmolejo2004harmonic}.

\subsection{Our contributions}
In this work, we derive novel \emph{first-kind} boundary integral equations for the Dirac equation $\mathsf{D}\vecU=0$ with suitable boundary and decay conditions. Two boundary integral operators are obtained and shown to satisfy generalized G{\aa}rding inequalities, making them Fredholm of index $0$. Their finite dimensional nullspaces are characterized in \Cref{sec: Kernels}, where we show that their dimension equals the number of topological invariants of the boundary---counted as the sum of its Betti numbers. Indeed, the integral representations of their associated bilinear forms turn out to be related to the variational formulations of the surface Dirac operators introduced in \Cref{surface dirac operators}. Recognizing these surface operators will simultaneously reveal how the boundary integral operators introduced in \Cref{Sec: boundary integral operators}, which are related to two different sets of boundary conditions, arise as ``rotated" versions of one another. The exterior representation formula of \Cref{lem: rep formula exterior domain} and the condition at infinity identified in \eqref{eq: decay condition lemma} eventually lead, together with the coercivity results of \Cref{sec: T-coercivity}, to well-posedness of Euclidean Dirac exterior boundary value problems in natural energy spaces in the complement of the finite dimensional nullspaces.

The new integral formulas display desirable properties: the surface potentials are
square-integrable and the kernels of the bilinear forms associated with the boundary
integral operators are merely weakly singular, i.e. they are bounded by
$\abs{x-y}^{-\alpha}$, $\alpha<2$, cf. \cite[Sec. 2.4]{kress1999linear}. Nevertheless, we
want to emphasize that the main result is the discovery that they relate to the
Hodge--Dirac operators of surface de Rham Hilbert complexes equipped with the non-local
inner products defined as the bilinear forms associated with the classical single-layer
potential for the Laplacian. As a consequence, we already know a lot about these
first-kind boundary integral operators for the Dirac operator. Moreover, this relationship
suggests that they are related to the first-kind boundary integral operators for the
Hodge--Laplacian.

For the sake of readability, we adopt the framework of classical vector analysis rather than exterior calculus. It is in this framework that the structural relationship between the following development and the standard theory for second-order elliptic operators seemed most explicit.

In summary, our main contributions are:
\begin{itemize}
	\item[$\triangleright$] We derive representation formulas for the Dirac equation posed on domains having a Lipschitz boundary by following the approach pioneered by M. Costabel \cite{costabel1988boundary}. The novelty here is to follow and extend the elegant strategy used in \cite{claeys2017first}---there used to find a representation formula for Hodge--Laplace and Helmholtz operators---that leads to potentials having simple explicit expressions. By adapting the arguments in the now classical monographs by W. McLean \cite[Chap. 7]{mclean2000strongly} and A. Sauter and C. Schwab \cite[Chap. 3]{sauter2010boundary}, we also establish an exterior representation formula. We will observe that the development of this theory is possible due to the strong structural similarity between integration by parts for the first-order Dirac operator and Green's second formula for second-order elliptic operators.
	\item[$\triangleright$] A sneak peek at the potentials presented in \eqref{eq: wedge potential} and \eqref{eq: vee potential} will already convince the reader that the approach we have adopted leads to simple formulas for the \emph{square-integrable} potentials involved in the representation formula. Some terms are recognizable from \cite{claeys2017first,claeys2018first}, while others occur in well-known theory for elliptic second-order operators. The simplicity that comes with the calculation procedure provided by \Cref{lem: int rep of boundary potential} allows for a straightforward analysis of their mapping and jump properties.
	\item[$\triangleright$] Given the previous items, it is not surprising that decay conditions at infinity for exterior boundary value problems posed on the unbounded domain $\Omega^+$ can be easily established by adapting the approach for second-order elliptic operators presented in \cite[Chap. 7]{mclean2000strongly}.
	\item[$\triangleright$] The crux of our calculations are the formulas \eqref{eq: wedge wedge bilinear form} and \eqref{eq: vee vee bilinear form} for the bilinear forms associated with the obtained weakly-singular first-kind boundary integral operators. We provide generalized G{\aa}rding inequalities for the two operators and characterize their null-spaces.
	\item[$\triangleright$] Our main discovery is presented in \Cref{surface dirac operators}, where we expose the relationship between these boundary integral operators and surface Dirac operators in an Hilbert complex framework.
\end{itemize} 

\section{Function spaces and traces}\label{Function spaces and traces}
As usual, $L^2(\Omega)$ and $\mathbf{L}^2(\Omega)$ denote the Hilbert spaces of complex square-integrable scalar and vector-valued functions defined over $\Omega$. We denote their inner products using round brackets, e.g. $(\cdot,\cdot)$. The spaces $H^1(\Omega)$ and $\mathbf{H}^1(\Omega)$ refer to the corresponding Sobolev spaces. The notation $C^\infty\bra{\Omega}$ is used for smooth functions. The subscript in $C^{\infty}_0\bra{\Omega}$ further specifies that these smooth functions have compact support in $\Omega$. $C^{\infty}(\overline{\Omega})$ is defined as
the space of uniformly continuous functions over Ω that have uniformly continuous
derivatives of all order. A subscript is used to identify spaces of locally integrable functions/vector  fields, e.g. $U\in L^2_{\text{loc}}(\Omega)$ if and only if $\varphi U$ is square-integrable for all $\varphi\in C^{\infty}_0(\mathbb{R}^3)$. We denote with an asterisk the spaces of functions with zero mean, e.g. $H^1_*(\Omega)$.

In general, given an operator $\mathsf{L}$ acting on square-integrable fields in the sense of distributions, we equip
\begin{equation}
\bl{H}\bra{\mathsf{L},\Omega}:=\{ \bl{U}\in\bra{\bl{L}^2\bra{\Omega}}^{\bullet}\Big\vert\,\mathsf{L}\bl{U}\in\bra{\bl{L}^2\bra{\Omega}}^{\dag}\}
\end{equation}
with the natural graph norm, where $\bullet=8$ or $3$ and $\dag=8$, $3$ or $1$. Important specimens are
\begin{align}
\mathbf{H}(\text{div}, \Omega) &:=\left\{\mathbf{U}\in \bra{L^2(\Omega)}^3 \,\vert\, \text{div}\,\mathbf{U}\in L^2(\Omega)\right\},\\
\mathbf{H}(\mathbf{curl}, \Omega) &:=\left\{\mathbf{U}\in \bra{L^2(\Omega)}^3 \,\vert\, \mathbf{curl}\,\mathbf{U}\in \bra{L^2(\Omega)}^3\right\}.
\end{align}
Of course, in all of the above definitions, $\Omega$ can be replaced by $\mathbb{R}^3$, or any other domain. We understand restrictions in the sense of distributions when working with domains having disconnected components. For example, in line with the above notation we mean in particular
\begin{equation}
\bl{H}\bra{\mathsf{D},\mathbb{R}^3\backslash\Gamma} := \bl{H}\bra{\mathsf{D},\Omega}\times \bl{H}\bra{\mathsf{D},\mathbb{R}^3\backslash\Omega} \subset \bra{L^2(\mathbb{R}^3)}^8.
\end{equation}

We use a prime superscript to denote dual spaces, for instance $C^{\infty}_0\bra{\Omega}'$ is the space of distributions in $\Omega$. Angular brackets indicate duality pairings, e.g. $\langle\cdot,\cdot\rangle_{\Omega}$ or $\llangle\cdot,\cdot\rrangle_{\Gamma}$. The former will be used for domain-based quantities in $\Omega$, while the latter will pair spaces on $\Gamma$.

Trace-related theory for Lipschitz domains can be found in \cite{buffa2001traces_a,buffa2001traces_b,buffa2002traces} and \cite{mclean2000strongly,girault2012finite}, where it is established that the traces
\begin{subequations}\label{ed: traces smooth}
	\begin{align}
	\gamma W&:=\bl{W}\big\vert_{\Gamma}, &&\forall\, W\in C^{\infty}(\overline{\Omega}),\\
	\gamma_n\bl{W}&:=\gamma\bl{W}\cdot\bl{n}, &&\forall\,\bl{W}\in \mathbf{C}^{\infty}(\overline{\Omega}),\label{eq: def neumann trace}\\
	\gamma_{\tau}\bl{W}&:=\gamma\bl{W}\times\bl{n}, &&\forall\,\bl{W}\in \mathbf{C}^{\infty}(\overline{\Omega}),\label{eq: def tau trace}\\
	\gamma_{t}\bl{W}&:=\mathbf{n}\times\bra{\gamma_{\tau}\bl{W}}, &&\forall\,\bl{W}\in \mathbf{C}^{\infty}(\overline{\Omega}),\label{eq: def t trace}
	\end{align}
\end{subequations}
extend to continuous and surjective linear operators
\begin{subequations}\label{eq:extension interior traces}
\begin{align}
	\gamma &: H^1\bra{\Omega}\rightarrow H^{1/2}\bra{\Gamma}, &&\text{\cite[Thm. 4.2.1]{Hsiao2008}}\\
	\gamma_n &: \mathbf{H}(\text{div}, \Omega)\rightarrow H^{-1/2}\bra{\Gamma},&&\text{\cite[Thm. 2.5, Cor. 2.8]{girault2012finite}}\\
	\gamma_{\tau}&:\mathbf{H}\bra{\mathbf{curl},\Omega} \rightarrow \mathbf{H}^{-1/2}(\text{div}_\Gamma,\Gamma),&&\text{\cite[Thm. 4.1]{buffa2002traces}}\\
	\gamma_{t}&:\mathbf{H}\bra{\mathbf{curl},\Omega} \rightarrow \mathbf{H}^{-1/2}(\text{curl}_\Gamma,\Gamma), &&\text{\cite[Thm. 4.1]{buffa2002traces}}
\end{align}
\end{subequations}
with nullspaces 
\begin{align}
    H^1_0(\Omega)&:=\overline{C_0^{\infty}(\Omega)}^{H^1(\Omega)}=\ker \gamma, &&\text{\cite[Thm 3.40]{mclean2000strongly}}\\
    \mathbf{H}_0(\text{div}, \Omega)&:=\overline{C_0^{\infty}(\Omega)^3}^{\mathbf{H}\bra{\text{div},\Omega}}=\ker \gamma_n, && \text{\cite[Thm. 3.25]{monk2003finite}}\\
    \mathbf{H}_0(\mathbf{curl}, \Omega)&:=\overline{C_0^{\infty}(\Omega)^3}^{\mathbf{H}\bra{\mathbf{curl},\Omega}}=\ker \gamma_\tau=\ker\gamma_t. && \text{\cite[Thm. 3.33]{monk2003finite}}
\end{align}
Here, $\mathbf{n}\in\mathbf{L}^{\infty}(\Gamma)$ is the essentially bounded unit normal vector field on $\Gamma$ directed toward the exterior of $\Omega^-$. Detailed definitions can be found in \cite{buffa2001traces_a,buffa2001traces_b,buffa2002traces} together with a study of the involved surface differential operators. Short practical summaries are also provided in \cite{buffa2003galerkin, claeys2017first, schulz2020coupled,kirchhart2020div}.

Similarly as for the Hodge--Laplace operator \cite{claeys2017first, claeys2018first, schulz2020coupled,schulz2020spurious}, a theory of boundary value problems for the Hodge--Dirac problem in three dimensions entails partitioning our collection of traces into two ``dual" pairs. Accordingly, we assemble the traces into 
	\begin{align}\label{eq: little wedge and vee mappings}
	&{\color{blue}\gamma_{\mathsf{T}}\bra{\vec{\bl{U}}} :=
		\begin{pmatrix}
		\gamma\bra{U_0} \\ \,\gamma_{t}\bra{\bl{U}_1} \\ \gamma_n\bra{\bl{U}_2}
		\end{pmatrix}}
	&\text{and}&
	&{\color{red}\gamma_{\mathsf{R}}\bra{\vec{\bl{U}}} :=
		\begin{pmatrix}
		\gamma_n\bra{\bl{U}_1} \\\gamma_{\tau}\bra{\bl{U}2} \\ \,\gamma\bra{U_3}
		\end{pmatrix}}.
	\end{align}
	
\begin{warning}
	We want to highlight that in spite of the notation, $\gamma_{\mathsf{T}}$ and $\gamma_{\mathsf{R}}$ are \textbf{not} defined as in \cite{claeys2017first}, \cite{claeys2018first} and related work.
\end{warning}
The trace spaces
\begin{subequations}
	\begin{align}
	{\color{blue}\mathcal{H}_{\mathsf{T}}}&{\color{blue}:= H^{1/2}\bra{\Gamma}\times\mathbf{H}^{-1/2}(\text{curl}_\Gamma,\Gamma) \times H^{-1/2}\bra{\Gamma}},\\
	{\color{red}\mathcal{H}_{\mathsf{R}}}&{\color{red}:= H^{-1/2}\bra{\Gamma}\times\mathbf{H}^{-1/2}(\text{div}_\Gamma,\Gamma) \times H^{1/2}\bra{\Gamma}},
	\end{align}
\end{subequations}
are dual to each other with respect to the $L^2\bra{\Gamma}$ duality pairing (c.f. \cite[Lem. 5.6]{buffa2002traces}). In this sense, we can identify
\begin{align}\label{eq: duality of trace spaces}
\HT'=\HR &&\text{and} &&\HR'=\HT.
\end{align}

Naturally, the traces can also be taken from the exterior domain. The extensions
\eqref{eq:extension interior traces} will be tagged with a minus subscript (only when
required to avoid confusion), e.g. $\gamma^-$, to distinguish them from the extensions
obtained from \eqref{ed: traces smooth} by replacing $\Omega$ with
$\Omega^+:=\mathbb{R}^3\backslash\overline{\Omega}$, which we will label with a plus
superscript, e.g. $\gamma^+$.
\begin{lemma}[{See \cite[Lem. 6.4]{claeys2017first}}]\label{lem: traces are continuous and surjective}
	The linear mappings
	\begin{align}
	&\gamma^{\pm}_{\mathsf{T}}:\bl{H}_{\emph{\text{loc}}}\bra{\mathsf{D},\Omega^{\pm}}\rightarrow \mathcal{H}_{\mathsf{T}},&
	&\gamma^{\pm}_{\mathsf{R}}:\bl{H}_{\emph{\text{loc}}}\bra{\mathsf{D},\Omega^{\pm}}\rightarrow \mathcal{H}_{\mathsf{R}},
	\end{align}
	defined  by \eqref{eq: little wedge and vee mappings} are continuous and surjective. There exist continuous lifting maps $\mathcal{E}_{\mathsf{T}}:\mathcal{H}_{\mathsf{T}}\rightarrow\bl{H}_{\emph{\text{loc}}}(\mathsf{D},\mathbb{R}^3\backslash\Gamma)$ and $\mathcal{E}_{\mathsf{R}}:\mathcal{H}_{\mathsf{R}}\rightarrow\bl{H}_{\emph{\text{loc}}}(\mathsf{D},\mathbb{R}^3\backslash\Gamma)$ such that $\gamma_\mathsf{T}\circ\mathcal{E}_{\mathsf{T}} = \id$ and $\gamma_\mathsf{R}\circ\mathcal{E}_{\mathsf{R}} = \id$.
\end{lemma}

\begin{lemma}[{See \cite[Lem. 6.4]{claeys2017first}}]\label{lem: properties of surface div and curl}
	The surface divergence extends to a continuous surjection $\text{\emph{div}}_{\Gamma}:\mathbf{H}^{-1/2}(\text{\emph{div}}_\Gamma,\Gamma)\rightarrow H^{-1/2}_*(\Gamma)$, while $\mathbf{curl}_{\Gamma}:H^{1/2}_*\rightarrow\mathbf{H}^{-1/2}(\text{\emph{div}}_\Gamma,\Gamma)$ is a bounded injection with closed range such that $\mathbf{curl}_{\Gamma}\xi=\nabla_{\Gamma}\xi\times\mathbf{n}$ for all $\xi\in H^{1/2}(\Gamma)$. These operators satisfy $\text{\emph{div}}_{\Gamma}\circ\mathbf{curl}_{\Gamma} = 0$.
\end{lemma}

\begin{lemma}\label{eq: IBP weakform}
	For all $\vecU\in\bl{H}(\bl{d},\Omega^\mp)$ and $\vecV\in\bl{H}(\bm{\delta},\Omega^\mp)$,
	\begin{equation}
	\int_{\Omega^\mp} \bl{d}\vecU\cdot\vecV\dif\bl{x}= \int_{\Omega^\mp}\vecU\cdot\bm{\delta}\vecV\dif\bl{x} \pm\llangle\gamma_{\mathsf{T}}\vecU,\gamma_{\mathsf{R}}\vecV\rrangle_{\Gamma}.
	\end{equation}
\end{lemma}
\begin{proof}
	We integrate by parts using Green's identities to obtain
	\begin{align*}
	\int_{\Omega^\mp}\bl{d}\bl{U}\cdot\bl{V}\dif\bl{x}
	&=\int_{\Omega^\mp}\nabla U_0\cdot\bl{V}_1\dif\bl{x} + \int_{\Omega^\mp}\bl{curl}\,\bl{U}_1\cdot\bl{V}_2\dif\bl{x} +\int_{\Omega^\mp}\bra{\text{div}\,\bl{U}_2}V_3\dif\bl{x}\\
	&=-\int_{\Omega^\mp}\bl{U}_0\,\bra{\text{div}\,\bl{V}_1}\dif\bl{x}
	+\int_{\Omega^\mp}\bl{U}_1\cdot\bl{curl}\,\bl{V}_2\dif\bl{x}
	-\int_{\Omega^\mp}\bl{U}_2\cdot\nabla V_3\dif\bl{x}\\
	&\qquad+\ip{\gamma U_0}{\gamma_n\bl{V}_1}_{\Gamma}
	+\ip{\gamma_t\bl{U}_1}{\gamma_{\tau}\bl{V}_2}_{\Gamma}
	+\ip{\gamma_n\bl{U}_2}{\gamma V_3}_{\Gamma}\\
	&=\int_{\Omega^\mp}\bl{U}\cdot\bm{\delta}\bl{V}\dif\bl{x}
	+\llangle\gamma_{\mathsf{T}}\bl{U},\gamma_{\mathsf{R}}\bl{V}\rrangle_{\Gamma}.
	\end{align*}
\end{proof}

\begin{corollary}[Green's formula for Dirac operator]
  \label{cor. second green identity}
	For all $\vec{\bl{U}},\vec{\bl{V}}\in\bl{H}(\mathsf{D},\Omega^\mp)$, we have
	\begin{equation}\label{eq: IBP Dirac}
	\int_{\Omega^\mp}\mathsf{D}\vec{\bl{U}}\cdot\vec{\bl{V}}\dif\bl{x} = \int_{\Omega^\mp}\vec{\bl{U}}\cdot\mathsf{D}\vec{\bl{V}}\dif\bl{x}\pm \llangle\gamma_{\mathsf{T}}\vec{\bl{U}},\gamma_{\mathsf{R}}\vec{\bl{V}}\rrangle_{\Gamma} \mp
	\llangle\gamma_{\mathsf{T}}\vec{\bl{V}},\gamma_{\mathsf{R}}\vec{\bl{U}}\rrangle_{\Gamma}.
	\end{equation}
\end{corollary}

\begin{remark}
	It is remarkable that despite the fact that $\mathsf{D}$ is a first-order operator, \cref{eq: IBP Dirac} nevertheless resembles Green's classical second formula for the Laplacian. This induces profound structural similarities between the representation formula, potentials and boundary integral equations for the Dirac operator established in the next sections and the already well-known theory for second-order elliptic operators.  As emphasized in \cite{schulz2020spurious}, a formula such as \cref{eq: IBP Dirac} paves the way for harnessing powerful established techniques.
\end{remark}

We will indicate with curly brackets the average $\{\gamma_{\bullet}\}:= \frac{1}{2}(\gamma_{\bullet}^+ + \gamma_{\bullet}^-)$ of a trace and with square brackets its jump $[\gamma_{\bullet}]:=\gamma_{\bullet}^- - \gamma_{\bullet}^+$ over the interface $\Gamma$.
\begin{warning}
	Notice the sign in the jump $[\gamma]=\gamma^- - \gamma^+$, which is often taken to be the opposite in the literature!
\end{warning}

\section{Boundary value problems}\label{sec: Hodge--Dirac operators in 3D Euclidean space}
In light of \Cref{lem: traces are continuous and surjective} and the duality in \eqref{eq: duality of trace spaces}, the integration by parts formula \eqref{eq: IBP Dirac} points towards two types of boundary conditions. Consider the boundary value problems of finding $\vecU\in\HD$ satisfying
\begin{align}\label{pb: T}
\color{blue}
\begin{cases}
    \Dirac\vecU &=\vec{\bl{0}},  \qquad\text{in }\Omega,\\
    \traceT\vecU &=\vec{\bl{b}}, \qquad\text{on }\Gamma,
\end{cases} &&\vecb\in\HT,\tag{\textsf{T}}
\end{align}
or
\begin{align}\label{pb: R}
\color{red}
\begin{cases}
    \Dirac\vecU &=\vec{\bl{0}},  \qquad\text{in }\Omega,\\
    \traceR\vecU &=\veca, \qquad\text{on }\Gamma,
\end{cases}&&\veca\in\HR.\tag{\textsf{R}}
\end{align}
For $\Omega=\Omega^+$, also impose the decay condition that $\vecU(\bl{x})\rightarrow 0$ uniformly as $\bl{x}\rightarrow\infty$, cf. \Cref{lem: decay condition M=0}. In the following sections, development related to problem \eqref{pb: T} will be colored in{ \color{blue}blue}, while {\color{red}red} will be used for \eqref{pb: R}.

When $\Omega$ is bounded, the self-adjoint Dirac operator behind \eqref{pb: R} is 
\begin{equation}
\color{red}\mathsf{D}^{\Omega}_{\mathsf{R}}=\bl{d}+\bl{d}^*,    
\end{equation} 
where $\bl{d}:L^2(\Omega)^8\rightarrow L^2(\Omega)^8$ is the closed densely defined Fredholm-nilpotent linear operator associated with the $L^2$ de Rham cochain complex \cite{arnold2018finite,leopardi2016abstract}
\begin{equation}\label{domain de Rham complex d 3D}
\xymatrix{
	\color{red}H^1\bra{\Omega}\ar@[red][r]^-{\color{red}\nabla}&\color{red} \mathbf{H}\bra{\mathbf{curl},\Omega}\ar@[red][r]^-{\color{red}\mathbf{curl}}& \color{red}\mathbf{H}\bra{\text{div},\Omega}  \ar@[red][r]^-{\color{red}\text{div}} & \color{red}L^2\bra{\Omega},
	}
\end{equation}
cf. \cite[Chap. 3-4]{arnold2018finite}, \cite[Sec. 2]{leopardi2016abstract}. The Hilbert space adjoint $\bl{d}^*$ is the nilpotent operator associated with the dual chain complex \cite[Sec. 4.3, Thm. 6.5]{arnold2018finite} 
\begin{equation}\label{domain d adjoint complex}
\xymatrix{
	\color{red}L_*^2\bra{\Omega}
	&\ar@[red][l]^-{\color{red}\,\,-\text{div}} \color{red}\mathbf{H}_0\bra{\text{div},\Omega}
	&\ar@[red][l]^-{\color{red}\mathbf{curl}} \color{red}\mathbf{H}_0\bra{\mathbf{curl},\Omega} 
	&\ar@[red][l]^-{\color{red}-\nabla} \color{red}\mathbf{H}_0^1\bra{\Omega}.
}
\end{equation}
The mapping properties of $\mathsf{D}_{\mathsf{R}}$ and its domain are detailed in \Cref{fig: natural boundary conditions}.

Similarly, the self-adjoint operator
\begin{equation}
    \color{blue}\mathsf{D}^{\Omega}_{\mathsf{T}}:= \bm{\delta} + \bm{\delta}^*
\end{equation}
behind \eqref{pb: T} arises from the dual perspective, where we view the codifferential operator $\bm{\delta}:L^2(\Omega)^8\rightarrow L^2(\Omega)^8$ as the nilpotent operator associated with the Hilbert chain complex
\begin{equation}\label{domain de Rham complex delta 3D}
\xymatrix{
	\color{blue}L^2\bra{\Omega}
	&\ar@[blue][l]^-{\color{blue}\,\,-\text{div}} \color{blue}\mathbf{H}\bra{\text{div},\Omega}
	&\ar@[blue][l]^-{\color{blue}\mathbf{curl}} \color{blue}\mathbf{H}\bra{\mathbf{curl},\Omega} 
	&\ar@[blue][l]^-{\color{blue}-\nabla} \color{blue}\mathbf{H}^1\bra{\Omega}.
}
\end{equation}
The adjoint $\bm{\delta}^*$ is spawned by the chain complex
\begin{equation}\label{domain delta adjoint complex}
\xymatrix{
	\color{blue}H_0^1\bra{\Omega}\ar@[blue][r]^-{\color{blue}\nabla}&\color{blue} \mathbf{H}_0\bra{\mathbf{curl},\Omega}\ar@[blue][r]^-{\color{blue}\mathbf{curl}}& \color{blue}\mathbf{H}_0\bra{\text{div},\Omega}  \ar@[blue][r]^-{\color{blue}\text{div}} & \color{blue}L_*^2\bra{\Omega}.
	}
\end{equation}
See \Cref{fig: essential boundary conditions} for the explicit mapping properties of $\mathsf{D}^{\Omega}_{\mathsf{T}}$ and its domain of definition.

\begin{figure}
	\begin{center}
		\makebox[\textwidth]{\includegraphics[width=0.6\linewidth]{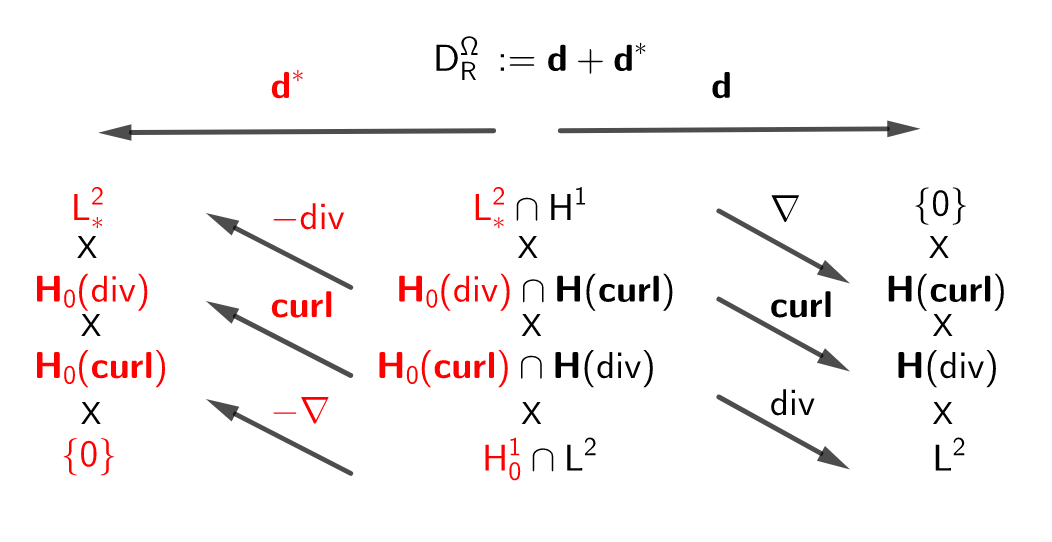}}
	\end{center}
	\caption{This diagram shows the mapping properties of the exterior derivatives and their Hilbert space adjoints corresponding to the functional analytic setting of \cite{leopardi2016abstract} for problem \eqref{pb: R} in $\Omega^-$. In the figure, the operators on the left-hand side are to be understood as the adjoint operators ${\color{red}-\text{div}}=\nabla^*$, ${\color{red}\mathbf{curl}}=\mathbf{curl}^*$ and ${\color{red}-\nabla}=\text{div}^*$.}\label{fig: natural boundary conditions}
\end{figure} 

\begin{figure}
	\begin{center}
		\makebox[\textwidth]{\includegraphics[width=0.6\linewidth]{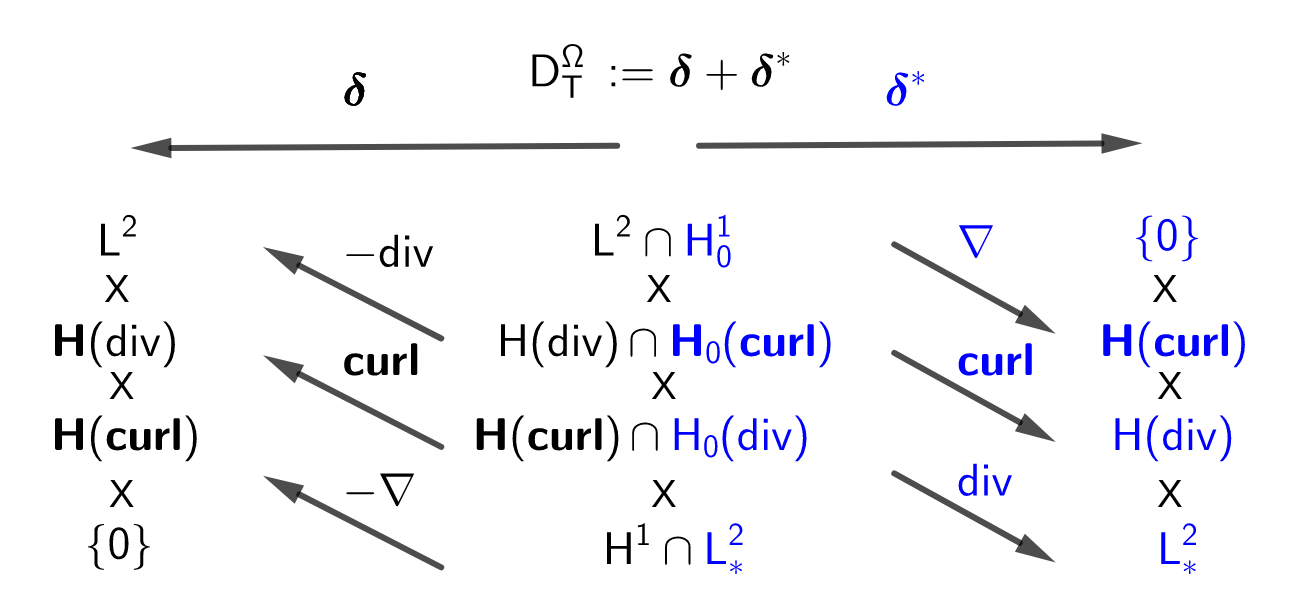}}
	\end{center}
	\caption{This diagram shows the mapping properties of the codifferentials and their Hilbert space adjoints corresponding to the functional analytic setting of \cite{leopardi2016abstract} for problem \eqref{pb: T} in $\Omega^-$. In the figure, the operators on the left-hand side are to be understood as the adjoint operators ${\color{blue}\nabla}=-\text{div}^*$, ${\color{blue}\mathbf{curl}}=\mathbf{curl}^*$ and ${\color{blue}\text{div}}=-\nabla$.}\label{fig: essential boundary conditions}
\end{figure}

So unlike second-order operators, the Hodge--Dirac operator admits two distinct fundamental symmetric bilinear forms
\begin{subequations}
\begin{align}
\color{blue}\mathcal{A}_{\extdel}(\vecU,\vecV)&=\color{blue}\intOmega{\extdel\vecU\cdot\vecV+\vecU\cdot\extdel\vecV}, & \vecU,\vecV\in\Hdelta,\\
\color{red}\mathcal{A}_{\extd}(\vecU,\vecV) &=\color{red}\intOmega{\extd\vecU\cdot\vecV+\vecU\cdot\extd\vecV}, &\vecU,\vecV\in\Hd,
\end{align}
\end{subequations}
that rest on an equal footing. They readily appear upon integrating by parts with \Cref{eq: IBP weakform} and they are involved in the first-order analogs of Green's identities
\begin{subequations}
\begin{align}
\int_{\Omega^{\mp}}{\Dirac\vecU\cdot\vecV}&=\mathcal{A}_{\extdel}(\vecU,\vecV)\pm\llangle\traceT\vecU,\traceR\vecV\rrangle_{\Gamma},\label{eq: Green first delta}\\
\int_{\Omega^{\mp}}{\Dirac\vecU\cdot\vecV}&=\mathcal{A}_{\extd}(\vecU,\vecV)\mp\llangle\traceT\vecV,\traceR\vecU\rrangle_{\Gamma},\label{eq: Green first d}
\end{align}
\end{subequations}
which hold for all $\vecU,\vecV\in\HD$. 

These identities lead to the variational problems:
\begin{align}
\color{blue}
&\vecU\in\Hdelta\,: &\quad \color{blue}\mathcal{A}_{\extdel}(\vecU,\vecV)&=\color{blue}-\llangle\vecb,\traceR\vecV\rrangle_{\Gamma}, &\quad &\forall\,\vecV\in\Hdelta,\tag{V\textsf{T}}\label{pb: VT}
\end{align}
and
\begin{align}
&\vecU\in\Hd\,: &\quad \color{red}\mathcal{A}_{\extd}(\vecU,\vecV)&=\color{red}\llangle\veca,\traceT\vecV\rrangle_{\Gamma}, &\quad &\,\,\,\,\forall\,\vecV\in\Hd.\tag{V\textsf{R}}\label{pb: VR}
\end{align}

\subsection{Compatibility conditions} Either from  Green's second formula for the Dirac operator \eqref{eq: IBP Dirac} or the variational problems themselves, we see that the boundary values $\vecb\in\HT$ and $\veca\in\HR$ must fulfill compatibility conditions. For the problems to admit solutions, we require that
\begin{align}\label{CCT}
   {\color{blue} \llangle\vecb,\traceR\vecV\rrangle_{\Gamma}=\vec{\bl{0}},\qquad\qquad\forall\,\vecV\in\mathfrak{H}_{T}\tag{CC\textsf{T}},}
\end{align}
and
\begin{align}\label{CCR}
   {\color{red} \llangle\veca,\traceT\vecV\rrangle_{\Gamma}=\vec{\bl{0}}, \qquad\qquad\forall\,\vecV\in\mathfrak{H}_{R}\tag{CC\textsf{R}},}
\end{align}
where
\begin{subequations}
\begin{align}
\color{blue}\mathfrak{H}_{\mathsf{T}}(\Omega)&:=\color{blue} \left\{\vecV\in\HD : \Dirac\vecV =0, \traceT\vecV=\vec{\bl{0}}\right\}
\end{align}
and
\begin{align}
\color{red}\mathfrak{H}_{\mathsf{R}}(\Omega) &\color{red}:= \{\vecV\in\HD : \Dirac\vecV =0, \traceR\vecV=\vec{\bl{0}}\}
\end{align}
\end{subequations}
are spaces of harmonic vector-fields. We refer to \cite{arnold2006finite,arnold2010finite,arnold2018finite} and \cite{leopardi2016abstract} for explanations on how these spaces exactly correspond to the nullspaces of the Hodge-Laplacian with natural and essential boundary conditions.

The fact that there are two distinct bilinear forms in the expressions \eqref{pb: VT} and \eqref{pb: VR} is one of the appealing use of the dual perspective involving the codifferential $\bm{\delta}$. It points to the symmetry presented in \Cref{sec: Duality and symmetry} below, and it highlights the necessity of imposing compatibility conditions on the data. For example, we could alternatively formulate \eqref{pb: T} as the variational problem
\begin{align}\label{eq: essential boundary condition}
&\vecU\in\Hd\,\quad\text{with}\quad\traceT\vecU=\vecb: &\quad \mathcal{A}_{\extd}(\vecU,\vecV)&=0, &\quad &\,\,\,\,\forall\,\vecV\in\Hdnot,
\end{align}
where $\Hdnot= H^1_0(\Omega)\times \mathbf{H}_0(\mathbf{curl}, \Omega)\times \mathbf{H}_0(\text{div}, \Omega)\times L^2(\Omega)$. But according to \eqref{eq: IBP Dirac} the condition \eqref{CCT} must remain, and it now appears less obviously so when the type of boundary condition is \textit{essential}. Anyway, in a formulation such as \eqref{eq: essential boundary condition}, one proceeds with a lifting of the boundary data and is left with the solvability of the problem
\begin{align}
&\vecU_0\in\Hdnot\,: &\quad \mathcal{A}_{\extd}(\vecU,\vecV)&=-\mathcal{A}_{\extd}(\mathcal{E}_{\mathsf{T}}\vecb,\vecV), &\quad &\,\,\,\,\forall\,\vecV\in\Hdnot.
\end{align}
So the question of compatibly cannot be avoided: integrating by parts with the right-hand side evaluated at a nullspace element in $\mathfrak{H}_{\mathsf{T}}$ using \eqref{eq: Green first d} leads to \eqref{CCT}. We discuss in greater details the reason why the two boundary conditions can be formulated both as \textit{natural} and \textit{essential} in \Cref{sec: Duality and symmetry}.

\subsection{Well-posedness}
Since the bilinear form $\color{blue}\mathcal{A}_{\extdel}$ is associated with the self-adjoint operator $\color{blue}\mathsf{D}_{\mathsf{T}}$ obtained from the chain complex \eqref{domain delta adjoint complex} and $\color{red}\mathcal{A}_{\bl{d}}$ to the self-adjoint operator $\color{red}\mathsf{D}_{\mathsf{R}}$ spawned by the cochain complex \eqref{domain de Rham complex d 3D}, they fit the framework of \cite[Sec. 2]{leopardi2016abstract}. The abstract inf-sup inequality supplied in \cite[Thm. 6]{leopardi2016abstract} applies to both bilinear forms and leads to well-posedness of the mixed variational problems:
\begin{align}
\begin{split}\label{pb: MVT}
    \color{blue}\mathcal{A}_{\extdel}(\vecU,\vecV) + \bra{\vec{\bl{P}},\vecV}_{\Omega} &\color{blue}=-\llangle\vecb,\traceR\vecV\rrangle_{\Gamma}  \quad\qquad\forall\,\vecV\in\bl{H}\bra{\bm{\delta},\Omega^-},\\
    \color{blue}\bra{\vecU,\vec{\bl{W}}}_{\Omega}&\color{blue}= 0 \qquad\qquad\qquad\qquad\forall\,\vec{\bl{W}}\in\ker\mathsf{D}^{\Omega}_{\mathsf{T}},
    \end{split}\tag{MV\textsf{T}}
\end{align}
and
\begin{align}
\begin{split}\label{pb: MVR}
    \color{red}\mathcal{A}_{\bl{d}}(\vecU,\vecV) + \bra{\vec{\bl{Q}},\vecV}_{\Omega} &\color{red}=\llangle\veca,\traceT\vecV\rrangle_{\Gamma}  \quad\qquad\forall\,\vecV\in\bl{H}\bra{\bl{d},\Omega^-},\\
    \color{red}\bra{\vecU,\vec{\bl{W}}}_{\Omega}&\color{red}= 0 \quad\qquad\qquad\qquad\forall\,\vec{\bl{W}}\in\ker\mathsf{D}^{\Omega}_{\mathsf{R}},
    \end{split}\tag{MV\textsf{R}}
\end{align}
for unknown pairs $(\vecU,\vec{\bl{P}})\in\bl{H}(\bm{\delta},\Omega^-)\times\ker\mathsf{D}_{\mathsf{T}}$ and $(\vecU,\vec{\bl{Q}})\in\bl{H}(\bl{d},\Omega^-)\times\ker\mathsf{D}_{\mathsf{R}}$.

Consistency of the right-hand side in \eqref{pb: VT} exactly corresponds to requiring that \eqref{CCT} holds for the given data $\vecb\in\HT$, while \eqref{CCR} similarly guarantees consistency of the right-hand side in \eqref{pb: VR}. We conclude that if the compatibility conditions are satisfied, solutions to \eqref{pb: VT} and \eqref{pb: VR} in $\Omega^-$ are unique up to contributions of harmonic vector-fields in $\ker\mathsf{D}_{\mathsf{T}}$ and $\ker\mathsf{D}_{\mathsf{R}}$. Moreover, they continuously depend on the boundary data.

\section{Representation formulas}
We derive interior and exterior representation formulas for solutions of the Dirac equation. It is expressed through known boundary potentials, whose jump properties across $\Gamma$ are elaborated.
\subsection{Fundamental solution}\label{subsec: Fundamental solution}
Convolution of a vector field $\vec{\bl{U}}:\mathbb{R}^3\rightarrow\mathbb{R}^8$ by a matrix-valued function $\mathsf{K}:\mathbb{R}^3\backslash\{\bl{0}\}\rightarrow\mathbb{R}^{8,8}$ possibly having a singularity at the $\bl{0}\in\mathbb{R}^3$ is defined, if the limit exists, as the Cauchy principal value
\begin{equation}\label{eq: conv by matrix}
\bra{\mathsf{K}*\vec{\bl{U}}}\bra{\mathbf{x}} :=\lim_{\epsilon\rightarrow 0}\int_{\mathbb{R}^3\backslash B_{\epsilon}(\bl{0})}\mathsf{K}(\bl{x}-\bl{y})\vec{\bl{U}}\bra{\bl{y}}\dif\bl{y}\in\mathbb{R}^8,
\end{equation}
where $B_{\epsilon}\bra{\bl{0}}\subset\mathbb{R}^3$ is a ball of radius $\epsilon$ centered at the origin.

Let $G:\mathbb{R}^3\backslash\left\{\bl{0}\right\}\rightarrow\mathbb{R}$ be given by $G\bra{\bl{z}}:=\bra{4\pi\abs{\bl{z}}}^{-1}$, and set 
\begin{equation}
\mathsf{G}\bra{\mathbf{z}} := G\bra{\bl{z}}
\mathsf{I}_{8}\in \mathbb{R}^{8,8},\qquad \bl{z}\neq\bl{0},
\end{equation} 
where $\mathsf{I}_{8}$ is the identity matrix on $\mathbb{R}^8$. Then, define $\Phi:\mathbb{R}^3\backslash\{\bl{0}\}\rightarrow \mathbb{R}^{8,8}$ by applying the Dirac operator to the columns of $\mathsf{G}$ as
\begin{equation*}
	\Phi\bra{\bl{z}} :=
	\begin{pmatrix}
	0 & - \bra{\nabla G}^{\top}\bra{\bl{z}} & \bl{0}^\top & 0\\
	\bra{\nabla G}\bra{\bl{z}} & \bl{0}_{\,3\times 3} &
	\mathsf{A}_{3\times 3}(\bl{z})
	& \bl{0}\\
	\bl{0} & \mathsf{A}_{3\times 3}(\bl{z}) & \bl{0}_{\,3\times 3} & - \bra{\nabla G}\bra{\bl{z}}\\
	0 & \bl{0}^{\top} & \bra{\nabla G}^{\top}\bra{\bl{z}} & 0
	\end{pmatrix}\in\mathbb{R}^{8\times 8},\,\,\bl{z}\neq \bl{0},
\end{equation*}
\noindent where the anti-symmetric blocks
\begin{equation}
\mathsf{A}_{3\times 3}(\bl{z}) := \begin{pmatrix}
0 & - \bra{\partial_3G}\bra{\bl{z}} & \bra{\partial_2G}\bra{\bl{z}}\\
\bra{\partial_3G}\bra{\bl{z}} & 0 & -\bra{\partial_1G}\bra{\bl{z}}\\
-\bra{\partial_2G}\bra{\bl{z}} & \bra{\partial_1G}\bra{\bl{z}} & 0
\end{pmatrix}\in\mathbb{R}^{3\times 3}, \qquad \bl{z}\neq \bl{0},
\end{equation}
are associated with the curl operator.

\begin{lemma}\label{lem: anti-symmetric properties of fundamental sol}
	For $\bl{z}\neq \bl{0}$,
	\begin{align}
	\Phi\bra{-\bl{z}} = - \Phi\bra{\bl{z}} &&\text{and}
	 &&\Phi\bra{\bl{z}}\vec{\bl{U}}\cdot\vec{\bl{V}}=-\vec{\bl{U}}\cdot\Phi\bra{\bl{z}}\vec{\bl{V}}
	\end{align}
	for all $\vec{\bl{U}},\vec{\bl{V}}\in\mathbb{R}^8$.
\end{lemma}
\begin{proof}
	Let $\mathsf{s}:\mathbb{R}^3\rightarrow\mathbb{R}^3$ be the sign flip operation $\mathsf{s}(\bl{z})=-\bl{z}$. For the fist identity, we simply rely on the fact that $G\bra{\bl{x}}=G\bra{\,\abs{\bl{x}}}$ to verify that for any $\vec{\bl{U}}\in\mathbb{R}^8$, 
	\begin{multline}
	\Phi\bra{-\bl{z}}\vec{\bl{U}}=\mathsf{D}\bra{\mathsf{G}\vec{\bl{U}}}\Big\vert_{\mathsf{s}\bra{\bl{z}}}= -\mathsf{D}_{\bl{x}}\bra{\mathsf{G}\bra{\sf{s}\bra{\bl{x}}}\vec{\bl{U}}}\Big\vert_{\bl{x}=\bl{z}} \\= -\mathsf{D}_{\bl{x}}\bra{G\bra{\sf{s}\bra{\bl{x}}}\vec{\bl{U}}}\Big\vert_{\bl{x}=\bl{z}}
	= -\mathsf{D}_{\bl{x}}\bra{G\bra{\bl{x}}\vec{\bl{U}}}\Big\vert_{\bl{x}=\bl{z}}= -\Phi\bra{\bl{z}}\vec{\bl{U}}.
	\end{multline}
    The second identity is clear by definition.
\end{proof}

This lemma allows to extend the domain of the Newton-type potential
\begin{align*}
\mathsf{N}:C_0^{\infty}(\mathbb{R}^3)^{8}&\rightarrow C^{\infty}(\mathbb{R}^3)^{8}\\
\vec{\bl{U}}&\mapsto\Phi * \vec{\bl{U}}
\end{align*}
to distributions.

\begin{lemma}\label{lem: symmetry Newton}
	For all $\vec{\bl{U}},\vec{\bl{V}}\in C^{\infty}_0(\mathbb{R}^3)^8$,
	\begin{equation}
	\left(\mathsf{N}\,\vec{\bl{U}},\vec{\bl{V}}\right) = \left(\vec{\bl{U}},\mathsf{N}\,\vec{\bl{V}}\right).
	\end{equation}
\end{lemma}
\begin{proof}
	Using \Cref{lem: anti-symmetric properties of fundamental sol}, we can change the order of integration using  Fubini's theorem and evaluate
	\begin{align}
	\left(\mathsf{N}\vec{\bl{U}},\vec{\bl{V}}\right) &= \int_{\mathbb{R}^3}\int_{\mathbb{R}^3}\Phi\bra{\bl{x}-\bl{y}}\vec{\bl{U}}\bra{\bl{y}}\cdot\vec{\bl{V}}\bra{\bl{x}}\dif\bl{x}\dif\bl{y}\\
	&=\int_{\mathbb{R}^3}\int_{\mathbb{R}^3}\vec{\bl{U}}\bra{\bl{y}}\cdot\Phi\bra{\bl{y}-\bl{x}}\vec{\bl{V}}\bra{\bl{x}}\dif\bl{x}\dif\bl{y}\\
	&=\int_{\mathbb{R}^3}\vec{\bl{U}}\bra{\bl{y}}\cdot\int_{\mathbb{R}^3}\Phi\bra{\bl{y}-\bl{x}}\vec{\bl{V}}\bra{\bl{x}}\dif\bl{x}\dif\bl{y}\\
	&=\left(\vec{\bl{U}},\mathsf{N}\,\vec{\bl{V}}\right).
	\end{align}
\end{proof}

\begin{remark}\Cref{lem: symmetry Newton} reflects the fact that the Dirac operator is symmetric as an unbounded operator on $(L^2(\mathbb{R}^3))^8$.
\end{remark} 

The extension 
\begin{equation}
\mathsf{N}:(C^{\infty}(\mathbb{R}^3)^8)'\rightarrow (C^{\infty}_0(\mathbb{R}^3)^8)'
\end{equation}
is obtained as in \cite[Sec. 3.1.1]{sauter2010boundary} via dual mapping by defining the action of the distribution $\mathsf{N}\vec{\bl{U}}\in (C^{\infty}_0(\mathbb{R}^3)^8)'$ on $\vec{\bl{V}}\in C^{\infty}_0(\mathbb{R}^3)^8$ as
\begin{equation}
\langle\mathsf{N}\,\vec{\bl{U}}, \vec{\bl{V}}\rangle := \langle\vec{\bl{U}}, \mathsf{N}\,\vec{\bl{V}}\rangle.
\end{equation}

\begin{proposition}[{Fundamental solution}]\label{prop: fundamental sol}
For all compactly supported distributions $\vec{\bl{U}}\in (C^\infty(\mathbb{R}^3)^8)'$,
	\begin{equation}
	\mathsf{N}\,\mathsf{D}\,\vec{\bl{U}} =\vec{\bl{U}} =\mathsf{D}\,\mathsf{N}\,\vec{\bl{U}}
	\end{equation}
holds in $(C^{\infty}_0(\mathbb{R}^3)^8)'$.
\end{proposition}
\begin{proof}
		We first show that for $\vec{\bl{U}}\in (C^{\infty}(\mathbb{R}^3)^8)'$, 
	\begin{equation}
	\langle\mathsf{N}\,\mathsf{D}\vec{\mathbf{U}},\vec{\mathbf{V}} \rangle = \langle\vec{\mathbf{U}},\vec{\mathbf{V}} \rangle
	\end{equation}
	for all $\vec{\bl{V}}\in C^{\infty}_0(\mathbb{R}^3)^8$.
	
	The argument is inspired by the proof of \cite[Thm.1]{evans2010partial}. Let $\bl{e}_i\in\mathbb{R}^3$ be the vector with $1$ at the i-th entry and zeros elsewhere, $i=1,2,3$. Since
	\begin{equation}
	\mathsf{N}\vec{\bl{V}}=\int_{\mathbb{R}^3}\Phi\bra{\bl{x}-\bl{y}}\vec{\bl{V}}\bra{\bl{y}}\dif\bl{y} = \int_{\mathbb{R}^3}\Phi\bra{\bl{y}}\vec{\bl{V}}\bra{\bl{x}-\bl{y}}\dif\bl{y},
	\end{equation}
	we have 
	\begin{equation}
	\frac{\mathsf{N}\vec{\bl{V}}\bra{\bl{x}+h\bl{e}_i}- \mathsf{N}\vec{\bl{V}}\bra{\bl{x}}}{h}=\int_{\mathbb{R}^3}\Phi\bra{\bl{y}}    \frac{\vec{\bl{V}}\bra{\bl{x}+h\bl{e}_i-\bl{y}}- \vec{\bl{V}}\bra{\bl{x}-\bl{y}}}{h}\dif\bl{y}.
	\end{equation}
	Hence,
	\begin{equation}\label{eq: Dirac of Newton inside integral}
	\mathsf{D}_{\bl{x}}\mathsf{N}\vec{\bl{V}}\bra{\bl{x}}=\int_{\mathbb{R}^3}\Phi\bra{\bl{y}}\mathsf{D}\vec{\bl{V}}\bra{\bl{x}-\bl{y}}\dif\bl{y},
	\end{equation}
	because the assumption that $\vec{\bl{V}}$ is smooth and compactly supported guarantees that
	\begin{equation}
	\frac{\vec{\bl{V}}\bra{\bl{x}+h\bl{e}_i-\bl{y}}- \vec{\bl{V}}\bra{\bl{x}-\bl{y}}}{h} \rightarrow \frac{\partial}{\partial\bl{x}_i}\vec{\bl{V}}\bra{\bl{x}-\bl{y}}
	\end{equation}
	uniformly for $h\rightarrow 0$. The main idea is to isolate $\Phi$'s singularity at the origin by splitting the right hand side of \cref{eq: Dirac of Newton inside integral} into two integrals as
	\begin{equation}\label{eq: split integral isolate singularity}
	\mathsf{D}_{\bl{x}}\mathsf{N}\vec{\bl{V}}\bra{\bl{x}} = \underbrace{\int_{B_{\epsilon}\bra{\bl{0}}}\Phi\bra{\bl{y}}\mathsf{D}\vec{\bl{V}}\bra{\bl{x}-\bl{y}}\dif\bl{y}}_{I_{\epsilon}} + \underbrace{\int_{\mathbb{R}^3\backslash B_{\epsilon}\bra{\bl{0}}}\Phi\bra{\bl{y}}\mathsf{D}\vec{\bl{V}}\bra{\bl{x}-\bl{y}}\dif\bl{y}}_{J_{\epsilon}}
	\end{equation}
	whose limits as $\epsilon\rightarrow 0$ we can control. 
	
	The main  difficulty is that we cannot readily mimic the standard proof commonly given for the Poisson equation, because the integration by parts formula supplied for the product of two vectors by \cref{eq: IBP Dirac} is not applicable to the matrix--vector multiplication involved in the integrands of \cref{eq: split integral isolate singularity}. The analysis of
	\begin{multline}
	\Phi\bra{\bl{y}}\mathsf{D}\vec{\bl{V}}\bra{\bl{x}-\bl{y}}=\\ 
	\begin{pmatrix}
	-{\color{alizarin}\nabla G\bra{\bl{y}}\cdot \nabla V_0\bra{\bl{x}-\bl{y}}} -{\color{ao}\nabla G\bra{\bl{y}}\cdot \mathbf{curl}\,\bl{V}_1\bra{\bl{x}-\bl{y}}}\\
	- {\color{orange}\text{div}\,\bl{V}_1\bra{\bl{x}-\bl{y}}\nabla G\bra{\bl{y}}}-{\color{violet}\nabla G\bra{\bl{y}}\times \nabla V_3\bra{\bl{x}-\bl{y}}} + {\color{teal}\nabla G\bra{\bl{y}}\times \bl{curl}\, \bl{V}_1\bra{\bl{x}-\bl{y}}}\\
	-{\color{violet}\nabla G\bra{\bl{y}}\times\nabla V_0\bra{\bl{x}-\bl{y}}}+{\color{teal}\nabla G\bra{\bl{y}}\times\bl{curl}\,\mathbf{V}_2\bra{\bl{x}-\bl{y}}} - {\color{orange}\text{div}\,\bl{V}_2\bra{\bl{x}-\bl{y}}\nabla G\bra{\bl{y}}}\\
	-{\color{alizarin}\nabla G\bra{\bl{y}}\cdot \nabla V_3\bra{\bl{x}-\bl{y}}} +{\color{ao} \nabla G\bra{\bl{y}}\cdot \bl{curl}\,\bl{V}_2\bra{\bl{x}-\bl{y}}}
	\end{pmatrix}
	\end{multline}
	is carried out component-wise.
	
	There are five different types of terms whose limit need to be investigated. Let $\bl{V}\in (C^{\infty}_0(\mathbb{R}^3))^3$ and $V\in C^{\infty}_0(\mathbb{R}^3)$ be arbitrary fields. To ease the reading, we write $V_{\bl{x}}(\bl{y}):=V(\bl{x}-\bl{y})$ and $\bl{V}_{\bl{x}}(\bl{y}):=\bl{V}(\bl{x}-\bl{y})$ . We denote by $\bl{n}_{\epsilon}$ the unit normal vector field pointing towards the interior of $B_{\epsilon}\bra{\bl{0}}$.
	
	Integrating by parts using that $\Delta G = 0$ in $\mathbb{R}^3\backslash\bra{\bl{0}}$ and $\bl{curl}\circ\nabla \equiv \bl{0}$, we find that
	\begin{multline}
	\int_{\mathbb{R}^3\backslash B_{\epsilon}\bra{\bl{0}}}{\color{alizarin}\nabla G\bra{\bl{y}}\cdot \nabla V\bra{\bl{x}-\bl{y}}}\dif\bl{y} =  \int_{\partial B_{\epsilon}\bra{\bl{0}}}\nabla G\bra{\bl{y}}\cdot\bl{n}_{\epsilon}\bra{\bl{y}}V\bra{\bl{x}-\bl{y}}\dif\sigma\bra{\bl{y}}\\
	=\frac{1}{4\pi}\int_{\partial B_{\epsilon}\bra{\bl{0}}}\frac{V\bra{\bl{x}-\bl{y}}}{\abs{\bl{y}}^3}\bra{-\bl{y}\cdot\frac{\bl{y}}{\abs{\bl{y}}}}\dif\sigma\bra{\bl{y}}
	=-\frac{1}{4\pi\epsilon^2}\int_{\partial B_{\epsilon}\bra{\bl{0}}}V\bra{\bl{x}-\bl{y}}\dif\sigma\bra{\bl{y}}\\
	=-\dashint_{\partial B_{\epsilon}\bra{\bl{x}}}V\bra{\bl{y}}\dif\sigma\bra{\bl{y}}\xrightarrow[\epsilon\rightarrow 0]{} - V\bra{\bl{x}}
	\end{multline}
	and
	\begin{multline}
	\int_{\mathbb{R}^3\backslash B_{\epsilon}\bra{\bl{0}}}{\color{ao}\nabla G\bra{\bl{y}}\cdot \mathbf{curl}\,\bl{V}\bra{\bl{x}-\bl{y}}}\dif\bl{y} \\
	=-\int_{\partial B_{\epsilon}\bra{\bl{0}}}\bra{\nabla  G\bra{\bl{y}}\times\bl{n}_{\epsilon}\bra{\bl{y}}}\cdot\bl{V}\bra{\bl{x}-\bl{y}}\dif\sigma\bra{\bl{y}}\\
	=-\frac{1}{4\pi\epsilon^4}\int_{\partial B_{\epsilon}\bra{\bl{0}}}\bra{\bl{y}\times\bl{y}}\cdot\bl{V}\bra{\bl{x}-\bl{y}}\dif\sigma\bra{\bl{y}}
	=0.
	\end{multline}
	Similarly, integrating by parts component-wise yields
	\begin{multline}
	\int_{\mathbb{R}^3\backslash B_{\epsilon}\bra{\bl{0}}}{\color{violet}\nabla G\bra{\bl{y}}\times \nabla V_{\bl{x}}(\bl{y})} \dif\bl{y}\\
	 =
	\int_{\mathbb{R}^3\backslash B_{\epsilon}\bra{\bl{0}}}
	\begin{pmatrix}
	\partial_2 G\bra{\bl{y}}\partial_3 V_{\bl{x}}(\bl{y}) - \partial_3 G\bra{\bl{y}}\partial_2V_{\bl{x}}(\bl{y})\\
	\partial_3 G\bra{\bl{y}}\partial_1 V_{\bl{x}}(\bl{y}) - \partial_1 G\bra{\bl{y}}\partial_3V_{\bl{x}}(\bl{y})\\
	\partial_1 G\bra{\bl{y}}\partial_2 V_{\bl{x}}(\bl{y}) - \partial_2 G\bra{\bl{y}}\partial_1V_{\bl{x}}(\bl{y})
	\end{pmatrix} \dif\bl{y}\\
	=
	\int_{\mathbb{R}^3\backslash B_{\epsilon}\bra{\bl{0}}}
	\begin{pmatrix}
	G\bra{\bl{y}}\partial_2\partial_3 V_{\bl{x}}(\bl{y}) - G\bra{\bl{y}}\partial_3\partial_2V_{\bl{x}}(\bl{y})\\
	G\bra{\bl{y}}\partial_3\partial_1 V_{\bl{x}}(\bl{y}) - G\bra{\bl{y}}\partial_1\partial_3V_{\bl{x}}(\bl{y})\\
	G\bra{\bl{y}}\partial_1\partial_2 V_{\bl{x}}(\bl{y}) - G\bra{\bl{y}}\partial_2\partial_1V_{\bl{x}}(\bl{y})
	\end{pmatrix} \dif\bl{y}\\
	+
	\int_{\partial B_{\epsilon}\bra{\bl{0}}}
	\begin{pmatrix}
	- \bra{\bl{n}_{\epsilon}}_2\bra{\bl{y}}G\bra{\bl{y}}\partial_3 V_{\bl{x}}(\bl{y}) + \bra{\bl{n}_{\epsilon}}_3\bra{\bl{y}}G\bra{\bl{y}}\partial_2V_{\bl{x}}(\bl{y})\\
	- \bra{\bl{n}_{\epsilon}}_2\bra{\bl{y}}G\bra{\bl{y}}\partial_1 V_{\bl{x}}(\bl{y}) + \bra{\bl{n}_{\epsilon}}_1\bra{\bl{y}}G\bra{\bl{y}}\partial_3V_{\bl{x}}(\bl{y})\\
	-\bra{\bl{n}_{\epsilon}}_1\bra{\bl{y}} G\bra{\bl{y}}\partial_2 V_{\bl{x}}(\bl{y}) + \bra{\bl{n}_{\epsilon}}_2\bra{\bl{y}}G\bra{\bl{y}}\partial_1V_{\bl{x}}(\bl{y})
	\end{pmatrix} \dif\bl{y}.
	\end{multline}
	Since $V$ is smooth everywhere in $\mathbb{R}^3$, partial derivatives commute and the volume integral vanishes, leading to
	\begin{align}
	\int_{\mathbb{R}^3\backslash B_{\epsilon}\bra{\bl{0}}}{\color{violet}\nabla G\bra{\bl{y}}\times \nabla V_{\bl{x}}(\bl{y})} \dif\bl{y} = - \int_{\partial B_{\epsilon}\bra{\bl{0}}} G\bra{\bl{y}}\bl{n}_{\epsilon}\bra{\bl{y}}\times\nabla V_{\bl{x}}(\bl{y})\dif\sigma\bra{\bl{y}}.
	\end{align}
	This integral vanishes under the limit $\epsilon\rightarrow 0$, because
	\begin{multline}\label{eq: boundary term is O epsilon}
	\sup_{\bl{x}\in\mathbb{R}^3} \,\abs{\int_{\partial B_{\epsilon}\bra{\bl{0}}} G\bra{\bl{y}}\bl{n}_{\epsilon}\bra{\bl{y}}\times\nabla V\bra{\bl{x}-\bl{y}}\dif\sigma\bra{\bl{y}}}\\
	 \leq \norm{\nabla V}_{\infty} \int_{\partial B_{\epsilon}\bra{\bl{0}}} \abs{G\bra{\bl{y}}}\dif\sigma\bra{\bl{y}}=\mathcal{O}\bra{\epsilon}.
	\end{multline}
	Moving on to the next term, one eventually obtains from similar calculations that
	\begin{multline}
	\int_{\mathbb{R}^3\backslash B_{\epsilon}\bra{\bl{0}}}{\color{teal}\nabla G\bra{\bl{y}}\times\bl{curl}\,\mathbf{V}_\bl{x}(\bl{y})} \dif\bl{y} = \int_{\mathbb{R}^3\backslash B_{\epsilon}\bra{\bl{0}}}G\bra{\bl{y}}\bl{curl}\,\bl{curl}\,\mathbf{V}_\bl{x}(\bl{y}) \dif\bl{y}\\
	+\int_{\partial B_{\epsilon}\bra{\bl{0}}} G\bra{\bl{y}}\bra{ \bl{curl}\,\mathbf{V}_\bl{x}(\bl{y})\times\bl{n}_{\epsilon}\bra{\bl{y}}}\dif\sigma\bra{\bl{y}}. 
	\end{multline}
	Since $\norm{\bl{curl}\,\bl{V}}_{\infty} < \infty$, the boundary integral on the right hand side vanishes under the limit by repeating the argument of \cref{eq: boundary term is O epsilon}.
	Finally, commuting partial derivatives after integrating by parts also yields
	\begin{multline}
	\int_{\mathbb{R}^3\backslash B_{\epsilon}\bra{\bl{0}}}{\color{orange}\text{div}\,\bl{V}\bra{\bl{x}-\bl{y}}\nabla G\bra{\bl{y}}}\dif\bl{y}\\
	 = \int_{\mathbb{R}^3\backslash B_{\epsilon}\bra{\bl{0}}} G\bra{\bl{y}}\nabla\text{div}\bl{V}\bra{\bl{x}-\bl{y}}
	-\int_{\partial B_{\epsilon}\bra{\bl{0}}} G\bra{\bl{y}}\text{div}\bl{V}\bra{\bl{x}-\bl{y}}\bl{n}_{\epsilon}\bra{\bl{y}}\dif\sigma\bra{\bl{y}}
	\end{multline}
	
	Putting the two previous calculations together, we find that
	\begin{multline}
	\lim_{\epsilon\rightarrow 0}\int_{\mathbb{R}^3\backslash B_{\epsilon}\bra{\bl{0}}}{\color{teal}\nabla G\bra{\bl{y}}\times\bl{curl}\,\mathbf{V}\bra{\bl{x}-\bl{y}}}  -{\color{orange}\text{div}\,\bl{V}\bra{\bl{x}-\bl{y}}\nabla G\bra{\bl{y}}}\dif\bl{y}
	\\= - \lim_{\epsilon\rightarrow 0}\int_{\mathbb{R}^3\backslash B_{\epsilon}\bra{\bl{0}}}G\bra{\bl{y}}\bm{\Delta}\bl{V}\bra{\bl{x}-\bl{y}}\dif\bl{y}
	= \bl{V}\bra{\bl{x}},
	\end{multline}
	where we recognized the vector (Hodge-) Laplace operator $-\bm{\Delta}\equiv\bl{curl}\,\bl{curl}-\nabla\,\text{div}$.
	
	We have found that $J_{\epsilon}\longrightarrow \vec{\bl{V}}\bra{\bl{x}}$ as $\epsilon\rightarrow 0$. Meanwhile,
	\begin{equation}
	\norm{I_{\epsilon}}_{\infty} \leq \norm{\mathsf{D}\vec{\bl{V}}}_{\infty} \int_{B_{\epsilon}\bra{\bl{0}}}\norm{\Phi}_{\infty}\dif\bl{y} = \mathcal{O}\bra{\int_{B_{\epsilon}\bra{\bl{0}}}\norm{\nabla G}_{\infty}\dif\bl{y}}=\mathcal{O}\bra{\epsilon}.
	\end{equation}
	
	The calculations for $\vec{\bl{U}} =\mathsf{D}\,\mathsf{N}\,\vec{\bl{U}}$ follow similarly starting from \eqref{eq: Dirac of Newton inside integral}.
\end{proof}

In light of \Cref{prop: fundamental sol}, we say that the kernel $\Phi$ of $\mathsf{N}$ is a fundamental solution for the Dirac operator.

\subsection{Surface potentials}\label{sec: properties of the potentials}
Adopting the perspective on first-kind boundary integral operators from \cite{costabel1988boundary}, \cite{mclean2000strongly}, \cite{sauter2010boundary} and \cite{claeys2017first}---in the later works for the study of second-order elliptic operators---for the first-order Dirac operator, we define the surface potentials
\begin{align}
\mathcal{L}_{\mathsf{T}}\bra{\vec{\bl{a}}} &:= \mathsf{N}\bra{ \gamma_{\mathsf{T}}'\vec{\bl{a}}},   &\forall\,\vec{\bl{a}}=\bra{\mathcal{a}_0,\bl{a}_1,a_2}\in\mathcal{H}_{\mathsf{R}},\\
\mathcal{L}_{\mathsf{R}}(\vec{\bl{b}}) &:= -\mathsf{N}\bra{ \gamma_{\mathsf{R}}'\vec{\bl{b}}},  &\forall\,\vec{\bl{b}}=\bra{b_0,\bl{b}_1,\mathcal{b}_2}\in\mathcal{H}_{\mathsf{T}},
\end{align}
where the mappings $\gamma_{\mathsf{T}}':\mathcal{H}_{\mathsf{R}}=\HT'\rightarrow \bl{H}_{\text{loc}}(\mathsf{D},\mathbb{R}^3\backslash\overline{\Omega})'$ and $\gamma_{\mathsf{R}}':\mathcal{H}_{\mathsf{T}}=\HR'\rightarrow \bl{H}_{\text{loc}}(\mathsf{D},\mathbb{R}^3\backslash\overline{\Omega})'$ are adjoint to the trace operators $\gamma_{\mathsf{T}}$ and $\gamma_{\mathsf{R}}$ defined in \eqref{eq: little wedge and vee mappings}.

It will be convenient to denote by $\Phi_{\bl{x}}$ the map $\bl{y}\mapsto\Phi\bra{\bl{x}-\bl{y}}$. Let $\vec{\mathbf{E}}_j\in\mathbb{R}^8$ denote the constant vector with $1$ at the $j$-th entry and zeros elsewhere, $j=1,...,8$. Similarly for $\bl{E}_k\in\mathbb{R}^3$, $k=1,2,3$.

Adapting the calculations found in \cite[Sec. 4.2]{claeys2017first}, we will establish integral representation formulas for these potentials by splitting the pairings into their components.	

\begin{lemma}\label{lem: int rep of boundary potential}
	Given  $\vec{\bl{a}}\in\mathcal{H}_{\mathsf{R}}$ and $\vec{\bl{b}}\in\mathcal{H}_{\mathsf{T}}$, it holds for $\bl{x}\in\Omega\backslash\Gamma$ that
	\begin{subequations}
		\begin{align}
		\mathcal{L}_{\mathsf{T}}\bra{\vec{\bl{a}}}\bra{\bl{x}}\cdot \vec{\bl{E}}_j &= -\llangle\vec{\bl{a}}\,,\,\gamma^-_{\mathsf{T}}\bra{\Phi_{\bl{x}}\,\vec{\bl{E}}_j}\rrangle_{\Gamma},\\
		\mathcal{L}_{\mathsf{R}}(\vec{\bl{b}})\bra{\bl{x}}\cdot \vec{\bl{E}}_j &=\llangle\vec{\bl{b}}\,,\,\gamma^-_{\mathsf{R}}\bra{\Phi_{\bl{x}}\,\vec{\bl{E}}_j}\rrangle_{\Gamma}.
		\end{align}
	\end{subequations}
\end{lemma}
\begin{proof}
	Let $V\in C^{\infty}_{0}(\mathbb{R}^3)$ and suppose that $\vec{\bl{a}}$ is the trace of a smooth $8$-dimensional vector-field. Using Fubini's theorem and the fact that $\Phi$ is smooth away from the origin,
	\begin{align}
	\langle\mathsf{N} \bra{\gamma_{\mathsf{T}}'\vec{\bl{a}}},\,V\vec{\bl{E}}_j\rangle_{\mathbb{R}^3}
	 &= \llangle\vec{\bl{a}},\gamma_{\mathsf{T}}\mathsf{N}\bra{V\vec{\bl{E}}_j}\rrangle_{\Gamma}\\
	&=\int_{\Gamma}\vec{\bl{a}}\bra{\bl{y}}\cdot\gamma_{\mathsf{T}}\int_{\mathbb{R}^3}\Phi\bra{\bl{y}-\bl{x}}V\bra{\bl{x}}\vec{\bl{E}}_j\bra{\bl{x}}\dif\bl{x}\dif\sigma\bra{\bl{y}}\\
	&\starequal-\int_{\mathbb{R}^3} V\bra{\bl{x}}\bra{\int_{\Gamma}\vec{\bl{a}}\bra{\bl{y}}\cdot\gamma_{\mathsf{T}}\Phi{\bra{\bl{x}-\bl{y}}}\vec{\bl{E}}_j\dif\sigma\bra{\bl{y}}}\dif\bl{x},\label{eq: int rep formula of components}
	\end{align}
	where the sign was obtained in ($*$) thanks to \Cref{lem: anti-symmetric properties of fundamental sol}. The integrals on the right-hand side of \eqref{eq: int rep formula of components} can be extended to duality pairings by a standard density argument exploiting \Cref{lem: traces are continuous and surjective}.
	
	Similar calculations can be carried out for $\mathcal{L}_{\mathsf{R}}$.
\end{proof}
In particular,
\begin{align}
\Phi_{\bl{x}}\bra{\bl{y}}\,\vec{\mathbf{E}}_1
&= 
\begin{pmatrix}
0\\
\nabla G\bra{\bl{x}-\bl{y}}\\
\bl{0}\\
0
\end{pmatrix},
&
\Phi_{\bl{x}}\bra{\bl{y}}\,\vec{\mathbf{E}}_8&=
\begin{pmatrix}
0\\
\bl{0}\\
-\nabla G\bra{\bl{x}-\bl{y}}\\
0
\end{pmatrix},
\end{align}
\begin{align}
\Phi_{\bl{x}}\bra{\bl{y}}\,\vec{\mathbf{E}}_i &= 
\begin{pmatrix}
-\frac{\partial}{\partial\bl{z}_{\mu(i)}}G\bra{\bl{z}}\, \\
\bl{0}\\
\nabla G\bra{\bl{z}}\times\bl{E}_{\mu\bra{i}}\\
0
\end{pmatrix}\Big\vert_{\bl{z}=\bl{x}-\bl{y}},
&
\Phi_{\bl{x}}\bra{\bl{y}}\,\vec{\mathbf{E}}_k &=
\begin{pmatrix}
0\\
\nabla G\bra{\bl{z}}\times\bl{E}_{\nu\bra{k}}\\
\bl{0}\\
\frac{\partial}{\partial\bl{z}_{\nu(k)}}G\bra{\bl{z}}\,
\end{pmatrix}\Big\vert_{\bl{z}=\bl{x}-\bl{y}},
\end{align}
for $i=2,3,4$, $k=5,6,7$, $\mu\bra{i} = i-1$ and $\nu\bra{k}=k-4$. 

Therefore, we can evaluate
\begin{subequations}
	\begin{align}
	\mathcal{L}_{\mathsf{T}}\bra{\vec{\bl{a}}}\bra{\bl{x}}\cdot \vec{\bl{E}}_1 
	&= -\int_{\Gamma}\bl{a}_1\bra{\bl{y}}\cdot\nabla G\bra{\bl{x}-\bl{y}}\dif\sigma\bra{\bl{y}}
	\\
	\mathcal{L}_{\mathsf{T}}\bra{\vec{\bl{a}}}\bra{\bl{x}}\cdot\vec{\bl{E}}_i &= \int_{\Gamma}\mathcal{a}_0\bra{\bl{y}}\,\partial_{\mu(i)}G\bra{\bl{x}-\bl{y}}\dif\sigma\bra{\bl{y}}\\
	&\qquad\qquad-\int_{\Gamma}a_2\bra{\bl{y}}\,\bra{\nabla G\bra{\bl{x}-\bl{y}}\times\bl{E}_{\mu\bra{i}}}\cdot\bl{n}\bra{\bl{y}}\dif\sigma\bra{\bl{y}}\nonumber\\
	&=\partial_{\mu(i)}\int_{\Gamma}\mathcal{a}_0\bra{\bl{y}}G\bra{\bl{x}-\bl{y}}\dif\sigma\bra{\bl{y}}\nonumber\\
	&\qquad\qquad+\bl{E}_{\mu\bra{i}}\cdot\int_{\Gamma}a_2\bra{\bl{y}}\nabla G\bra{\bl{x}-\bl{y}}\times\bl{n}\bra{\bl{y}}\dif\sigma\bra{\bl{y}}\nonumber
\end{align}
\begin{align}
	\mathcal{L}_{\mathsf{T}}\bra{\vec{\bl{a}}}\bra{\bl{x}}\cdot\vec{\bl{E}}_k
	&=-\int_{\Gamma}\bl{a}_1\bra{\bl{y}}\cdot\bra{\nabla G\bra{\bl{x}-\bl{y}}\times\bl{E}_{\nu\bra{k}}}\dif\sigma\bra{\bl{y}}\\
	&=\bl{E}_{\nu\bra{k}}\cdot\int_{\Gamma}\nabla_{\bl{y}} G\bra{\bl{x}-\bl{y}}\times\bl{a}_1\bra{\bl{y}}\dif\sigma\bra{\bl{y}}\nonumber
	\\
	\mathcal{L}_{\mathsf{T}}\bra{\vec{\bl{a}}}\bra{\bl{x}}\cdot \vec{\bl{E}}_8 &= \int_{\Gamma}a_2\bra{\bl{y}}\nabla_{\bl{y}} G_{\bl{x}}\bra{\bl{y}}\cdot\bl{n}\bra{\bl{y}}\dif\sigma\bra{\bl{y}},
	\end{align}
\end{subequations}
where we have used the fact that $\bl{a}_1\in\mathbf{H}^{-1/2}(\text{div}_\Gamma,\Gamma)$ was ``tangential" to safely drop the trace $\gamma_t$ everywhere. Similarly as in the proof of \Cref{lem: int rep of boundary potential}, all these integrals should be understood as duality pairings and the following explicit representations do not only hold in the sense of distributions, but also pointwise on $\mathbb{R}^3\backslash\Gamma$.

We collect the above entries to obtain
\begin{greyFrame}
	\begin{equation}\label{eq: wedge potential}
	\mathcal{L}_{\mathsf{T}}\bra{\vec{\bl{a}}} =
	\begin{pmatrix}
	-\,\text{div}\,\bl{\Psi}\bra{\bl{a}_1}\\
	\nabla\psi\bra{\mathcal{a}_0} + \bl{curl}\,\bl{\Upsilon}\bra{a_2}\\
	\,\bl{curl}\,\bl{\Psi}\bra{\bl{a}_1}\\
	\text{div}\,\bl{\Upsilon}\bra{a_2}
	\end{pmatrix},\quad \text{pointwise on }\mathbb{R}^3\backslash\Gamma,
	\end{equation}
\end{greyFrame}
\noindent where we respectively recognize in
\begin{subequations}\label{eq: known scalar potentials}
	\begin{align}
	\psi(q)(\bl{x})&:=\int_{\Gamma}q(\bl{y})G\bra{\bl{x}-\bl{y}}\dif\sigma(\mathbf{y}), &\bl{x}\in\mathbb{R}^3\backslash\Gamma, \label{scalar single layer}\\
	\bm{\Psi}\bra{\mathbf{p}}\bra{\bl{x}}&:=\int_{\gamma}\mathbf{p}(\mathbf{y})G\bra{\bl{x}-\bl{y}}\dif\sigma(\mathbf{y}), &\bl{x}\in\mathbb{R}^3\backslash\Gamma,\label{vector single layer}\\
	\bl{\Upsilon}\bra{q}\bra{\bl{x}}&:=\int_{\Gamma}q\bra{\bl{y}}G\bra{\bl{x}-\bl{y}}\bl{n}(\bl{y})\dif\sigma\bra{\bl{y}}
	&\bl{x}\in\mathbb{R}^3\backslash\Gamma,\label{normal vector single layer}
	\end{align}
\end{subequations}
the well-known single layer,  vector single layer and normal vector single layer potentials. They notably enter \cref{eq: wedge potential} in the expression for the classical double layer potential $\text{div}\,\bl{\Upsilon}\bra{q}$ and for the Maxwell double layer potential $\bl{curl}\bl{\Psi}\bra{\bl{p}}$ as they arise in acoustic and electromagnetic scattering respectively.

Similarly, for $i=2,3,4$ and $k=5,6,7$,
\begin{subequations}
	\begin{align}
	\mathcal{L}_{\mathsf{R}}(\vec{\bl{b}})\bra{\bl{x}}\cdot \vec{\bl{E}}_1 &= \int_{\Gamma} b_0\bra{\bl{y}} \nabla G\bra{\bl{x}-\bl{y}}\cdot\bl{n}(\mathbf{y})\dif\sigma\bra{\bl{y}}\\
	\mathcal{L}_{\mathsf{R}}(\vec{\bl{b}})\bra{\bl{x}}\cdot \vec{\bl{E}}_i &=\int_{\Gamma}\bl{b}_1\bra{\bl{y}}\cdot\bra{\nabla G\bra{\bl{x}-\bl{y}}\times \bl{E}_{\mu\bra{i}}}\times \bl{n}(\bl{y})\dif\sigma\bra{\bl{y}}\\
	&=\int_{\Gamma}\bra{\nabla G\bra{\bl{x}-\bl{y}}\times \bl{E}_{\mu\bra{i}}}\cdot\bl{n}(\bl{y})\times\bl{b}_1\bra{\bl{y}}\dif\sigma\bra{\bl{y}}\nonumber\\
	&=\,\bl{E}_{\mu\bra{i}}\cdot\int_{\Gamma}\bra{\bl{n}(\bl{y})\times\bl{b}_1\bra{\bl{y}}}\times\nabla G\bra{\bl{x}-\bl{y}}\dif\sigma\bra{\bl{y}}\nonumber
	\end{align}
\begin{align}
	\mathcal{L}_{\mathsf{R}}(\vec{\bl{b}})\bra{\bl{x}}\cdot \vec{\bl{E}}_k&=\int_{\Gamma}b_0\bra{\bl{y}}\bra{\nabla G\bra{\bl{x}-\bl{y}}\times\bl{E}_{\nu\bra{k}}}\cdot\bl{n}(\bl{y})\dif\sigma\bra{\bl{y}}\\
	&\qquad\qquad + \int_{\Gamma}\mathcal{b}_2\bra{\bl{y}}\partial_jG\bra{\bl{x}-\bl{y}}\dif\sigma\bra{\bl{y}}\nonumber\\
	&=\,\bl{E}_{\nu\bra{k}}\cdot\int_{\Gamma}b_0\bra{\bl{y}}\bl{n}(\bl{y})\times\nabla G\bra{\bl{x}-\bl{y}}\dif\sigma\bra{\bl{y}}\nonumber\\
	&\qquad\qquad  + \int_{\Gamma}\mathcal{b}_2\bra{\bl{y}}\partial_jG\bra{\bl{x}-\bl{y}}\dif\sigma\bra{\bl{y}}\nonumber\\
	\mathcal{L}_{\mathsf{R}}(\vec{\bl{b}})\bra{\bl{x}}\cdot\vec{\bl{E}}_8 &=-\int_{\Gamma}\bl{b}_1\bra{\bl{y}}\cdot\nabla G\bra{\bl{x}-\bl{y}}\times\bl{n}(\bl{y})\dif\sigma\bra{\bl{y}}\\
	&=-\int_{\Gamma}\nabla G\bra{\bl{x}-\bl{y}}\cdot\bl{n}(\bl{y})\times\bl{b}_1\bra{\bl{y}}\dif\sigma\bra{\bl{y}}\nonumber
	\end{align}
\end{subequations}
so that we have
\begin{greyFrame}
	\begin{equation}\label{eq: vee potential}
	\mathcal{L}_{\mathsf{R}}\bra{\vec{\bl{b}}} = 
	\begin{pmatrix}
	\text{div}\bl{\Upsilon}\bra{b_0}\\
	\bl{curl}\bl{\Psi}\bra{\bl{b}_1\times\bl{n}}\\
	-\,\bl{curl}\bl{\Upsilon}\bra{b_0} + \nabla\psi\bra{\mathcal{b}_2}\\
	\text{div}\bl{\Psi}\bra{\bl{b}_1\times\bl{n}}
	\end{pmatrix},\quad \text{pointwise on }\mathbb{R}^3\backslash\Gamma.
	\end{equation}
\end{greyFrame}

\subsection{Mapping properties of the surface potentials}
Fortunately, we already know a lot about each potential entering \cref{eq: wedge potential} and \cref{eq: vee potential}.
\begin{lemma}\label{lem: mapping properties potentials}
	The potentials  $\mathcal{L}_{\mathsf{T}}:\mathcal{H}_{\mathsf{R}}\rightarrow\bl{H}(\mathsf{D},\mathbb{R}^3\backslash\Gamma)$ and $\mathcal{L}_{\mathsf{R}}:\mathcal{H}_{\mathsf{T}}\rightarrow\bl{H}(\mathsf{D},\mathbb{R}^3\backslash\Gamma)$ explicitly given by \cref{eq: wedge potential} and \cref{eq: vee potential}  are continuous.
\end{lemma}
\begin{proof}
	Recall that if $\bl{b}_1\in\mathbf{H}^{-1/2}\bra{\text{curl}_{\Gamma},\Gamma}$, then $\bl{n}\times \bl{b}_1 \in\mathbf{H}^{-1/2}\bra{\text{div}_{\Gamma},\Gamma}$. So the proof simply boils down to extracting from the discussion of Section 5 in \cite{claeys2017first} the mapping properties
	\begin{subequations}
		\begin{align}
		\nabla\psi:& \,H^{-1/2}\bra{\Gamma}\rightarrow\mathbf{H}_{\text{loc}}(\mathbf{curl}^2,\mathbb{R}^3\backslash \Gamma)\cap \mathbf{H}_{\text{loc}}(\nabla\text{div},\mathbb{R}^3\backslash \Gamma),\\
		\bl{\text{div}}\bl{\Upsilon}:&\,H^{1/2}\rightarrow H_{\text{loc}}^1(\Delta,\mathbb{R}^3\backslash \Gamma),\\
		\bl{curl}\,\bl{\Upsilon}:&\,\mathbf{H}^{-1/2}\bra{\text{div}_{\Gamma},\Gamma}\rightarrow \mathbf{H}_{\text{loc}}(\bl{curl},\mathbb{R}^3\backslash \Gamma),\\
		\text{div}\,\bl{\Psi}:&\,\mathbf{H}^{-1/2}\bra{\text{div}_{\Gamma},\Gamma}\rightarrow H_{\text{loc}}^1(\mathbb{R}^3\backslash \Gamma),\\
		\bl{curl}\,\bl{\Psi}:&\,\mathbf{H}^{-1/2}(\text{div}_{\Gamma},\Gamma)\rightarrow \mathbf{H}_{\text{loc}}(\bl{curl},\mathbb{R}^3\backslash \Gamma).
		\end{align}
	\end{subequations}
	Since $\text{div}\circ\bl{curl}\equiv 0$, we have in particular
	\begin{subequations}
		\begin{align}
		\bl{curl}\,\bl{\Upsilon}:&\,\mathbf{H}^{-1/2}(\text{div}_{\Gamma},\Gamma)\rightarrow \mathbf{H}_{\text{loc}}(\bl{curl},\mathbb{R}^3\backslash \Gamma)\cap\mathbf{H}_{\text{loc}}(\text{div},\mathbb{R}^3\backslash \Gamma),\\
		\bl{curl}\,\bl{\Psi}:&\,\mathbf{H}^{-1/2}(\text{div}_{\Gamma},\Gamma)\rightarrow \mathbf{H}_{\text{loc}}(\bl{curl},\mathbb{R}^3\backslash \Gamma)\cap\mathbf{H}_{\text{loc}}(\text{div},\mathbb{R}^3\backslash \Gamma).
		\end{align}	
	\end{subequations}

Now, for $\mathbf{z}\neq\bl{0}$, the kernels of the two surface potentials decay as 
\begin{equation*}
\norm{\nabla G\bra{\bl{z}}}\lesssim\, \norm{\bl{z}}^{-2},
\end{equation*}
thus are not only square-integrable locally, but in fact belong to $(L^2(\mathbb{R}^3\backslash\Gamma))^8$. 
\end{proof}

The next lemma shows that the surface potentials solve the homogeneous Dirac equation.

\begin{lemma}\label{lem: Dirac of potential vanish}
	For all $\vec{\bl{b}}\in \mathcal{H}_{\mathsf{T}}$ and $\vec{\bl{a}}\in \mathcal{H}_{\mathsf{R}}$, it holds on $\mathbb{R}^3\backslash\Gamma$ that 
	\begin{subequations}
		\begin{align}
		&\mathsf{D}\mathcal{L}_{\mathsf{R}}(\vec{\bl{b}}) \equiv \vec{\bl{0}},\\ &\mathsf{D}\mathcal{L}_{\mathsf{T}}\bra{\vec{\bl{a}}} \equiv \vec{\bl{0}}.
		\end{align}
	\end{subequations}
\end{lemma}
\begin{proof}
	The well-known vector and scalar potentials of \eqref{eq: known scalar potentials} are harmonic. Hence, since $\text{div}\circ\bl{curl}\equiv 0$ and $\bl{curl}\circ\nabla\equiv 0$, we directly evaluate 
	\begin{align}
	\begin{split}
	\mathsf{D}\mathcal{L}_{\mathsf{T}}\bra{\vec{\bl{a}}} &=
	\begin{pmatrix}
	-\text{div}\,\nabla\psi\bra{\mathsf{a}_0} - \text{div}\,\bl{curl}\bl{\Upsilon}\bra{a_2}\\
	-\nabla\text{div}\,\bl{\Psi}\bra{\bl{a}_1} +\bl{curl}\,\bl{curl}\bl{\Psi}\bra{\bl{a}_1}\\
	\bl{curl}\,\nabla\psi\bra{\mathsf{a}_0} +\bl{curl}\,\bl{curl}\bl{\Upsilon}\bra{a_2} - \nabla\text{div}\bl{\Upsilon}\bra{a_2}\\
	\text{div}\,\bl{curl}\bl{\Psi}\bra{\bl{a}_1}
	\end{pmatrix}\\
	&=
	\begin{pmatrix}
	-\Delta\psi\bra{\mathsf{a}_0} \\
	-\nabla\text{div}\,\bl{\Psi}\bra{\bl{a}_1}+\bl{curl}\,\bl{curl}\bl{\Psi}\bra{\bl{a}_1}\\
	-\nabla\,\text{div}\,\bl{\Upsilon}\bra{a_2} + \bl{curl}\,\bl{curl}\,\bl{\Upsilon}\bra{a_2}\\
	0
	\end{pmatrix}
	=\vec{\bl{0}}.
	\end{split}
	\end{align}
	A similar calculation holds for $\mathsf{D}\mathcal{L}_{\mathsf{R}}(\vec{\bl{b}})$.
\end{proof}
\begin{remark}
	\Cref{lem: Dirac of potential vanish} was proved using the explicit representations \eqref{eq: wedge potential} and \eqref{eq: vee potential}. The technique revealed some structure behind the two boundary potentials. However, notice that adapting the argument found in the proof of \cite[Thm. 3.1.6]{sauter2010boundary}, the desired result could also be obtained by observing that
	\begin{equation}\gamma_{\mathsf{T}}':\mathcal{H}_{\mathsf{R}}\rightarrow (\bl{H}_{\text{loc}}(\mathsf{D},\mathbb{R}^3\backslash\Gamma))'\subset (C^{\infty}(\mathbb{R}^3\backslash\Gamma)^8)',
	\end{equation}
	together with \Cref{prop: fundamental sol}, guarantees the equality $\mathsf{D}\mathcal{L}_{\mathsf{T}}\vec{\bl{a}}=\gamma_{\mathsf{T}}'\vec{\bl{a}}$ as continuous linear functionals on $C_0^{\infty}(\mathbb{R}^3\backslash\Gamma)$.
\end{remark}
\begin{remark}
	It is a nice and unusual property for the potentials to belong to $(L^2(\Omega^+))^8$ as opposed to being only locally square-integrable. We see from \Cref{lem: int rep of boundary potential} that this is a consequence of two ingredients: the stronger singularity of the Dirac fundamental solution, combined with the absence of differential operators acting on the relevant traces.
\end{remark}

\begin{lemma}[{Jump relations}]\label{lem: jump relations}
	For all $\vec{\bl{a}}\in\mathcal{H}_{\mathsf{R}}$ and $\vec{\bl{b}}\in\mathcal{H}_{\mathsf{T}}$,
		\begin{align}
		\left[\gamma_{\mathsf{T}}\right]\mathcal{L}_{\mathsf{T}}(\vec{\bl{a}})&=\vec{\bm{0}}, & \left[\gamma_{\mathsf{R}}\right]\mathcal{L}_{\mathsf{T}}(\vec{\bl{a}})&=\id,\\
		\left[\gamma_{\mathsf{T}}\right]\mathcal{L}_{\mathsf{R}}(\vec{\bl{b}})&= \id, & \left[\gamma_{\mathsf{R}}\right]\mathcal{L}_{\mathsf{R}}(\vec{\bl{b}})&=\vec{\bm{0}}.
		\end{align}
\end{lemma}
\begin{proof}
	For the most part, the following jump relations can be inferred from known theory. We carefully evaluate
	\begin{subequations}
		\begin{align}
		\left[\gamma_{\mathsf{T}}\right]\mathcal{L}_{\mathsf{T}}(\vec{\bl{a}}) &=
		\begin{pmatrix}
		-\left[\gamma\right]\text{div}\bl{\Psi}\bra{\bl{a}_1}\\
		\left[\gamma_t\right]\nabla\psi\bra{\mathcal{a}_0}+\left[\gamma_t\right]\bl{curl}\bl{\Upsilon}\bra{a_2}\\
		\left[\gamma_n\right]\bl{curl}\bl{\Psi}\bra{\bl{a}_1}
		\end{pmatrix}
		=\begin{pmatrix}
		0\\
		\bl{0}\\
		0
		\end{pmatrix},\label{eq: jump wedge wedge}\\
		\label{eq: jump vee wedge}
		\left[\gamma_{\mathsf{R}}\right]\mathcal{L}_{\mathsf{T}}(\vec{\bl{a}}) &=
		\begin{pmatrix}
		\left[\gamma_n\right]\nabla\psi\bra{\mathsf{a}_0}+\left[\gamma_n\right]\bl{curl}\bl{\Upsilon}\bra{a_2}\\
		\left[\gamma_{\tau}\right]\bl{curl}\bl{\Psi}\bra{\bl{a}_1}\\
		\left[\gamma\right]\text{div}\bl{\Upsilon}\bra{a_2}
		\end{pmatrix}
		=\begin{pmatrix}
		\mathcal{a}_0\\
		\bl{a}_1\\
		a_2
		\end{pmatrix},\\
		\left[\gamma_{\mathsf{T}}\right]\mathcal{L}_{\mathsf{R}}(\vec{\bl{b}})&=
		\begin{pmatrix}
		\left[\gamma\right]\text{div}\bl{\Upsilon}\bra{b_0}\\
		\left[\gamma_{t}\right]\bl{curl}\bl{\Psi}\bra{\bl{b}_1\times\bl{n}}\\
		-\left[\gamma_{n}\right]\bl{curl}\bl{\Upsilon}\bra{b_0}+\left[\gamma_{n}\right]\nabla\psi\bra{\mathcal{b}_2}
		\end{pmatrix}
		=\begin{pmatrix}
		b_0\\
		\bl{b}_1\\
		\mathcal{b}_2
		\end{pmatrix},\label{eq: jump wedge vee}\\
		\left[\gamma_{\mathsf{R}}\right]\mathcal{L}_{\mathsf{R}}(\vec{\bl{b}}) &=
		\begin{pmatrix}
		\left[\gamma_n\right]\bl{curl}\bl{\Psi}\bra{\bl{b}_1\times\bl{n}}\\
		-\left[\gamma_{\tau}\right]\bl{curl}\bl{\Upsilon}\bra{b_0}+\left[\gamma_{\tau}\right]\nabla\psi\bra{\mathcal{b}_2}\\
		\left[\gamma\right]\text{div}\Psi\bra{\bl{b}_1\times\bl{n}}
		\end{pmatrix}
		=\begin{pmatrix}
		0\\
		\bl{0}\\
		0
		\end{pmatrix}.\label{eq: jump vee vee}
		\end{align}
	\end{subequations}
	
	The individual terms appearing in the above calculations can be found in \cite[Sec. 5]{claeys2017first} and \cite[Sec. 4]{hiptmair2003coupling}, possibly up to tangential rotation by $90^{\circ}$. Some terms slightly differ. In both \cref{eq: jump vee wedge} and \cref{eq: jump wedge vee}, we are particularly concerned with the normal jump of $\mathbf{curl}\,\bl{\Upsilon}$ across $\Gamma$. Fortunately, we know that the restriction of $\bl{\Upsilon}$ to $H^{1/2}\left(\Gamma\right)$ is a continuous map with codomain $\mathbf{H}_{\text{loc}}(\mathbf{curl}^2,\Omega)$. Its image is therefore regular enough for the identity
	\begin{equation*}
	\left[\gamma_n\right]\mathbf{curl}\,\bl{\Upsilon}\left(q\right)=\text{div}_{\Gamma}\left(\left[\gamma_{\tau}\right]\bl{\Upsilon}\left(q\right)\right) = 0
	\end{equation*}
	to hold for all $q\in H^{1/2}(\Gamma)$ \cite[Eq. 8]{buffa2003galerkin}.
\end{proof}

\begin{remark}
	The formal structure of these jump relations is the same as that of the jump identities for the potentials associated with other operators such as
	\begin{itemize}
		\item[$\diamond$] scalar second-order strongly elliptic operators \cite{mclean2000strongly,sauter2010boundary},
		\item[$\diamond$] second-order Maxwell wave operators \cite{buffa2002maxwell,buffa2003galerkin},
		\item[$\diamond$] Hodge--Laplace and Hodge--Helmholtz operators \cite{claeys2017first,claeys2018first}.
	\end{itemize}
\end{remark}

\subsection{Representation by surface potentials}\label{sec: Representation by surface potentials}
Following McLean in \cite[Chap. 7]{mclean2000strongly}, we mimic the approach introduced by Costabel and Dauge \cite{costabel1997representation,costabel1988boundary}. \Cref{cor. second green identity} plays the role of Green's second identity. We begin with the case where a solution of the Dirac equation defines a compactly supported distribution. This covers for instance interior problems and yields a representation formula in $\Omega^-$. However, a condition on the behavior of solutions at infinity will be needed for $\Omega^+$. 

\begin{proposition}[Interior representation formula]\label{prop: rep formula for compactly supported}
	If $\vec{\bl{U}}\in\bl{H}(\mathsf{D},\mathbb{R}^3\backslash\Gamma)$ is compactly supported and $\vec{\bl{F}}\in(L^2(\mathbb{R}^3))^8$ is such that $\vec{\bl{F}}\big\vert_{\Omega}:=\bra{\mathsf{D}\bl{U}}\big\vert_{\Omega}$ and $\vec{\bl{F}}\big\vert_{\Omega^+}:=\bra{\mathsf{D}\bl{U}}\big\vert_{\Omega^+}$. Then
	\begin{equation}\label{eq: representation formula compact}
	\vec{\bl{U}}(\bl{x}) = \Phi * \vec{\bl{F}}(\bl{x}) + \mathcal{L}_{\mathsf{T}}\left[\gamma_{\mathsf{R}}\vec{\bl{U}}\right](\bl{x}) +\mathcal{L}_{\mathsf{R}}\left[\gamma_{\mathsf{T}}\vec{\bl{U}}\right](\bl{x}),\qquad\qquad \bl{x}\in\mathbb{R}^3\backslash\Gamma.
	\end{equation}
\end{proposition}
\begin{proof}
	According to \cref{eq: IBP Dirac}, 
	\begin{align}
	\begin{split}
	\langle\mathsf{D}\vec{\bl{U}},\vec{\bl{V}}\rangle_{\mathbb{R}^3}
	&\starequal \int_{\Omega}\vec{\bl{U}}\cdot\mathsf{D}\vec{\bl{V}}\dif\bl{x} + \int_{\mathbb{R}^3\backslash\overline{\Omega}}\vec{\bl{U}}\cdot\mathsf{D}\vec{\bl{V}}\dif\bl{x}\\
	&=\int_{\Omega}\vec{\bl{F}}\cdot\vec{\bl{V}}\dif\bl{x} + \llangle\gamma^-_{\mathsf{R}}\vec{\bl{U}},\gamma^-_{\mathsf{T}}\vec{\bl{V}}\rrangle_{\Gamma} -\llangle\gamma^-_{\mathsf{T}}\vec{\bl{U}},\gamma^-_{\mathsf{R}}\vec{\bl{V}}\rrangle_{\Gamma} \\
	&\qquad + \int_{\mathbb{R}^3\backslash\overline{\Omega}}\vec{\bl{F}}\cdot\vec{\bl{V}}\dif\bl{x} - \llangle\gamma^+_{\mathsf{R}}\vec{\bl{U}},\gamma^+_{\mathsf{T}}\vec{\bl{V}}\rrangle_{\Gamma}  +\llangle\gamma^+_{\mathsf{T}}\vec{\bl{U}},\gamma^+_{\mathsf{R}}\vec{\bl{V}}\rrangle_{\Gamma} \\
	&=\int_{\mathbb{R}^3}\vec{\bl{F}}\cdot\vec{\bl{V}}\dif\bl{x} + 
	\llangle\left[\gamma_{\mathsf{R}}\vec{\bl{U}}\right],\gamma_{\mathsf{T}}\vec{\bl{V}}\rrangle_{\Gamma}  -\llangle\left[\gamma_{\mathsf{T}}\vec{\bl{U}}\right],\gamma_{\mathsf{R}}\vec{\bl{V}}\rrangle_{\Gamma} 
	\end{split}
	\end{align}
	for all $\vec{\bl{V}}\in (C_0^{\infty}(\mathbb{R}^3))^8$. The regularity assumptions on $\vec{\bl{U}}$ guarantee that the traces are well-defined.  We have used the fact that $\vec{\bl{V}}$ is smooth across the boundary to obtain the last equality, because smoothness guarantees that $\gamma^-_{\mathsf{T}}\vec{\bl{V}}=\gamma^+_{\mathsf{T}}\vec{\bl{V}}$ and $\gamma^-_{\mathsf{R}}\vec{\bl{V}}=\gamma^+_{\mathsf{R}}\vec{\bl{V}}$.
	Therefore, in the sense of distributions, we have
	\begin{equation}
	\mathsf{D}\vec{\bl{U}} = \bl{F} + \bra{\gamma^-_{\mathsf{T}}}'\left[\gamma^-_{\mathsf{R}}\vec{\bl{U}}\right] - \bra{\gamma^-_{\mathsf{R}}}'\left[\gamma^-_{\mathsf{T}}\vec{\bl{U}}\right].
	\end{equation}
	
	Since $\vec{\bl{U}}$ is assumed to have compact support, it can interpreted as a continuous linear functional on $C^{\infty}(\mathbb{R}^3)^8$ and convolution with $\Phi$ using \Cref{prop: fundamental sol} shows that the identity is valid when interpreted in the sense distributions. \Cref{lem: mapping properties potentials} confirms that the equality holds in $(L^2(\mathbb{R}^3))^8$.
\end{proof}

In the following, we will work over the domains defined as the interior $B_{\rho}$ and exterior $B_{\rho}^+$ of an open ball of radius $\rho$. Therefore, we must introduce the traces $\gamma^\rho_\mathsf{T}$ and $\gamma^\rho_\mathsf{R}$ that extend the operators defined in \eqref{ed: traces smooth} where $\Gamma$ is replaced by the boundary $\partial B_{\rho}$ of the open ball. The surface potentials $\mathcal{L}^{\rho}_{\mathsf{R}}$ and $\mathcal{L}^{\rho}_\mathsf{T}$ are defined accordingly with respect to these trace mappings. Similarly, a dagger $\dag$ will refer to any given Lipschitz domain $\Omega_\dag\subset\mathbb{R}^3$. The following development parallels that of \cite[Sec. 7]{mclean2000strongly}.

\begin{lemma}\label{lem: harmonic contribution to rep formula}
	For $\vec{\bl{U}}\in (C^{\infty}_0(\Omega^+)^8)'$ such that $\mathsf{D}\vec{\bl{U}}$ has compact support  in $\Omega^+$, there exists a \underline{unique} vector field $\mathsf{M}\vec{\bl{U}}\in(C^{\infty}(\mathbb{R}^3))^8$ such that 
	\begin{equation}\label{eq: MU def}
	\mathsf{M}\vec{\bl{U}}\bra{\bl{x}} = \mathcal{L}^{\dag}_{\mathsf{T}}\bra{\gamma_{\mathsf{R}}^{\dag}\vec{\bl{U}}}\bra{\bl{x}} +\mathcal{L}_{\mathsf{R}}^{\dag}\bra{\gamma_{\mathsf{T}}^{\dag}\vec{\bl{U}}}\bra{\bl{x}}
	\end{equation}
	for all $\bl{x}$ inside \underline{any} bounded Lipschitz domain $\Omega_{\dag}$ such that
	\begin{equation}
	\overline{\Omega}\cup\text{\emph{supp}}(\mathsf{D}\vec{\bl{U}})\Subset \Omega_{\dag}. 
	\end{equation}
\end{lemma}
\begin{remark}
	It is key in the statement of \Cref{lem: harmonic contribution to rep formula} that the vector field $\mathsf{M}\vec{\bl{U}}$ is independent of $\Omega_{\dag}$.
\end{remark}
\begin{proof}
	Under the above hypotheses, $\vec{\bl{U}}$ is harmonic in $\Omega^{+}\backslash\text{supp}(\mathsf{D}\vec{\bl{U}})$, because $\mathsf{D}\vec{\bl{U}}=\vec{\bl{0}}$ implies that $\vec{\bl{\Delta}}\vec{\bl{U}}=\mathsf{D}^2\vec{\bl{U}}=\vec{\bl{0}}$.  Standard elliptic regularity theory \cite[Thm. 6.4]{mclean2000strongly} further tells us that $\vec{\bl{U}}$ is a regular distribution whose components are smooth in that domain. Therefore, we can \emph{define} $\mathsf{M}\vec{\bl{U}}$ in $B_{\rho_1}$ as in the right hand side of \eqref{eq: MU def} by
	\begin{equation}
	\mathsf{M}\vec{\bl{U}}\bra{\bl{x}} := \mathcal{L}^{\rho_1}_{\mathsf{T}}\bra{\gamma_{\mathsf{R}}^{\rho_1}\vec{\bl{U}}}\bra{\bl{x}} +\mathcal{L}_{\mathsf{R}}^{\rho_1}\bra{\gamma_{\mathsf{T}}^{\rho_1}\vec{\bl{U}}}\bra{\bl{x}},
	\end{equation}
	where the radius $\rho_1$ is large enough that $	\overline{\Omega}\cup\text{supp}(\mathsf{D}\vec{\bl{U}})\Subset B_{\rho_1}$.
	
	Applying \cref{eq: IBP Dirac} inside $B_{\rho_2}\backslash\overline B_{\rho_1}$ with $\rho_1<\rho_2$ eventually shows that this definition is independent of the radius. Indeed, for any $\bl{x}\in B_{\rho_1}$, $\Phi_{\bl{x}}$ is a smooth matrix in $\mathbb{R}^3\backslash \overline{B_{\rho_1}}$, and, thus, $\text{supp}(\mathsf{D}\vec{\bl{U}})\Subset B_{\rho_1}$ guarantees for $i=1,...,8$ that
	\begin{multline}
	0 = \int_{B_{\rho_2}\backslash\overline B_{\rho_1}}\Phi\bra{\bl{x}-\bl{y}}\mathsf{D}\vec{\bl{U}}\bra{\bl{y}}\dif\bl{y}\cdot\vec{\bl{E}}_i = \int_{B_{\rho_2}\backslash\overline B_{\rho_1}}\Phi_{i,:}\bra{\bl{x}-\bl{y}}\mathsf{D}\vec{\bl{U}}\bra{\bl{y}}\dif\bl{y}\\\starequal
	-\int_{B_{\rho_2}\backslash\overline B_{\rho_1}}\Phi_{:,i}\bra{\bl{x}-\bl{y}}\cdot\mathsf{D}\vec{\bl{U}}\bra{\bl{y}}\dif\bl{y}
	= -\int_{B_{\rho_2}\backslash\overline B_{\rho_1}}\mathsf{D}_{\bl{y}}\Phi_{:,i}\bra{\bl{x}-\bl{y}}\cdot\vec{\bl{U}}\bra{\bl{y}}\dif\bl{y}\\
	- \llangle\gamma^{\rho_2}_{\mathsf{T}}\vec{\bl{U}},\gamma^{\rho_2}_{\mathsf{R}}\Phi_{:,i}\bra{\bl{x}-\cdot}\rrangle_{\Gamma}  -
	\llangle\gamma^{\rho_2}_{\mathsf{T}}\Phi_{:,i}\bra{\bl{x}-\cdot},\gamma^{\rho_2}_{\mathsf{R}}\vec{\bl{U}}\rrangle_{\Gamma}  \\
	+ \llangle\gamma^{\rho_1}_{\mathsf{T}}\vec{\bl{U}},\gamma^{\rho_1}_{\mathsf{R}}\Phi_{:,i}\bra{\bl{x}-\cdot}\rrangle_{\Gamma}  +
	\llangle\gamma^{\rho_1}_{\mathsf{T}}\Phi_{:,i}\bra{\bl{x}-\cdot},\gamma^{\rho_1}_{\mathsf{R}}\vec{\bl{U}}\rrangle_{\Gamma},
	\end{multline}
	where $\Phi_{i,:}$ corresponds to the $i$-th row of $\Phi$, $\Phi_{:,j}$ to its $j$-th column, and \Cref{lem: anti-symmetric properties of fundamental sol} was used to obtain $(*)$.
	
	On the one hand, for $\bl{x}\neq\bl{y}$,
	\begin{multline}
	\mathsf{D}_{\bl{y}}\Phi_{:,i}\bra{\bl{x}-\bl{y}}\cdot\vec{\bl{U}}\bra{\bl{y}}
	=
	\mathsf{D}_{\bl{x}}\bra{\Phi_{\bl{x}}\bra{\bl{y}}\vec{\bl{U}}\bra{\bl{y}}}\cdot\vec{\bl{E}}_i\\= \mathsf{D}_{\bl{x}}\mathsf{D}_{\bl{x}}\bra{\mathsf{G}\bra{\bl{x}-\bl{y}}\vec{\bl{U}}\bra{\bl{y}}}\cdot\vec{\bl{E}}_i
	=\bra{-\Delta_{\bl{x}} G\bra{\bl{x}-\bl{y}}}\vec{\bl{U}}\bra{\bl{y}}\cdot\vec{\bl{E}}_i = 0.
	\end{multline}
	
  On the other hand,
	\begin{equation}
	\llangle\gamma^{\rho_2}_{\mathsf{T}}\vec{\bl{U}},\gamma^{\rho_2}_{\mathsf{R}}\Phi_{:,i}\bra{\bl{x}-\cdot}\rrangle = -\llangle\gamma^{\rho_2}_{\mathsf{T}}\vec{\bl{U}},\gamma^{\rho_2}_{\mathsf{R}}\bra{\Phi_{\bl{x}}\vec{\bl{E}}_i}\rrangle = - \mathcal{L}^{\rho_2}_{\mathsf{R}}\bra{\gamma^{\rho_2}_{\mathsf{T}}\vec{\bl{U}}}\bra{\bl{x}}\cdot \vec{\bl{E}}_j
	\end{equation}
	by \Cref{lem: int rep of boundary potential}, and similarly for the remaining boundary terms. These two pieces of information together prove the validity of the independence claim.

	In fact, the same argument can be repeated in $B_{\rho_1}\backslash\overline{\Omega_{\dag}}$ to confirm that \eqref{eq: MU def} holds independently of the chosen Lipschitz domain satisfying the hypotheses.
	
	Smoothness of $\mathsf{M}\vec{\bl{U}}$ is inherited from the smoothness of the integrands.
\end{proof}

\begin{lemma}\label{lem: rep formula exterior domain}
	Let $\vec{\bl{F}}\in(L^2(\Omega^+))^8$ be compactly supported and suppose that $\vec{\bl{U}}\in (C^{\infty}_0(\Omega^+)^8)'$  satisfies $\mathsf{D}\vec{\bl{U}}=\vec{\bl{F}}$ on $\Omega^+$. If the restriction of $\vec{\bl{U}}$ to $\Omega^+\cap B_{\rho}$ belongs to $\bl{H}(\mathsf{D},\Omega^+\cap B_{\rho})$ for some $\rho$ large enough that $\Omega\cup\Gamma\Subset B_{\rho}$ and $\text{\emph{supp}}\,\vec{\bl{F}}\Subset\Omega^+\cap B_{\rho}$, then
	\begin{equation}\label{eq: rep formula exterior domain with harmonic term}
	\vec{\bl{U}} = \Phi * \vec{\bl{F}} - \mathcal{L}_{\mathsf{T}}\gamma^+_{\mathsf{R}}\vec{\bl{U}} -\mathcal{L}_{\mathsf{R}}\gamma^+_{\mathsf{T}}\vec{\bl{U}}
	+ \mathsf{M}\vec{\bl{U}}
	\end{equation}
	holds in $\mathbf{H}(\mathsf{D},\Omega^+)$.
\end{lemma}
\begin{proof}
	Upon applying \Cref{prop: rep formula for compactly supported} to the distribution
	\begin{equation}
	\vec{\bl{U}}_0: = 
	\begin{cases}
	\vec{\bl{0}}, & \text{in }\Omega,\\
	\vec{\bl{U}}, & \text{in }\Omega^+\cap B_{\rho},\\
	\vec{\bl{0}}, & \text{in }\mathbb{R}^3\backslash\overline{B_{\rho}},
	\end{cases}
	\end{equation}
	that is compactly supported and belongs to $\bl{H}_{\text{loc}}\bra{\mathsf{D},\mathbb{R}^3\backslash\bra{\Gamma\cup\partial B_{\rho}}}$, we obtain
	\begin{equation}
	\vec{\bl{U}}_0 = \Phi * \vec{\bl{F}} - \mathcal{L}_{\mathsf{T}}\bra{\gamma^+_{\mathsf{R}}\vec{\bl{U}}} -\mathcal{L}_{\mathsf{R}}\bra{\gamma^+_{\mathsf{T}}\vec{\bl{U}}} +\mathcal{L}^{\rho}_{\mathsf{T}}\bra{\gamma_{\mathsf{R}}^{\rho}\vec{\bl{U}}} +\mathcal{L}_{\mathsf{R}}^{\rho}\bra{\gamma_{\mathsf{T}}^{\rho}\vec{\bl{U}}}
	\end{equation}
	as a functional on $(C^{\infty}_0(\mathbb{R}^3))^8$. Since $B_{\rho}$ satisfies the hypotheses imposed on $\Omega_{\dag}$ in the statement of \Cref{lem: harmonic contribution to rep formula}, we recognize that
	\begin{equation}
	\mathcal{L}^{\rho}_{\mathsf{T}}\bra{\gamma_{\mathsf{R}}^{\rho}\vec{\bl{U}}}\bra{\bl{x}} +\mathcal{L}_{\mathsf{R}}^{\rho}\bra{\gamma_{\mathsf{T}}^{\rho}\vec{\bl{U}}}\bra{\bl{x}} = \mathsf{M}\vec{\bl{U}}\bra{\bl{x}}
	\end{equation}
	for all $\bl{x}\in B_{\rho}$. Hence,
	\begin{equation}\label{eq: rep formula exterior domain with harmonic term inside proof}
	\vec{\bl{U}} = \Phi * \vec{\bl{F}} - \mathcal{L}_{\mathsf{T}}\gamma^+_{\mathsf{R}}\vec{\bl{U}} -\mathcal{L}_{\mathsf{R}}\gamma^+_{\mathsf{T}}\vec{\bl{U}} +\mathsf{M}\vec{\bl{U}}\qquad\qquad\text{in }\Omega^+\cap B_{\rho}.
	\end{equation}
	
	As in \Cref{lem: harmonic contribution to rep formula}, it follows from $\text{supp}\,\vec{\bl{F}}\subset B_{\rho}$ that $\vec{\bl{U}}$ is harmonic in $\mathbb{R}^3\backslash B_{\rho}$, and thus smooth everywhere outside the ball $B_{\rho}$ by well-known elliptic regularity theory \cite[Thm. 6.4]{mclean2000strongly}. Hence, the hypothesis that $\vec{\bl{U}}\in\bl{H}(\mathsf{D},\Omega^+\cap B_{\rho})$ for at least one ball $B_{\rho}$ satisfying the hypotheses in fact guarantees that it belongs to that space independently of the radius satisfying those same requirements. Therefore, \eqref{lem: rep formula exterior domain} holds in the whole of $\Omega^+$. Based on \Cref{lem: harmonic contribution to rep formula}, the mapping properties of the potentials established in \Cref{lem: mapping properties potentials} and \Cref{prop: fundamental sol}, we conclude that the equality \eqref{eq: rep formula exterior domain with harmonic term inside proof} holds in fact not only in $\mathbf{H}_{\text{loc}}(\mathsf{D},\Omega^+)$, but in $\mathbf{H}(\mathsf{D},\Omega^+)$---which is the desired result.
\end{proof}

\begin{lemma}\label{lem: decay condition M=0}
	Under the hypotheses of \Cref{lem: rep formula exterior domain},
	\begin{equation}
	\mathsf{M}\vec{\bl{U}}=\vec{\bl{0}}
	\end{equation}
	if and only if 
	\begin{equation}\label{eq: decay condition lemma}
	\norm{\vec{\bl{U}}\bra{\bl{z}}}\rightarrow 0\text{ uniformly as }\bl{z}\rightarrow \infty.
	\end{equation}
\end{lemma}
\begin{proof}
	The condition \eqref{eq: decay condition lemma} is well-defined, because as in \Cref{lem: rep formula exterior domain}, there exists a radius $\rho_1$ large enough that the vector-field $\vec{\bl{U}}$ is smooth outside $B_{\rho_1}$. For the same reason, the traces of $\vec{\bl{U}}$ appearing in the following inequalities are smooth boundary fields. 
	
	Recall that for $\bl{z}\neq\bl{0}$,
	\begin{equation}\label{eq: asymptotic bound on fundamental sol}
	\norm{\nabla G\bra{\bl{z}}}\lesssim\, \norm{\bl{z}}^{-2}.
	\end{equation}
	Therefore, it is easily seen from \eqref{eq: wedge potential} and \eqref{eq: vee potential} that if $\rho_2>\rho_1$, 
	\begin{equation}\label{eq: bound on norm of potential}
	\norm{\mathcal{L}^{\rho_2}_{\bullet}\bra{\gamma^{\rho_2}_{\bullet}\vec{\bl{U}}}\bra{\bl{x}}}\lesssim\, \rho_2^{-2}\norm{\int_{\partial B_{\rho_2}}\gamma_{\bullet}\vec{\bl{U}}\bra{\bl{y}}\dif\sigma\bra{\bl{y}}}\lesssim\, \max_{\bl{y}\in\partial B_{\rho_2}}\norm{\vec{\bl{U}}\bra{\bl{y}}}
	\end{equation}
	for all $\bl{x}\in B_{\rho_1}$, $\bullet = \mathsf{T}$ or $\mathsf{R}$. Notice that the left hand side of \eqref{eq: bound on norm of potential} is well-defined, because as in \Cref{lem: harmonic contribution to rep formula}, \Cref{lem: Dirac of potential vanish} and $\mathsf{D}^2=-\Delta$ guarantee that away from the boundary $\partial B_{\rho_2}$, the potentials are smooth harmonic vector fields.  No differential operator appears in the definition of the trace mappings $\gamma_{\mathsf{R}}$ and $\gamma_{\mathsf{T}}$. The independence of $\mathsf{M}\vec{\bl{U}}$ from its domain of definition thus directly yields one implication of the lemma upon taking $\rho_2\rightarrow\infty$.
	
	The converse follows from the exterior representation formula \eqref{eq: rep formula exterior domain with harmonic term} with $\mathsf{M}\vec{\bl{U}}=\vec{\bl{0}}$ and an analysis exploiting \eqref{eq: asymptotic bound on fundamental sol} that leads to an inequality similar to \eqref{eq: bound on norm of potential}. However, this time the potentials are computed as integrals (duality pairings) on the fixed boundary $\Gamma$ and an inverse square decay is inherited from the decay of the fundamental solution.
\end{proof}

\begin{proposition}[{Exterior representation formula}]
If $\vec{\bl{U}}\in\bl{H}_{\emph{loc}}(\mathsf{D},\Omega^+)$ is such that $\vec{\bl{U}}\bra{\bl{z}}\rightarrow 0$ as $\bl{z}\rightarrow \infty$ and $\vec{\bl{F}}:=\mathsf{D}\bl{U}$ is compactly supported. Then
\begin{equation}\label{eq: representation formula exterior no harmonic term}
\vec{\bl{U}}(\bl{x}) = \Phi * \vec{\bl{F}}(\bl{x}) - \mathcal{L}_{\mathsf{T}}\gamma_{\mathsf{R}}^+\vec{\bl{U}}(\bl{x}) -\mathcal{L}_{\mathsf{R}}^+\gamma_{\mathsf{T}}\vec{\bl{U}}(\bl{x}),\qquad\qquad \bl{x}\in\Omega^+.
\end{equation}
\end{proposition}

\section{Boundary integral equations}\label{Sec: boundary integral operators}
Boundary integral equations are obtained by taking the traces $\gamma_{\mathsf{R}}$ and $\gamma_{\mathsf{T}}$ on both sides of the representation formulas \eqref{eq: representation formula compact} and \eqref{eq: representation formula exterior no harmonic term}. The operator form of the interior and exterior Calder\'on projectors defined on $\mathcal{H}_{\mathsf{R}}\times\mathcal{H}_{\mathsf{T}}$, which we denote $\mathsf{P}^-$ and $\mathsf{P}^+$ respectively, enter the Calder\'on identities
\begin{equation}\label{eq: interior Calderon identity}
\underbrace{\begin{pmatrix}
	\{\gamma_{\mathsf{R}}\}\,\mathcal{L}_{\mathsf{T}} +\frac{1}{2}\id & \{\gamma_{\mathsf{R}}\}\,\mathcal{L}_{\mathsf{R}} \\
	\{\gamma_{\mathsf{T}}\}\,\mathcal{L}_{\mathsf{T}} & \{\gamma_{\mathsf{T}}\}\,\mathcal{L}_{\mathsf{R}} +\frac{1}{2}\id
	\end{pmatrix}}_{\mathsf{P}^-}
\begin{pmatrix}
\gamma^-_{\mathsf{R}}\bra{\bl{U}}\\
\gamma^-_{\mathsf{T}}\bra{\bl{U}}
\end{pmatrix}
=
\begin{pmatrix}
\gamma^-_{\mathsf{R}}\bra{\bl{U}}\\
\gamma^-_{\mathsf{T}}\bra{\bl{U}}
\end{pmatrix},
\end{equation}
\begin{equation}
\underbrace{\begin{pmatrix}
	-\{\gamma_{\mathsf{R}}\}\,\mathcal{L}_{\mathsf{T}} +\frac{1}{2}\id & -\{\gamma_{\mathsf{R}}\}\,\mathcal{L}_{\mathsf{R}} \\
	-\{\gamma_{\mathsf{T}}\}\,\mathcal{L}_{\mathsf{T}} & -\{\gamma_{\mathsf{T}}\}\,\mathcal{L}_{\mathsf{R}} +\frac{1}{2}\id
	\end{pmatrix}}_{\mathsf{P}^+}
\begin{pmatrix}
\gamma^+_{\mathsf{R}}\bra{\bl{U}}\\
\gamma^+_{\mathsf{T}}\bra{\bl{U}}
\end{pmatrix}
=
\begin{pmatrix}
\gamma^+_{\mathsf{R}}\bra{\bl{U}}\\
\gamma^+_{\mathsf{T}}\bra{\bl{U}}
\end{pmatrix}.
\end{equation}

For example, extend a solution $\vec{\bl{U}}\in\bl{H}(\mathsf{D},\Omega)$ of the homogeneous Dirac equation in $\Omega^-$ to the whole of $\mathbb{R}^3$ by zero. Using \Cref{prop: rep formula for compactly supported},
\begin{equation}
\vec{\bl{U}}(\bl{x}) =  \mathcal{L}_{\mathsf{T}}\gamma_{\mathsf{R}}^-\vec{\bl{U}}(\bl{x}) +\mathcal{L}_{\mathsf{R}}^-\gamma_{\mathsf{T}}\vec{\bl{U}}(\bl{x}),\qquad\qquad \bl{x}\in\mathbb{R}^3\backslash\Gamma.
\end{equation}
Then, applying $\gamma_{\mathsf{R}}^-$ on both sides of the equation yields
\begin{equation}
\gamma_{\mathsf{R}}^-\vec{\bl{U}}(\bl{x}) =  \gamma_{\mathsf{R}}^-\mathcal{L}_{\mathsf{T}}\gamma_{\mathsf{R}}^-\vec{\bl{U}}(\bl{x}) +\gamma_{\mathsf{R}}^-\mathcal{L}_{\mathsf{R}}^-\gamma_{\mathsf{T}}\vec{\bl{U}}(\bl{x}),\qquad\qquad \bl{x}\in\Gamma.
\end{equation}
It is a simple calculation to verify that the jump identities of \Cref{lem: jump relations} implies 
		\begin{align}
\{\gamma_{\mathsf{T}}\}\mathcal{L}_{\mathsf{T}}(\vec{\bl{a}})&=\gamma_{\mathsf{T}}^-\mathcal{L}_{\mathsf{T}}(\vec{\bl{a}}), & \{\gamma_{\mathsf{R}}\}\mathcal{L}_{\mathsf{T}}(\vec{\bl{a}})&=\gamma_{\mathsf{R}}^-\mathcal{L}_{\mathsf{T}}(\vec{\bl{a}})-\frac{1}{2}\vec{\bl{a}},\\
\{\gamma_{\mathsf{T}}\}\mathcal{L}_{\mathsf{R}}(\vec{\bl{b}})&= \gamma_{\mathsf{T}}^-\mathcal{L}_{\mathsf{R}}(\vec{\bl{b}}) - \frac{1}{2}\vec{\bl{b}}, & \{\gamma_{\mathsf{R}}\}\mathcal{L}_{\mathsf{R}}(\vec{\bl{b}})&=\gamma_{\mathsf{R}}^-\mathcal{L}_{\mathsf{R}}(\vec{\bl{b}}).
\end{align}
Substituting the interior traces for the averages using these relations leads to the top row of \eqref{eq: interior Calderon identity}. The other identities are obtained similarly.

A classical argument, cf. \cite[lem. 6.18]{steinbach2007numerical}, shows that $\mathsf{P}^-$ and $\mathsf{P}^+$ are indeed projectors, i.e. $(\mathsf{P}^\mp)^2=\mathsf{P}^\mp$. The proof, which for the homogeneous Dirac equation is essentially based on \Cref{lem: Dirac of potential vanish}, also shows as a byproduct, cf. \cite[Thm. 3.7]{von1989boundary}, that the images of 
$\mathsf{P}^-$ and $\mathsf{P}^+$ are spaces of valid interior and exterior Cauchy data, respectively. In fact, as observed in \cite[Sec. 5]{buffa2003galerkin}, we have $\mathsf{P}^- + \mathsf{P}^+ = \id$. So the range of $\mathsf{P}^-$ coincides with the nullspace of $\mathsf{P}^+$ and vice versa. Therefore, we find the important property that $(\vec{\bm{a}},\vec{\mathbf{b}})\in\mathcal{H}_\mathsf{R}\times\mathcal{H}_\mathsf{T}$ is valid interior or exterior Cauchy data if and only if it lies in the nullspace of $\mathsf{P}^+$ or $\mathsf{P}^-$, respectively.

The two direct boundary integral equations of the first-kind related to \eqref{pb: R} and \eqref{pb: T} then read as follows. Given $\traceR\vecU=\veca\in\HR$, the task is to determine the unknown $\vecb=\traceT\vecU\in\HT$ by solving
\begin{equation}\label{pb: BR}
\color{red}
\traceR\LR(\vecb) = \frac{1}{2}\veca - \{\traceR\}\LT(\veca).\tag{B\textsf{R}}
\end{equation}
If $\vecb\in\HT$ is known instead, then we solve
\begin{equation}\label{pb: BT}
\color{blue}
\traceT\LT(\veca) = \frac{1}{2}\vecb - \{\traceT\}\LR(\vecb)\tag{B\textsf{T}}
\end{equation}
for the unknown $\veca\in\HR$.

\remark[{Duality and symmetry}]\label{sec: Duality and symmetry} Let us revisit the boundary value problems of \Cref{sec: Hodge--Dirac operators in 3D Euclidean space}. We wish to highlight that \eqref{pb: T} and \eqref{pb: R} are really the same problem in hiding. For example, we can always relabel the components of an unknown vector-field $\vecU\in\HD$ to
\begin{equation}\label{eq: iso domain}
    V_0 := U_3,\qquad\bl{V}_1:=-\bl{U}_2\qquad\bl{V}_2:=-\bl{U}_1\qquad\text{and}\qquad V_3:=V_0,
\end{equation}
and set
\begin{equation}\label{eq: iso boundary}
a_0 := - b_2\qquad \bl{a}_1 := \bl{n}\times\bl{b}_1\qquad\text{and}\qquad a_3 = b_0.
\end{equation}
This turns a problem \eqref{pb: T} for $\vecU$ into a problem \eqref{pb: R} for $\vecV\in\HD$.

Since both a solution $\vecU$ of \eqref{pb: T} and a solution $\vecV$ of \eqref{pb: R} can be written using the representation formula \eqref{eq: representation formula compact}, we expect \eqref{eq: iso boundary} to define an isomorphism $\Xi:\HT\rightarrow\HR$ that also turns one of the boundary integral equation into the other. And indeed, one can verify that
\begin{equation*}
\{\traceR\}\LT\left(\Xi\,\vecb\right)=\Xi\,\traceT\mathcal{L}_{\mathsf{R}}\left(\vecb\right)\qquad\qquad\text{and}\qquad\qquad\{\traceT\}\LT\left(\Xi\,\vecb\right) = \Xi\,\{\traceR\}\LR\left(\vecb\right).
\end{equation*}
Hence, \eqref{pb: BT} can be equivalently formulated as a problem \eqref{pb: BR} with unknown ``$\Xi^{-1}\veca$" and given data $\Xi\vecb$ .

Let us take a closer look at the bilinear forms naturally associated with the continuous first-kind boundary integral operators 
\begin{align}
\color{blue}\gamma_{\mathsf{T}}\mathcal{L}_{\mathsf{T}}:&\color{blue}\mathcal{H}_{\mathsf{R}}\rightarrow\mathcal{H}_{\mathsf{T}},\\ \color{red}\gamma_{\mathsf{R}}\mathcal{L}_{\mathsf{R}}:&\color{red}\mathcal{H}_{\mathsf{T}}\rightarrow\mathcal{H}_{\mathsf{R}},
\end{align} 
that map trace spaces to their dual spaces.

Let $\vec{\bl{a}}$ and $\vec{\bl{c}}$ be trial and test boundary vector fields lying in $\mathcal{H}_{\mathsf{R}}$, and similarly for $\vec{\bl{b}}$ and $\vec{\bl{d}}$ in $\mathcal{H}_{\mathsf{T}}$. Catching up with the calculations of \cref{sec: properties of the potentials}, we want to derive convenient integral formulas for
\begin{multline*}
\llangle\vec{\bl{c}},\gamma_{\mathsf{T}}\mathcal{L}_{\mathsf{T}}\left(\vec{\bl{a}}\right)\rrangle =-\langle \mathcal{c}_0,\gamma\,\text{div}\bm{\Psi}(\bl{a}_1)\rangle_{\Gamma} + \langle\bl{c}_1,\gamma_t\nabla\psi(\mathcal{a}_0)\rangle_{\tau}\\
+ \langle\bl{c}_1,\gamma_t\,\bl{curl}\bl{\Upsilon}(a_2)\rangle_{\tau} 
+ \langle c_2,\gamma_n\,\bl{curl}\bm{\Psi}(\bm{a}_1)\rangle_{\Gamma}
\end{multline*}
and
\begin{multline*}
\llangle\vec{\bl{d}},\gamma_{\mathsf{R}}\mathcal{L}_{\mathsf{R}}(\vec{\bl{b}})\rrangle =\langle d_0,\gamma_n\,\bl{curl}\bm{\Psi}(\bl{b}_1\times\bl{n})\rangle_{\Gamma} - \langle\bl{d}_1,\gamma_{\tau}\,\bl{curl}\bl{\Upsilon}(b_0)\rangle_{\tau}\\
+ \langle\bl{d}_1,\gamma_{\tau}\nabla\psi(\mathcal{b}_2)\rangle_{\tau} + \langle \mathcal{d}_2,\gamma\,\text{div}\bm{\Psi}(\bl{b}_1\times\bl{n})\rangle_{\Gamma}.
\end{multline*}

In the course of our derivation, we will often rely implicitly on the fact that $\bl{a}_1$ and $\bl{b}_1$ are tangential vector fields.

Using the fact that $\text{div}\,\bm{\Psi}(\bl{a}_1)=\psi\left(\text{div}_{\Gamma}\,\bl{a}_1\right)$ and $\text{div}\,\bm{\Psi}(\bl{b}_1\times\bl{n})=\psi\left(\text{curl}_{\Gamma}\bl{b}_1\right)$ \cite[Lem. 2.3]{MacCamy1984}, we immediately find that
\begin{align*}
\langle \mathcal{c}_0,\gamma\,\text{div}\bm{\Psi}(\bl{a}_1)\rangle_{\Gamma} = \int_{\Gamma}\int_{\Gamma}G_{\bl{x}}(\bl{y})\,\mathcal{c}_0(\bl{x})\,\text{div}_{\Gamma}\bl{a}_1(\bl{y})\dif\sigma(\bl{x})\dif\sigma(\bl{y})
\end{align*}
and
\begin{align}
\langle\mathcal{d}_2,\gamma\,\text{div}\bm{\Psi}(\bl{b}_1\times\bl{n})\rangle_{\Gamma} = \int_{\Gamma}\int_{\Gamma}G_{\bl{x}}(\bl{y})\,\mathcal{d}_2(\bl{x})\,\text{curl}_{\Gamma}\bl{b}_1(\bl{y})\dif\sigma(\bl{y})\dif\sigma(\bl{x}).
\end{align}
We know from \cite[Sec. 6.4]{claeys2017first} that 
\begin{align*}
&\langle\mathbf{d}_1,\gamma_{\tau}\,\bl{curl}\bl{\Upsilon}(b_0)\rangle_{\tau}\\ &\qquad\qquad= -\int_{\Gamma}\int_{\Gamma}G_{\bl{x}}(\bl{y})\,\left(\bl{n}(\bl{x})\times\bl{d}_1(\bl{x})\right)\cdot\left(\bl{n}(\bl{y})\times\nabla_{\Gamma}\,b_0(\bl{y})\right)\dif\sigma(\bl{y})\dif\sigma(\bl{x})\\
&\qquad\qquad= \int_{\Gamma}\int_{\Gamma}G_{\bl{x}}(\bl{y})\,\left(\bl{n}(\bl{x})\times\bl{d}_1(\bl{x})\right)\cdot\bl{curl}_{\Gamma}\,b_0(\bl{y})\dif\sigma(\bl{y})\dif\sigma(\bl{x}).
\end{align*}
Adapting the arguments, we also obtain
\begin{align*}
&\langle\mathbf{c}_1,\gamma_{t}\,\bl{curl}\bl{\Upsilon}(a_2)\rangle_{\tau} = \langle\mathbf{c}_1\times\mathbf{n},\gamma_{\tau}\,\bl{curl}\bl{\Upsilon}(\mathcal{a}_0)\rangle_{\tau}\\ &\qquad\qquad=\int_{\Gamma}\int_{\Gamma}G_{\bl{x}}(\bl{y})\,\left(\bl{n}(\bl{x})\times\bra{\bl{c}_1(\bl{x})\times\mathbf{n}(\mathbf{x})}\right)\cdot\bl{curl}_{\Gamma}\,a_2(\bl{y})\dif\sigma(\bl{y})\dif\sigma(\bl{x})\\
&\qquad\qquad=\int_{\Gamma}\int_{\Gamma}G_{\bl{x}}(\bl{y})\,\bl{c}_1(\bl{x})\cdot\bl{curl}_{\Gamma}\,a_2(\bl{y})\dif\sigma(\bl{y})\dif\sigma(\bl{x}).
\end{align*}
Again, from \cite[Sec. 6.4]{claeys2017first}, we can similarly extract
\begin{align*}
\langle c_2,\gamma_n\,\bl{curl}\bm{\Psi}(\mathbf{a}_1)\rangle_{\Gamma} &= -\int_{\Gamma}\int_{\Gamma}G_{\bl{x}}(\bl{y})\,\bl{a}_1(\bl{y})\cdot\left(\bl{n}(\bl{x})\times\nabla_{\Gamma}\,c_2(\bl{x})\right)\dif\sigma(\bl{y})\dif\sigma(\bl{x})\\
&=\int_{\Gamma}\int_{\Gamma}G_{\bl{x}}(\bl{y})\,\bl{a}_1(\bl{y})\cdot\bl{curl}_{\Gamma}\,c_2(\bl{x})\dif\sigma(\bl{y})\dif\sigma(\bl{x})
\end{align*}
and
\begin{align*}
\langle d_0,\gamma_n\,\bl{curl}\bl{\Psi}(\mathbf{b}_1\times\bl{n})\rangle_{\Gamma} &=-\int_{\Gamma}\int_{\Gamma}G_{\bl{x}}(\bl{y})\,\bra{\mathbf{n}(\mathbf{y})\times\bl{b}_1(\bl{y})}\cdot\bl{curl}_{\Gamma}\,d_0(\bl{x})\dif\sigma(\bl{y})\dif\sigma(\bl{x})
\end{align*}
Finally, it follows almost directly by definition that
\begin{align*}
\langle\bl{c}_1,\gamma_t\nabla\psi(\mathcal{a}_0)\rangle_{\tau} = -\int_{\Gamma}\int_{\Gamma}G_{\bl{x}}(\bl{y})\,\mathcal{a}_0(\bl{y})\,\text{div}_{\Gamma}\,\bl{c}_1(\bl{x})\dif\sigma(\bl{y})\dif\sigma(\bl{x}),
\end{align*}
and
\begin{align*}
\langle\bl{d}_1,\gamma_{\tau}\nabla\psi(\mathcal{b}_2)\rangle_{\tau} &=\int_{\Gamma}\int_{\Gamma}G_{\bl{x}}(\bl{y})\,\mathcal{b}_2(\bl{y})\,\text{curl}_{\Gamma}\,\bl{d}_1(\bl{x})\dif\sigma(\bl{y})\dif\sigma(\bl{x}).
\end{align*} 

Putting everything together yields the symmetric bilinear forms
\begin{greyFrame}
	\begin{align}\label{eq: wedge wedge bilinear form}
	\begin{split}
	\llangle\vec{\bl{c}},\gamma_{\mathsf{T}}\mathcal{L}_{\mathsf{T}}\left(\vec{\bl{a}}\right)\rrangle &=
	-\int_{\Gamma}\int_{\Gamma}G(\bl{x}-\bl{y})\,\mathcal{c}_0(\bl{x})\,\text{div}_{\Gamma}\bl{a}_1(\bl{y})\dif\sigma(\bl{x})\dif\sigma(\bl{y})\\
	&\qquad-\int_{\Gamma}\int_{\Gamma}G(\bl{x}-\bl{y})\,\mathcal{a}_0(\bl{y})\,\text{div}_{\Gamma}\,\bl{c}_1(\bl{x})\dif\sigma(\bl{y})\dif\sigma(\bl{x})\\
	&\qquad+\int_{\Gamma}\int_{\Gamma}G(\bl{x}-\bl{y})\,\bl{c}_1(\bl{x})\cdot\bl{curl}_{\Gamma}\,a_2(\bl{y})\dif\sigma(\bl{y})\dif\sigma(\bl{x})\\
	&\qquad+\int_{\Gamma}\int_{\Gamma}G(\bl{x}-\bl{y})\,\bl{a}_1(\bl{y})\cdot\bl{curl}_{\Gamma}\,c_2(\bl{x})\dif\sigma(\bl{y})\dif\sigma(\bl{x}),
	\end{split}
	\end{align}
\end{greyFrame}
\begin{greyFrame}
	\begin{align}\label{eq: vee vee bilinear form}
	\begin{split}
	\llangle\vec{\bl{d}},\gamma_{\mathsf{R}}\mathcal{L}_{\mathsf{R}}(\vec{\bl{b}})\rrangle & =
	-\int_{\Gamma}\int_{\Gamma}G(\bl{x}-\bl{y})\left(\bl{n}(\bl{y})\times\bl{b}_1(\bl{y})\right)\cdot\bl{curl}_{\Gamma}\,d_0(\bl{x})\dif\sigma(\bl{y})\dif\sigma(\bl{x})\\ &\qquad -\int_{\Gamma}\int_{\Gamma}G(\bl{x}-\bl{y})\left(\bl{n}(\bl{x})\times\bl{d}_1(\bl{x})\right)\cdot\bl{curl}_{\Gamma}\,b_0(\bl{y})\dif\sigma(\bl{y})\dif\sigma(\bl{x})\\ &\qquad+\int_{\Gamma}\int_{\Gamma}G(\bl{x}-\bl{y})\,\mathcal{b}_2(\bl{y})\,\text{curl}_{\Gamma}\,\bl{d}_1(\bl{x})\dif\sigma(\bl{y})\dif\sigma(\bl{x})\\
	&\qquad+\int_{\Gamma}\int_{\Gamma}G(\bl{x}-\bl{y})\,\mathcal{d}_2(\bl{x})\,\text{curl}_{\Gamma}\bl{b}_1(\bl{y})\dif\sigma(\bl{y})\dif\sigma(\bl{x}).
	\end{split}
	\end{align}
\end{greyFrame}
The above integrals must be understood as duality pairings.
\begin{remark}
	Let us highlight here, as we have announced in the introduction, that in the sense of \cite[Chap. 2.5]{kress1999linear}, these double integrals feature only weakly singular kernels!
\end{remark}

The non-local inner products
\begin{align}
\bra{u,v}_{-1/2} &:=\int_{\Gamma}\int_{\Gamma}G_\mathbf{x}\bra{\mathbf{y}}u(\mathbf{x})\,v(\mathbf{y})\dif\sigma(\mathbf{x})\dif\sigma(\mathbf{y}),\label{Dirac scalar prod}\\
\bra{\mathbf{u},\mathbf{v}}_{-1/2,\mathsf{T}} &:=\int_{\Gamma}\int_{\Gamma}G_\mathbf{x}\bra{\mathbf{y}}\mathbf{u}(\mathbf{x})\cdot\mathbf{v}(\mathbf{y})\dif\sigma(\mathbf{x})\dif\sigma(\mathbf{y}), \label{Dirac T prod}\\
\bra{\mathbf{u},\mathbf{v}}_{-1/2,\mathsf{R}} &:=\int_{\Gamma}\int_{\Gamma}G_\mathbf{x}\bra{\mathbf{y}}\bra{\mathbf{n}\bra{\mathbf{x}}\times\mathbf{u}(\mathbf{x})}\cdot\bra{\mathbf{n}\bra{\mathbf{y}}\times\mathbf{v}(\mathbf{y})}\dif\sigma(\mathbf{x})\dif\sigma(\mathbf{y}),\label{Dirac R prod}
\end{align}
respectively defined over $H^{-1/2}\bra{\Gamma}$, $\mathbf{H}_\mathsf{T}^{-1/2}\bra{\Gamma}:=(\mathbf{H}^{1/2}_\mathsf{T}\bra{\Gamma})'$ and $\mathbf{H}_\mathsf{R}^{-1/2}\bra{\Gamma}:=(\mathbf{H}^{1/2}_\mathsf{R}\bra{\Gamma})'$, where 
\begin{align}
&\mathbf{H}^{1/2}_\mathsf{T}(\Gamma):=\gamma_t(\mathbf{H}^1\bra{\Omega}) &\text{and} & &\mathbf{H}_\mathsf{R}^{1/2}(\Gamma):=\gamma_\tau(\mathbf{H}^1\bra{\Omega}),
\end{align}
are positive definite Hermitian forms, and induce equivalent norms on the trace spaces \cite[Sec. 4.1]{buffa2002maxwell}. In the following, we will concern ourselves with the coercivity and geometric structure of the bilinear forms
\begin{greyFrame}
	\begin{align}\label{bilinear form Dirac T}
	\begin{split}
	\color{blue}\mathcal{B}_{\mathsf{T}}\bra{\vec{\bl{a}},\vec{\bl{c}}}&:=\color{blue}\llangle\gamma_{\mathsf{T}}\mathcal{L}_{\mathsf{T}}\left(\vec{\bl{a}}\right),\vec{\bl{c}}\rrangle\\
	&\,\color{blue}=\bra{-\text{div}_{\Gamma}\,\bl{a}_1,\mathcal{c}_0}_{-1/2}   +\bra{\mathcal{a}_0,-\text{div}_{\Gamma}\,\bl{c}_1}_{-1/2} \\
	&\color{blue}\qquad\qquad+\bra{\bl{curl}_{\Gamma}\,a_2,\bl{c}_1}_{-1/2,\mathsf{T}} 
	+\bra{\bl{a}_1,\bl{curl}_{\Gamma}c_2}_{-1/2,\mathsf{T}}
	\end{split}
	\end{align}
\end{greyFrame}
\noindent and
\begin{greyFrame}
	\begin{align}\label{bilinear form Dirac R}
	\begin{split}
	\color{red}\mathcal{B}_{\mathsf{R}}\bra{\vec{\bl{b}},\vec{\bl{d}}}&:=\color{red}\llangle\gamma_{\mathsf{R}}\mathcal{L}_{\mathsf{R}}\left(\vec{\bl{b}}\right),\vec{\bl{d}}\rrangle\\
	&\,=\color{red}\left(\mathbf{b}_1,\nabla_{\Gamma}\,d_0\right)_{-1/2,\mathsf{R}}+\left(\nabla_{\Gamma}\,b_0,\mathbf{d}_1\right)_{-1/2,\mathsf{R}}\\
	&\color{red}\qquad\qquad+\left( \mathcal{b}_2,\text{curl}_{\Gamma}\bl{d}_1\right)_{-1/2}+\left( \text{curl}_{\Gamma}\bl{b}_1,\mathcal{d}_2\right)_{-1/2}.
	\end{split}
	\end{align}
\end{greyFrame}

\section{T-coercivity} \label{sec: T-coercivity}
Based on the space decomposition introduced by the next lemma, we design isomorphisms $\mathcal{H}_{\mathsf{R}}\rightarrow\mathcal{H}_{\mathsf{R}}$ and $\mathcal{H}_{\mathsf{T}}\rightarrow\mathcal{H}_{\mathsf{T}}$ that are instrumental for obtaining the desired generalized G{\aa}rding inequalities for $\mathcal{B}_{\mathsf{T}}$ and $\mathcal{B}_{\mathsf{R}}$.

\begin{lemma}[{See \cite[Sec. 7]{hiptmair2003coupling} and \cite[Lem. 2]{buffa2003galerkin}}]\label{lem: trace projection}
	There exists a continuous projection $\mathsf{Z}^{\Gamma}:\mathbf{H}^{-1/2}\bra{\text{\emph{div}}_{\Gamma},\Gamma}\rightarrow\mathbf{H}^{1/2}_\mathsf{R}(\Gamma)$ with 
	\begin{equation}
	\ker(\mathsf{Z}^{\Gamma})=\ker\bra{\text{\emph{div}}_{\Gamma}}\cap\mathbf{H}^{-1/2}\bra{\text{\emph{div}}_{\Gamma},\Gamma}
	\end{equation}
	and satisfying
	\begin{equation}\label{eq: divergence invariant}
	\text{\emph{div}}_{\Gamma}\bra{\mathsf{Z}^{\Gamma}(\mathbf{v})}=\text{\emph{div}}_{\Gamma}\bra{\mathbf{v}}.
	\end{equation}
\end{lemma}

The closed subspaces $\mathbf{X}\bra{\text{div}_{\Gamma},\Gamma}:=\mathsf{Z}^{\Gamma}\bra{\mathbf{H}^{-1/2}\bra{\text{div}_{\Gamma},\Gamma}}$ and $\mathbf{N}\bra{\text{div}_{\Gamma},\Gamma}:=\ker\bra{\text{div}_{\Gamma}}\cap\mathbf{H}^{-1/2}\bra{\text{div}_{\Gamma},\Gamma}$ provide a stable direct regular decomposition \begin{equation}
\mathbf{H}^{-1/2}\bra{\text{div}_{\Gamma},\Gamma} = \mathbf{X}\bra{\text{div}_{\Gamma},\Gamma}\oplus \mathbf{N}\bra{\text{div}_{\Gamma},\Gamma}.
\end{equation} 
Hence, it follows from \eqref{eq: divergence invariant} that
\begin{equation}\label{eq:equiv norm on boundary}
\mathbf{v}\mapsto\norm{\text{div}_{\Gamma}\bra{\mathbf{v}}}_{-1/2} + \norm{(\id - \mathsf{Z}^{\Gamma})\,\mathbf{v}}_{-1/2}
\end{equation} 
also defines an equivalent norm in $\mathbf{H}^{-1/2}\bra{\text{div}_{\Gamma},\Gamma}$. 

Note that since, by Rellich's embedding theorem, $\mathbf{H}^{1/2}_\mathsf{R}(\Gamma)$ compactly embeds in the space $\mathbf{L}^2_{t}\bra{\Gamma}:=\{\mathbf{u}\in\mathbf{L}^2\bra{\Gamma}\,\vert\, \mathbf{u}\cdot\mathbf{n}\equiv 0\}$ of square-integrable tangential vector-fields, this is also the case for $\mathbf{X}\bra{\text{div}_{\Gamma},\Gamma}$.

From \Cref{lem: properties of surface div and curl}, $\text{div}_{\Gamma}:\mathbf{X}\bra{\text{div}_{\Gamma},\Gamma}\rightarrow H^{-1/2}_*(\Gamma)$ is a continuous bijection, thus the bounded inverse theorem guarantees the existence of a continuous inverse $\bra{\text{div}_{\Gamma}}^{\dag}:H^{-1/2}_*(\Gamma)\rightarrow \mathbf{X}\bra{\text{div}_{\Gamma},\Gamma}$ such that
\begin{align*}
\bra{\text{div}_{\Gamma}}^{\dag}\circ\text{div}_{\Gamma} &=\id\Big\vert_{\mathbf{X}\bra{\text{div}_{\Gamma},\Gamma}}, & &\text{div}_{\Gamma}\circ\bra{\text{div}_{\Gamma}}^{\dag}=\id\Big\vert_{H^{-1/2}_*(\Gamma)}.
\end{align*}

The existence of an operator $\mathbf{curl}_{\Gamma}^{\dag}:\mathbf{N}\bra{\text{div}_{\Gamma},\Gamma}\rightarrow H_*^{1/2}\bra{\Gamma}$ satisfying $\mathbf{curl}_{\Gamma}^{\dag}\circ\mathbf{curl}_{\Gamma}=\id$ and $\mathbf{curl}_{\Gamma}\circ\mathbf{curl}_{\Gamma}^{\dag} =\text{`}\mathbf{H}^{-1/2}(\text{div}_\Gamma,\Gamma)$-orthogonal projection onto (surface) divergence-free vector-fields' also follows by \Cref{lem: properties of surface div and curl}.

In the following, we will denote by $Q_*$ both the projection $H^{1/2}(\Gamma)\rightarrow H^{1/2}_*(\Gamma)$ onto mean zero functions and the projection $H^{-1/2}(\Gamma)\rightarrow H^{-1/2}_*(\Gamma) $ onto the space of annihilators of the characteristic function.
\begin{lemma}
	The bounded linear operator 
	\begin{equation*}
	\Xi:H_{*}^{-1/2}\bra{\Gamma}\times\mathbf{H}^{-1/2}(\text{\emph{div}}_\Gamma,\Gamma) \times H_*^{1/2}\bra{\Gamma}\rightarrow  H^{-1/2}_{*}\bra{\Gamma}\times\mathbf{H}^{-1/2}(\text{\emph{div}}_\Gamma,\Gamma) \times H_{*}^{1/2}\bra{\Gamma}
	\end{equation*}
	defined by
	\begin{equation*}
	\Xi
	\begin{pmatrix}
	\mathcal{a}_0\\
	\bl{a}_1\\
	a_2
	\end{pmatrix}
	=
	\begin{pmatrix}
	-\text{\emph{div}}_{\Gamma}\,\bl{a}_1 \\
	-\bra{\text{\emph{div}}_{\Gamma}}^{\dag}\bra{Q_*\mathcal{a}_0} + \bl{curl}_{\Gamma}\bra{Q_*a_2}\\
	\left(\mathbf{curl}_{\Gamma}\right)^{\dag}\bra{\bra{\id-\mathsf{Z}^{\Gamma}}\bl{a}_1}
	\end{pmatrix}
	\end{equation*}
	is a continuous involution. In particular, $\Xi$ is an isomorphism of Banach spaces.
\end{lemma}
\begin{proof}
	We directly evaluate 
	\begin{align*}
	&\Xi^2
	\begin{pmatrix}
	\mathcal{a}_0\\
	\bl{a}_1\\
	a_2
	\end{pmatrix}
	=
	\Xi\begin{pmatrix}
	-\text{div}_{\Gamma}\,\bl{a}_1\\
	-\bra{\text{div}_{\Gamma}}^{\dag}\bra{Q_*\mathcal{a}_0} + \bl{curl}_{\Gamma}\bra{Q_*a_2}\\
	\left(\mathbf{curl}_{\Gamma}\right)^{\dag}\bra{\bra{\id-\mathsf{Z}^{\Gamma}}\bl{a}_1}
	\end{pmatrix}\\
	&\qquad=
	\begin{psmallmatrix}
	\text{div}_{\Gamma}\bra{\bra{\text{div}_{\Gamma}}^{\dag}\bra{Q_*\mathcal{a}_0}}-\text{div}_{\Gamma}\bra{\bl{curl}_{\Gamma}\bra{Q_*a_2}}\\
	\bra{\text{div}_{\Gamma}}^{\dag}\bra{Q_*\bra{\text{div}_{\Gamma}\,\bl{a}_1}} + \bl{curl}_{\Gamma}\bra{Q_*\left(\mathbf{curl}_{\Gamma}\right)^{\dag}\bra{\bra{\id-\mathsf{Z}^{\Gamma}}\bl{a}_1}}\\
	-\left(\mathbf{curl}_{\Gamma}\right)^{\dag}\bra{\bra{\id-\mathsf{Z}^{\Gamma}}\bra{\bra{\text{div}_{\Gamma}}^{\dag}\bra{Q_*\mathcal{a}_0}}} +\left(\mathbf{curl}_{\Gamma}\right)^{\dag}\bra{\bra{\id-\mathsf{Z}^{\Gamma}}\bra{\bl{curl}_{\Gamma}\bra{Q_*a_2}}}
	\end{psmallmatrix}\\
	&\qquad=
	\begin{pmatrix}
	Q_*\mathcal{a}_0\\
	\mathsf{Z}^{\Gamma}\bl{a}_1 + \bra{\id-\mathsf{Z}^{\Gamma}}\bl{a}_1\\
	Q_*a_2
	\end{pmatrix}=\begin{pmatrix}
	\mathcal{a}_0\\
	\bl{a}_1\\
	a_2
	\end{pmatrix}.
	\end{align*}
\end{proof}

\begin{greyFrame}
	\begin{proposition}
		There exists a constant $C>0$ and a compact bilinear form $\mathcal{C}:\mathcal{H}_{\mathsf{R}}\times\mathcal{H}_{\mathsf{R}}\rightarrow\mathbb{R}$ such that   
		\begin{equation}
		\big\vert\llangle\Xi\,\vec{\bl{a}},\gamma_{\mathsf{T}}\mathcal{L}_{\mathsf{T}}\left(\vec{\bl{a}}\right)\rrangle_{\times} + \mathcal{C}\bra{\vec{\bl{a}},\vec{\bl{a}}}\big\vert\geq C\norm{\vec{\bl{a}}}^2_{\mathcal{H}_{\mathsf{R}}}\qquad\forall\vec{\bl{a}}\in\mathcal{H}_{\mathsf{R}}.
		\end{equation}
	\end{proposition}
\end{greyFrame}
\begin{proof}
	The operator $\mathbf{curl}_{\Gamma}:H^1_*(\Gamma)\rightarrow \mathbf{H}^{-1/2}(\text{div}_\Gamma)$ is a continuous injection with closed range, it is thus bounded below. Since the mean operator has finite rank, it is compact. Moreover, $\bra{\text{div}_{\Gamma}}^{\dag}\bra{H^{-1/2}_*(\Gamma)}\subset\mathbf{H}^{1/2}_\mathsf{R}(\Gamma)$ is compactly embedded in $\mathbf{L}_t^2(\Gamma)$. Hence, the proof ultimately follows from
	\begin{align*}
	&\llangle\Xi\,\vec{\bl{a}},\gamma_{\mathsf{T}}\mathcal{L}_{\mathsf{T}}\left(\vec{\bl{a}}\right)\rrangle_{\times} \,\hat{=}\,
	\bra{\text{div}_{\Gamma}\,\bl{a}_1,\text{div}_{\Gamma}\,\bl{a}_1}_{-1/2} + \bra{a_2,Q_*a_2}_{-1/2} \\
	&\qquad\qquad\qquad+ \bra{\bra{\text{div}_{\Gamma}}^{\dag}Q_*a_2,\bl{curl}_{\Gamma}\,\mathcal{a}_0}_{-1/2}
	+\bra{\bl{curl}_{\Gamma}\,Q_*\mathcal{a}_0,\bl{curl}_{\Gamma}\,\mathcal{a}_0}_{-1/2}\\
	&\qquad\qquad\qquad+\bra{\bl{a}_1,\bra{\id-\mathsf{Z}^{\Gamma}}\bl{a}_1}_{-1/2}
	\end{align*}
	and the opening observations of this section.
\end{proof}

Since $\text{curl}_{\Gamma}\bra{\mathbf{d}} = \text{div}_{\Gamma}\bra{\mathbf{n}\times\mathbf{d}}$ for all $\mathbf{d}\in\mathbf{H}^{-1/2}(\text{curl}_\Gamma,\Gamma)$, tinkering with the signs and introducing rotations in the definition of $\Xi$ easily leads to an analogous generalized G{\aa}rding inequality for $\gamma_{\mathsf{R}}\mathcal{L}_{\mathsf{R}}$.

\begin{greyFrame}
	\begin{corollary}\label{cor: BIOs are Fredholm}
		The boundary integral operators $\gamma_{\mathsf{T}}\mathcal{L}_{\mathsf{T}}:\mathcal{H}_{\mathsf{R}}\rightarrow\mathcal{H}_{\mathsf{T}}$ and $\gamma_{\mathsf{R}}\mathcal{L}_{\mathsf{R}}:\mathcal{H}_{\mathsf{T}}\rightarrow\mathcal{H}_{\mathsf{R}}$ are Fredholm of index $0$.
	\end{corollary}
\end{greyFrame}

\section{Kernels}\label{sec: Kernels}

We conclude from \Cref{cor: BIOs are Fredholm} that the nullspaces of $\gamma_{\mathsf{T}}\mathcal{L}_{\mathsf{T}}$ and $\gamma_{\mathsf{R}}\mathcal{L}_{\mathsf{R}}$ are finite dimensional. In this section, we proceed similarly as in \cite[Sec. 7.1]{claeys2017first} and \cite[Sec. 3]{claeys2018first} to characterize them explicitly.

Suppose that $\vec{\bl{a}}\in\mathcal{H}_{\mathsf{R}}$ is such that $\gamma_{\mathsf{T}}\mathcal{L}_{\mathsf{T}}\bra{\vec{\bl{a}}}=0$. 
\begin{itemize}
	\item Since $\text{div}_{\Gamma}\,\bl{a}_1\in H^{-1/2}\bra{\Gamma}$, we can test the bilinear form of \Cref{eq: wedge wedge bilinear form} with $\mathcal{c}_0=\text{div}_{\Gamma}\,\bl{a}_1$, $\bl{c}_1=0$ and $c_2=0$ to find that $\text{div}_{\Gamma}\,\bl{a}_1=0$. 
	
	\item Testing with $\mathcal{c}_0=0$ and $\bl{c}_1=0$ shows that $\bra{\bl{a}_1,\mathbf{curl}_{\Gamma}\,v}_{-1/2} = 0$ $\forall\,v\in H^{1/2}(\Gamma)$. 
	
	\item Because $\text{div}_{\Gamma}\circ\mathbf{curl}_{\Gamma} = 0$, we can choose $c_2=0$, $\mathcal{c}_0=0$ and $\bl{c}_1=\mathbf{curl}_{\Gamma}a_2$ to conclude that $\mathbf{curl}_{\Gamma}a_2=0$.
	\item We are left with $\bra{\mathcal{a}_0,\text{div}_{\Gamma}\bl{v}}_{-1/2}=0$ $\forall\,\mathbf{v}\in\mathbf{H}^{-1/2}(\text{div}_\Gamma,\Gamma)$.
\end{itemize}

In $H^{1/2}\bra{\Gamma}$, $\ker\bra{\mathbf{curl}_{\Gamma}}=\ker\bra{\nabla_{\Gamma}}$ is the space of functions $\mathcal{C}\bra{\Gamma}$ that are constant over connected components of $\Gamma$. Defining $\bl{\Psi}_{t}:= \gamma_{t}\bl{\Psi}$, we have found that
\begin{greyFrame}
	\begin{multline}\label{eq:ker wedge wedge}
	\ker\bra{\gamma_{\mathsf{T}}\mathcal{L}_{\mathsf{T}}} =\\ \left\{\vec{\bl{a}}\in\mathcal{H}_{\mathsf{R}}\,\,\Big\vert\,\,a_0\in\mathcal{C}\bra{\Gamma},\,\text{curl}_{\Gamma}\bm{\Psi}_t(\bl{a}_1)=0,\,\text{div}_{\Gamma}\bl{a}_1=0,\,\nabla_{\Gamma}\psi\bra{a_0'}=0\right\}.
	\end{multline}
\end{greyFrame}

Now, suppose that $\vec{\bl{b}}\in\mathcal{H}_{\mathsf{T}}$ is such that $\gamma_{\mathsf{R}}\mathcal{L}_{\mathsf{R}}(\vec{\bl{b}})=0$. 
\begin{itemize}
	\item As $\text{curl}_{\Gamma}\bra{\bl{b}_1}\in H^{-1/2}\bra{\Gamma}$, we may test \Cref{eq: vee vee bilinear form} with $\mathcal{d}_2=\text{curl}_{\Gamma}\,\bl{b}_1$, $\bl{d}_1=0$ and $d_0=0$ to find that $\text{curl}_{\Gamma}\,\bl{b}_1=0$. 
	
	\item Testing with $\mathcal{d}_2=0$ and $\bl{d}_1=0$, we find that $\bra{\mathbf{n}\times\bl{b}_1,\mathbf{curl}_{\Gamma}\,v}_{-1/2} = 0$ for all $v\in H^{1/2}(\Gamma)$. 
	
	\item Since $\text{curl}_{\Gamma}\circ\nabla_{\Gamma} = 0$, we can choose $d_0=0$, $\mathcal{d}_2=0$ and $\bl{d}_1=\nabla_{\Gamma}b_0$ to conclude that $\mathbf{curl}_{\Gamma}\,b_0=0$.
	\item Finally, it follows that $\bra{\mathcal{b}_2,\text{curl}_{\Gamma}\,\bl{v}}_{-1/2}=0$ for all $\mathbf{v}\in\mathbf{H}^{-1/2}(\text{curl}_\Gamma,\Gamma)$.
\end{itemize}
Notice that since $\nabla_{\Gamma}\bra{v}$ is tangential for all $v\in H^{1/2}\bra{\Gamma}$,
\begin{equation*}
\bra{\mathbf{n}\times\bl{b}_1,\mathbf{curl}_{\Gamma}\,v}_{-1/2} = \bra{\mathbf{n}\times\bl{b}_1,\nabla_{\Gamma}v\times\mathbf{n}}_{-1/2}
=\langle \mathbf{n}\times\bl{\Psi}\bra{\mathbf{n}\times\bl{b}_1},\nabla_{\Gamma}v\rangle
\end{equation*}
for all $v\in H^{1/2}\bra{\Gamma}$. Therefore, we let $\bl{\Psi}_{\tau}\bra{\cdot}:= -\gamma_{\tau}\bl{\Psi}\bra{\bl{n}\times \cdot}$ and conclude that
\begin{greyFrame}
	\begin{multline}\label{eq:ker vee vee}
	\ker\bra{\gamma_{\mathsf{R}}\mathcal{L}_{\mathsf{R}}} =\\ \left\{\vec{\bl{b}}\in\mathcal{H}_{\mathsf{T}}\,\Big\vert\,\, b_0\in\mathcal{C}\bra{\Gamma},\text{curl}_{\Gamma}\,\bl{b}_1=0,\text{div}_{\Gamma}\bl{\Psi}_{\tau}\bra{\bl{b}_1}=0,\mathbf{curl}_{\Gamma}\psi\bra{b_0'}=0\right\}.
	\end{multline}
\end{greyFrame}

\Cref{eq:ker wedge wedge} and \Cref{eq:ker vee vee} together with the mapping properties of the scalar and vector single layer potentials allow us to determine as in \cite[Sec. 7.2]{claeys2017first} and \cite[Lem. 2, Lem. 6]{claeys2018first} that the dimension of these nullspaces relate to the Betti numbers of $\Gamma$.
\begin{greyFrame}
	\begin{proposition}
		The dimensions of $\ker\bra{\gamma_{\mathsf{T}}\mathcal{L}_{\mathsf{T}}}$ and $\ker\bra{\gamma_{\mathsf{R}}\mathcal{L}_{\mathsf{R}}}$ are finite and equal to the sum of the Betti numbers $\beta_0\bra{\Gamma}+\beta_1\bra{\Gamma} + \beta_2\bra{\Gamma}$.
	\end{proposition}
\end{greyFrame}
\begin{remark}
	The zeroth Betti number $\beta_0\bra{\Gamma}$ indicates the number of connected components of $\Gamma$. The first Betti number $\beta_1\bra{\Gamma}$ amounts to the number of equivalence classes of non-bounding cycles in $\Gamma$. For the second Betti number, it holds that $\beta_2\bra{\Gamma}=\beta_2\bra{\Omega^+} + \beta_2\bra{\Omega^-}$, which sums the number of holes in $\Omega^+$ and $\Omega^-$, respectively.
\end{remark}
\iffalse
\begin{greyFrame}
	\begin{corollary}
		The restrictions of $\gamma_{\mathsf{T}}\mathcal{L}_{\mathsf{T}}$ and $\gamma_{\mathsf{R}}\mathcal{L}_{\mathsf{R}}$ to $H^{-1/2}\bra{\Omega}\times\mathbf{H}^{-1/2}(\text{\emph{div}}_\Gamma,\Gamma) \times H_*^{1/2}\bra{\Omega}$ and $H_*^{1/2}\bra{\Omega}\times\mathbf{H}^{-1/2}(\text{\emph{curl}}_\Gamma,\Gamma) \times H^{-1/2}\bra{\Omega}$, respectively, have dimension $\beta_0\bra{\Gamma} + \beta_1\bra{\Gamma}$.
	\end{corollary}
\end{greyFrame}
\fi

\section{Surface Dirac operators}\label{surface dirac operators}
In this section, we reveal the geometric structure behind the formulas of the bilinear forms $\mathcal{B}_{\mathsf{R}}$ and $\mathcal{B}_{\mathsf{T}}$ established in \Cref{Sec: boundary integral operators}. They turn out to be associated with the 2D surface Dirac operators induced by the {\color{red}chain} and {\color{blue}cochain} Hilbert complexes
\begin{equation}\label{base complex t}\color{red}
\xymatrix{
	H^{-1/2}\bra{\Gamma} \ar[r]^-{\nabla_{\Gamma}}& \mathbf{H}_\mathsf{T}^{-1/2}\bra{\Gamma} \ar[r]^-{\text{curl}_{\Gamma}}& H^{-1/2}\bra{\Gamma}
}
\end{equation}
and
\begin{equation}\label{base complex tau}\color{blue}
\xymatrix{
	H^{-1/2}\bra{\Gamma}
	&\ar[l]^-{\,\,-\text{div}_{\Gamma}} \mathbf{H}_\mathsf{R}^{-1/2}\bra{\Gamma}
	&\ar[l]^-{\mathbf{curl}_{\Gamma}} H^{-1/2}\bra{\Gamma},
}
\end{equation}
equipped with the non-local inner products \eqref{Dirac scalar prod}, \eqref{Dirac T prod} and \eqref{Dirac R prod}. Their associated domain complexes
\begin{equation}\label{de Rham cochain}\color{red}
\xymatrix{
	H^{1/2}\bra{\Gamma} \ar[r]^-{\nabla_{\Gamma}}& \mathbf{H}^{-1/2}(\text{curl}_\Gamma,\Gamma) \ar[r]^-{\text{curl}_{\Gamma}}& H^{-1/2}\bra{\Gamma}
}
\end{equation}
and
\begin{equation}\label{de Rham chain}\color{blue}
\xymatrix{
	H^{-1/2}\bra{\Gamma}
	&\ar[l]^-{\,\,-\text{div}_{\Gamma}} \mathbf{H}^{-1/2}(\text{div}_\Gamma,\Gamma)
	&\ar[l]^-{\mathbf{curl}_{\Gamma}} H^{1/2}\bra{\Gamma},
}
\end{equation}
are equipped with the natural graph inner products.
\begin{remark}
	Notice that \eqref{de Rham cochain} and \eqref{de Rham chain} are dual to each other with respect to the duality pairing on the boundary introduced in \Cref{Function spaces and traces}.
\end{remark}

The Hilbert space adjoint $\color{blue}\mathbf{d}_{\Gamma}^*$ and $\color{red}\bm{\delta}_{\Gamma}^*$ of the nilpotent operators
\begin{align}
    \color{red}\mathbf{d}_{\Gamma}&\color{red}:\mathcal{H}_{\mathsf{T}}\rightarrow\mathcal{H}_{\mathsf{T}},\\
    \color{blue}\bm{\delta}_{\Gamma}&\color{blue}:\mathcal{H}_{\mathsf{R}}\rightarrow\mathcal{H}_{\mathsf{R}},
\end{align}
represented by the block operator matrices
	\begin{align*}
	&{\color{red}\mathbf{d}_{\Gamma}:=
    \begin{pmatrix}0 & \bl{0}^\top &0\\
    \surfgrad & \bl{0}_{3\times 3} & \bl{0} \\
    0 & \surfcurl & 0\end{pmatrix}} &\text{and}&
	&{\color{blue}\bm{\delta}_{\Gamma}:=
   \begin{pmatrix}
   0 & -\text{div}_{\Gamma} &0\\
    \bl{0}& \bl{0}_{3\times 3} & \bf{curl}_{\Gamma} \\
    0 & \bl{0}^\top & 0\end{pmatrix}
    }
	\end{align*}
\noindent are \emph{non-local} operators. 

In terms of variational formulations, the bilinear forms associated with the surface Dirac operators
\begin{align}
    \color{red}\mathsf{D}^{\Gamma}_{\mathsf{R}}&\color{red}:=\mathbf{d}_{\Gamma} + \mathbf{d}_{\Gamma}^*\\
    \color{blue}\mathsf{D}^{\Gamma}_{\mathsf{T}}&\color{blue}:=\bm{\delta}_{\Gamma} + \bm{\delta}_{\Gamma}^*
\end{align}
are precisely $\color{red}\mathcal{B}_{\mathsf{R}}$ and $\color{blue}\mathcal{B}_{\mathsf{T}}$ defined in \eqref{bilinear form Dirac R} and \eqref{bilinear form Dirac T}, previously associated to the boundary integral operators $\color{red}\gamma_{\mathsf{R}}\mathcal{L}_{\mathsf{R}}$ and $\color{blue}\gamma_{\mathsf{T}}\mathcal{L}_{\mathsf{T}}$:
\begin{greyFrame}
	\begin{align}\label{bilinear T is Dirac T}
	\begin{split}
	\bra{{\color{red}\mathsf{D}^{\Gamma}_{\mathsf{R}}}\,\vec{\bl{b}},\vec{\bl{d}}}_{\mathcal{H}_\mathsf{T}} &=\bra{\mathbf{d}_{\Gamma}\vec{\bl{b}},\vec{\bl{d}}}_{\mathcal{H}_\mathsf{T}} 
	+ \bra{\vec{\bl{b}},\mathbf{d}_{\Gamma}\vec{\bl{d}}}_{\mathcal{H}_\mathsf{T}} \\
	&=\left(\nabla_{\Gamma}\,b_0,\mathbf{d}_1\right)_{-1/2,\mathsf{R}}
	+\left(\text{curl}_{\Gamma}\bl{b}_1,\mathsf{d}_2\right)_{-1/2}\\
	&\qquad\qquad+\left(\mathbf{b}_1,\nabla_{\Gamma}\,d_0\right)_{-1/2,\mathsf{R}}
	+\left( \mathsf{b}_2,\text{curl}_{\Gamma}\bl{d}_1\right)_{-1/2}\\
	&=\mathcal{B}_{\mathsf{R}}\bra{\vec{\bl{b}},\vec{\bl{d}}},
	\end{split}
	\end{align}
\end{greyFrame}
\noindent and similarly
\begin{greyFrame}
	\begin{align}
	\begin{split}\label{bilinear R is Dirac R}
	\bra{{\color{blue}\mathsf{D}^{\Gamma}_{\mathsf{T}}}\,\vec{\bl{a}},\vec{\bl{c}}}_{\mathcal{H}_\mathsf{R}}
	&=\bra{\bm{\delta}_{\Gamma}\vec{\bl{a}},\vec{\bl{c}}}_{\mathcal{H}_\mathsf{R}} 
	+ \bra{\vec{\bl{a}},\bm{\delta}_{\Gamma}\vec{\bl{c}}}_{\mathcal{H}_\mathsf{R}} \\
	&= \bra{-\text{div}_{\Gamma}\,\bl{a}_1,\mathcal{c}_0}_{-1/2}   +\bra{\mathcal{a}_0,-\text{div}_{\Gamma}\,\bl{c}_1}_{-1/2}\\
	&\qquad\qquad +\bra{\bl{curl}_{\Gamma}\,a_2,\bl{c}_1}_{-1/2,\mathsf{T}} 
	+\bra{\bl{a}_1,\bl{curl}_{\Gamma}c_2}_{-1/2,\mathsf{T}}\\
	&=\mathcal{B}_{\mathsf{T}}\bra{\vec{\bl{a}},\vec{\bl{c}}}.
	\end{split}
	\end{align}
\end{greyFrame}

\emph{First-kind boundary integral operators spawned by the (volume) Dirac operators in 3D Euclidean space thus coincide with (surface) Dirac operators on 2D boundaries:} boundary value problems related to $\color{red}\mathsf{D}^{\Omega}_{\mathsf{R}}=\bl{d}+\bl{d}^*$ in $\Omega$ can be formulated as problems for $\color{red}\mathsf{D}^{\Gamma}_{\mathsf{R}}=\mathbf{d}_{\Gamma}+\mathbf{d}_{\Gamma}^*$ in $\Gamma$, and similarly problems for $\color{blue}\mathsf{D}^{\Omega}_{\mathsf{T}}=\bm{\delta}+\bm{\delta}^*$ in $\Omega$ correspond to problems for $\color{blue}\mathsf{D}^{\Gamma}_{\mathsf{T}}=\bm{\delta}_{\Gamma} + \bm{\delta}_{\Gamma}^*$ in $\Gamma$.

This explains why the dimension of the nullspaces of first-kind boundary integral operators is the sum of the dimensions of the standard spaces of surface harmonic scalar and vector fields.

\section{Solvability}
Thanks to the duality between the trace spaces, \eqref{pb: BT} and \eqref{pb: BR} can be reformulated into the variational problems:
\begin{align}\label{pb: BVT}
&\color{blue}\veca\in\HR: &\color{blue}\mathcal{B}_{\mathsf{T}}\bra{\vec{\bl{a}},\vec{\bl{c}}}= \ell_{\mathsf{T}}(\vecc), &&\color{blue}\forall\,\vecc\in\HR,\tag{BV\textsf{T}}
\end{align}
and
\begin{align}\label{pb: BVR}
&\color{red}\vecb\in\HT: &\color{red}\mathcal{B}_{\mathsf{R}}\bra{\vec{\bl{b}},\vec{\bl{d}}}= \ell_{\mathsf{R}}(\vecd), &&\color{red}\forall\,\vecd\in\HT,\tag{BV\textsf{R}}
\end{align}
with right-hand side functionals
\begin{equation}\label{eq: right hand side ell T}
\color{blue}
    \ell_{\mathsf{T}}(\vecc) = \llangle \frac{1}{2}\vecb - \{\traceT\}\LR(\vecb),\vecc\rrangle_{\Gamma}
\end{equation}
and
\begin{equation}\label{eq: right hand side ell R}
\color{red}
    \ell_{\mathsf{R}}(\vecd) = \llangle \frac{1}{2}\veca - \{\traceR\}\LT(\veca),\vecd\rrangle_{\Gamma}.
\end{equation}

As explained in \Cref{sec: Duality and symmetry}, it is sufficient when it comes to well-posedness to restrict our considerations to only one of the two boundary integral equations stated in \Cref{Sec: boundary integral operators}. The following result makes explicit the condition under which a solution to \eqref{pb: BVR} exists.
\begin{proposition}\label{lem: compatibility rhs T}
If the boundary data $\veca\in\HR$ satisfies the compatibility condition \eqref{CCR}, then the right-hand side $\ell\in\HT'$ of \eqref{pb: BVR} is consistent in the sense that
\begin{equation}
    \ell_{\mathsf{R}}(\vecd) = 0, \qquad\quad\forall\,\vecd\in\ker\biT.
\end{equation}
\end{proposition}
\begin{proof}
Following the strategy found in the proofs of \cite[Lem. 4]{claeys2018first} and \cite[Lem. 8]{claeys2018first}, we use \eqref{eq: wedge potential} to directly evaluate
\begin{align*}
\ell_{\mathsf{R}}(\vecd) &= \llangle \frac{1}{2}\veca - \{\traceR\}\LT(\veca),\vecd\rrangle_{\Gamma}\\
&= \llangle \frac{1}{2}\veca,\vecd\rrangle_{\Gamma}-\langle\{\gamma_n\}\,\curl \Upsilon (a_2), d_0\rangle_{\Gamma}+\langle \mathsf{K}'(a_0), d_0\rangle_{\Gamma} \\ &\qquad\qquad\qquad\qquad\qquad\qquad\qquad\qquad\qquad+
\langle \bl{a}_1, \mathsf{C}(\bl{b}_1)\rangle_{\Gamma}-\langle \mathsf{K}(a_2), \mathcal{b}_2\rangle_{\Gamma}\\
&=\langle (\frac{1}{2}\id-\mathsf{K}')a_0, d_0\rangle_{\Gamma} + \langle\{\gamma_n\}\,\curl \Upsilon (a_2), d_0\rangle_{\Gamma}\\ &\qquad\qquad\qquad\qquad\qquad\qquad\qquad+
\langle \bl{a}_1, (\frac{1}{2}\id+\mathsf{C})\bl{b}_1\rangle_{\Gamma}+\langle (\frac{1}{2}\id+\mathsf{K})a_2, \mathcal{b}_2\rangle_{\Gamma},
\end{align*}
where we recognize the ``Maxwell double layer boundary integral operator" $\mathsf{C}$, and the double layer boundary integral operator $\mathsf{K}$ for the Laplacian.

Locally constant functions are trivially harmonic. They can thus be written using the classical representation formula for the scalar Laplacian in which the Neumann trace vanishes to yield $d_0 = \gamma\,(\frac{1}{2}\id-\mathsf{K})d_0$. Since $\mathsf{K}$ is dual to $\mathsf{K}'$, the first term on the right-hand side vanishes because of the compatibility condition $\eqref{CCR}$.

The second term also evaluates to zero. On the one hand, $\ker\surfcurlbl = \ker\surfgrad$. On the other hand, $\gamma_n\,\curl= \surfcurl\,\gamma_t$ in $\Hcurl$, and $\surfcurl$ is dual to $\surfcurlbl$.

The third and fourth terms are shown to vanish in \cite[Lem.4]{claeys2018first} and \cite[Lem.3]{claeys2018first} with similar arguments.
\end{proof}

In the framwork of \Cref{surface dirac operators}, a standard result is the Poincar\'e inequality: $\exists\,C>0$, only depending on $\Gamma$, such that \cite{leopardi2016abstract,arnold2018finite}
\begin{align}\label{eq: Poincare inequality}
    \norm{\vecb}_{\HR} \leq C \norm{\bl{d}_{\Gamma}\vecb}_{\HR},\qquad\qquad \forall\, \vecb\in\mathfrak{K},
\end{align}
where $\mathfrak{K}:=\left(\ker\mathbf{d}_{\Gamma}\right)^{\perp}\cap\text{dom}(\bl{d}_{\Gamma})$ and orthogonality is taken in the non-local inner products introduced in \Cref{Sec: boundary integral operators}. From the complex \eqref{de Rham cochain},
\begin{equation}
    \text{dom}(\bl{d}_{\Gamma}) = H^{1/2}\bra{\Gamma}\times \mathbf{H}^{-1/2}(\text{curl}_\Gamma,\Gamma)
    \times H^{-1/2}\bra{\Gamma},
\end{equation}
and thus
\begin{equation}
\mathfrak{K}= \mathfrak{K}_{0}\times\mathfrak{K}_{1}\times\mathfrak{K}_{2} \in\HT
\end{equation}
with
\begin{align*}
    \mathfrak{K}_{0} %&:=\left\{b_0\in\Hhalf : \surfgrad b_0 = \bl{0}\right\}\\
    := \ker\surfgrad,
    &&\mathfrak{K}_{1} %&:=\left\{\bl{b}_1\in\Hhalfcurl:\surfcurl \bl{b}_1=0,\,(\bl{b}_1,\surfgrad v )_{\ipT}=0\,\,\forall v\in\Hhalf\right\}\\
    :=\ker\surfcurl\cap\left(\surfgrad\Hhalf\right)^{\perp},
    &&\mathfrak{K}_{2} %&:= \left\{\mathcal{b}_2\in\Hminushalf:(\mathcal{b}_2,\surfcurl v)_{\iphalf}=0\,\,\forall v\in\Hhalfcurl\right\}\\
    :=\left(\surfcurl\,\Hhalfcurl\right)^{\perp}.\label{eq: KT2 def}
\end{align*}
It is routine to verify from \eqref{eq:ker vee vee} that $\mathfrak{K}=\ker\mathcal{B}_{\mathsf{R}}$. Hence, due to the inf-sup inequality supplied in \cite[Thm. 2.4]{leopardi2016abstract}, the problem of finding $\vecb\in\HT$ and $\vec{\bl{p}}\in\mathfrak{K}$ such that
\begin{align}
\begin{split}\label{pb: MBVR}
    \color{red}\mathcal{B}_{\mathsf{R}}\left(\vecb,\vecd\right) + \llangle\vec{\bl{p}},\vecd\rrangle_{\Gamma} &\color{red}= \ell_{\mathsf{R}}(\vecd)  \qquad\,\,\forall\,\vecd\in\HT,\\
    \color{red}\llangle\vec{\bl{b}},\vec{\bl{g}}\rrangle_{\Gamma}&\color{red}= 0 \,\qquad\qquad\forall\,\vec{\bl{g}}\in\ker\mathcal{B}_{\mathsf{R}}
    \end{split}\tag{MBV\textsf{R}}
\end{align}
is well-posed.

Similarly, the problem of solving
\begin{align}
\begin{split}\label{pb: MBVT}
    \color{blue}\mathcal{B}_{\mathsf{T}}\left(\veca,\vecc\right) + \llangle\vec{\bl{q}},\vecc\rrangle_{\Gamma} &\color{blue}= \ell_{\mathsf{T}}(\vecc),  \qquad\,\,\forall\,\vecc\in\HT,\\
    \color{blue}\llangle\vec{\bl{a}},\vec{\bl{g}}\rrangle_{\Gamma}&\color{blue}= 0, \,\qquad\qquad\forall\,\vec{\bl{g}}\in\ker\mathcal{B}_{\mathsf{T}}
    \end{split}\tag{MBV\textsf{T}}
\end{align}
for the unknown pair $(\veca,\vec{\bl{q}})\in\HR\times\ker\mathcal{B}_{\mathsf{T}}$ is well-posed.

\begin{theorem}
The mixed variational problems \eqref{pb: MBVR} has a unique solution $\vecb\in\HT$ such that $\vecb\perp\ker\mathcal{B}_{\mathsf{R}}$. Moreover,
\begin{equation}\label{eq: stability continuous BVR}
\norm{\vecb}_{-1/2}+ \norm{\vec{\bl{p}}}_{-1/2}\lesssim \norm{ \frac{1}{2}\veca - \{\traceR\}\LT(\veca)}_{\HR},
\end{equation}
    where the constant depends only on the constant in the Poincar\'e inequality \eqref{eq: Poincare inequality}. If $\veca$ satisfies \eqref{CCR}, then this result extends to the variational problem \eqref{pb: BVR} and \eqref{eq: stability continuous BVR} holds with $\vec{\bl{p}}=0$.
    
    Similarly, the mixed variational problems \eqref{pb: MBVT} has a unique solution $\veca\in\HR$ such that $\veca\perp\ker\mathcal{B}_{\mathsf{T}}$. Moreover,
\begin{equation}\label{eq: stability continuous BVT}
\norm{\veca}_{-1/2}+ \norm{\vec{\bl{q}}}_{-1/2}\lesssim \norm{ \frac{1}{2}\vecb - \{\traceT\}\LR(\vecb)}_{\HT},
\end{equation}
    where the constant depends only on the constant in the Poincar\'e inequality for $\bm{\delta}_{\Gamma}$. If $\vecb$ satisfies \eqref{CCT}, then this result extends to the variational problem \eqref{pb: BVT} and \eqref{eq: stability continuous BVT} holds with $\vec{\bl{q}}=0$.
\end{theorem}

\section{Conclusion}
First-kind boundary integral equations are appealing to the numerical analysis community because they lead to variational problems posed in natural ``energy" trace spaces that are generally well-suited for Galerkin discretization. Therefore, on the one hand, the new equations pave the way for  development of new Galerkin boundary element methods. On the other hand, our results simultaneously open a new perspective towards the recent developments in boundary integral equations for Hodge-Laplace problems. As it stands, the rich theories of Hilbert complexes and nilpotent operators not only support our observations with the help of already established abstract inf-sup conditions, but in fact also supply the framework and analysis tools needed to relate the studied non-standard surface Dirac operators to the mixed variational formulations associated with the first-kind boundary integral operators for the Hodge-Laplacian. In fact, this insight already led us to observe that the variational formulation \cite[Eq. 25]{claeys2018first} is associated with the Laplace-Beltrami of the Hilbert complex \eqref{base complex t}. We note that \cite[Eq. 34]{claeys2018first} also appears to be related to higher-order differential forms on surfaces. The significant observation that our integral operators arise as ``non-standard'' surface Dirac operators associated to trace Hilbert complexes suggests a new analysis of Hodge-Dirac and Hodge-Laplace related first-kind boundary integral equations which has yet to be explored.

\bibliographystyle{siamplain}
\bibliography{bibliography}

\begin{thebibliography}{10}

\bibitem{arnold2018finite}
{\sc D.~N. Arnold}, {\em Finite element exterior calculus}, vol.~93 of CBMS-NSF
  Regional Conference Series in Applied Mathematics, Society for Industrial and
  Applied Mathematics (SIAM), Philadelphia, PA, 2018,
  \url{https://doi.org/10.1137/1.9781611975543.ch1},
  \url{https://doi.org/10.1137/1.9781611975543.ch1}.

\bibitem{arnold2006finite}
{\sc D.~N. Arnold, R.~S. Falk, and R.~Winther}, {\em Finite element exterior
  calculus, homological techniques, and applications}, Acta Numer., 15 (2006),
  pp.~1--155, \url{https://doi.org/10.1017/S0962492906210018},
  \url{https://doi.org/10.1017/S0962492906210018}.

\bibitem{arnold2010finite}
{\sc D.~N. Arnold, R.~S. Falk, and R.~Winther}, {\em Finite element exterior
  calculus: from {H}odge theory to numerical stability}, Bull. Amer. Math. Soc.
  (N.S.), 47 (2010), pp.~281--354,
  \url{https://doi.org/10.1090/S0273-0979-10-01278-4},
  \url{https://doi.org/10.1090/S0273-0979-10-01278-4}.

\bibitem{axelsson2006transmission}
{\sc A.~Axelsson}, {\em Transmission problems for {M}axwell's equations with
  weakly {L}ipschitz interfaces}, Math. Methods Appl. Sci., 29 (2006),
  pp.~665--714, \url{https://doi.org/10.1002/mma.705},
  \url{https://doi.org/10.1002/mma.705}.

\bibitem{axelsson2001harmonic}
{\sc A.~Axelsson, R.~Grognard, J.~Hogan, and A.~McIntosh}, {\em Harmonic
  analysis of {D}irac operators on {L}ipschitz domains}, in Clifford analysis
  and its applications ({P}rague, 2000), vol.~25 of NATO Sci. Ser. II Math.
  Phys. Chem., Kluwer Acad. Publ., Dordrecht, 2001, pp.~231--246,
  \url{https://doi.org/10.1007/978-94-010-0862-4\_22},
  \url{https://doi.org/10.1007/978-94-010-0862-4_22}.

\bibitem{axelsson2006quadratic}
{\sc A.~Axelsson, S.~Keith, and A.~McIntosh}, {\em Quadratic estimates and
  functional calculi of perturbed {D}irac operators}, Invent. Math., 163
  (2006), pp.~455--497, \url{https://doi.org/10.1007/s00222-005-0464-x},
  \url{https://doi.org/10.1007/s00222-005-0464-x}.

\bibitem{axelsson2012hodge}
{\sc A.~Axelsson and A.~McIntosh}, {\em {H}odge decompositions on weakly
  {L}ipschitz domains}, in Advances in analysis and geometry, Trends Math.,
  Birkh\"{a}user, Basel, 2004, pp.~3--29.

\bibitem{buffa2001traces_a}
{\sc A.~Buffa and P.~Ciarlet, Jr.}, {\em On traces for functional spaces
  related to {M}axwell's equations. {I}. {A}n integration by parts formula in
  {L}ipschitz polyhedra}, Math. Methods Appl. Sci., 24 (2001), pp.~9--30,
  \url{https://doi.org/10.1002/1099-1476(20010110)24:1<9::AID-MMA191>3.0.CO;2-2},
  \url{https://doi.org/10.1002/1099-1476(20010110)24:1<9::AID-MMA191>3.0.CO;2-2}.

\bibitem{buffa2001traces_b}
{\sc A.~Buffa and P.~Ciarlet, Jr.}, {\em On traces for functional spaces
  related to {M}axwell's equations. {II}. {H}odge decompositions on the
  boundary of {L}ipschitz polyhedra and applications}, Math. Methods Appl.
  Sci., 24 (2001), pp.~31--48,
  \url{https://doi.org/10.1002/1099-1476(20010110)24:1<9::AID-MMA191>3.0.CO;2-2},
  \url{https://doi.org/10.1002/1099-1476(20010110)24:1<9::AID-MMA191>3.0.CO;2-2}.

\bibitem{buffa2002maxwell}
{\sc A.~Buffa, M.~Costabel, and C.~Schwab}, {\em Boundary element methods for
  {M}axwell's equations on non-smooth domains}, Numer. Math., 92 (2002),
  pp.~679--710, \url{https://doi.org/10.1007/s002110100372},
  \url{https://doi.org/10.1007/s002110100372}.

\bibitem{buffa2002traces}
{\sc A.~Buffa, M.~Costabel, and D.~Sheen}, {\em On traces for {${\bf H}({\bf
  curl},\Omega)$} in {L}ipschitz domains}, J. Math. Anal. Appl., 276 (2002),
  pp.~845--867, \url{https://doi.org/10.1016/S0022-247X(02)00455-9},
  \url{https://doi.org/10.1016/S0022-247X(02)00455-9}.

\bibitem{buffa2003galerkin}
{\sc A.~Buffa and R.~Hiptmair}, {\em Galerkin boundary element methods for
  electromagnetic scattering}, in Topics in computational wave propagation,
  vol.~31 of Lect. Notes Comput. Sci. Eng., Springer, Berlin, 2003,
  pp.~83--124, \url{https://doi.org/10.1007/978-3-642-55483-4\_3},
  \url{https://doi.org/10.1007/978-3-642-55483-4_3}.

\bibitem{christiansen2018eigenmode}
{\sc S.~H. Christiansen}, {\em On eigenmode approximation for {D}irac
  equations: differential forms and fractional {S}obolev spaces}, Math. Comp.,
  87 (2018), pp.~547--580, \url{https://doi.org/10.1090/mcom/3233},
  \url{https://doi.org/10.1090/mcom/3233}.

\bibitem{claeys2017first}
{\sc X.~Claeys and R.~Hiptmair}, {\em First-kind boundary integral equations
  for the {H}odge-{H}elmholtz operator}, SIAM J. Math. Anal., 51 (2019),
  pp.~197--227, \url{https://doi.org/10.1137/17M1128101},
  \url{https://doi.org/10.1137/17M1128101}.

\bibitem{claeys2018first}
{\sc X.~Claeys and R.~Hiptmair}, {\em First-kind {G}alerkin boundary element
  methods for the {H}odge-{L}aplacian in three dimensions}, Math. Methods Appl.
  Sci., 43 (2020), pp.~4974--4994, \url{https://doi.org/10.1002/mma.6203},
  \url{https://doi.org/10.1002/mma.6203}.

\bibitem{costabel1988boundary}
{\sc M.~Costabel}, {\em Boundary integral operators on {L}ipschitz domains:
  elementary results}, SIAM J. Math. Anal., 19 (1988), pp.~613--626,
  \url{https://doi.org/10.1137/0519043}, \url{https://doi.org/10.1137/0519043}.

\bibitem{costabel1997representation}
{\sc M.~Costabel and M.~Dauge}, {\em On representation formulas and radiation
  conditions}, Math. Methods Appl. Sci., 20 (1997), pp.~133--150,
  \url{https://doi.org/10.1002/(SICI)1099-1476(19970125)20:2<133::AID-MMA841>3.0.CO;2-Y},
  \url{https://doi.org/10.1002/(SICI)1099-1476(19970125)20:2<133::AID-MMA841>3.0.CO;2-Y}.

\bibitem{evans2010partial}
{\sc L.~C. Evans}, {\em Partial differential equations}, vol.~19 of Graduate
  Studies in Mathematics, American Mathematical Society, Providence, RI,
  second~ed., 2010, \url{https://doi.org/10.1090/gsm/019},
  \url{https://doi.org/10.1090/gsm/019}.

\bibitem{girault2012finite}
{\sc V.~Girault and P.-A. Raviart}, {\em Finite element methods for
  {N}avier-{S}tokes equations}, vol.~5 of Springer Series in Computational
  Mathematics, Springer-Verlag, Berlin, 1986,
  \url{https://doi.org/10.1007/978-3-642-61623-5},
  \url{https://doi.org/10.1007/978-3-642-61623-5}.
\newblock Theory and algorithms.

\bibitem{helsing2019dirac}
{\sc J.~Helsing and A.~Ros{\'e}n}, {\em Dirac integral equations for dielectric
  and plasmonic scattering}, arXiv preprint arXiv:1911.00788,  (2019).

\bibitem{hiptmair2003coupling}
{\sc R.~Hiptmair}, {\em Coupling of finite elements and boundary elements in
  electromagnetic scattering}, SIAM J. Numer. Anal., 41 (2003), pp.~919--944,
  \url{https://doi.org/10.1137/S0036142901397757},
  \url{https://doi.org/10.1137/S0036142901397757}.

\bibitem{Hsiao2008}
{\sc G.~C. Hsiao and W.~L. Wendland}, {\em Boundary integral equations},
  vol.~164 of Applied Mathematical Sciences, Springer-Verlag, Berlin, 2008,
  \url{https://doi.org/10.1007/978-3-540-68545-6},
  \url{https://doi.org/10.1007/978-3-540-68545-6}.

\bibitem{kirchhart2020div}
{\sc M.~Kirchhart and E.~Schulz}, {\em Div-curl problems and stream functions
  in 3d {L}ipschitz domains}, arXiv preprint arXiv:2005.11764,  (2020).

\bibitem{kress1999linear}
{\sc R.~Kress}, {\em Linear integral equations}, vol.~82 of Applied
  Mathematical Sciences, Springer-Verlag, New York, second~ed., 1999,
  \url{https://doi.org/10.1007/978-1-4612-0559-3},
  \url{https://doi.org/10.1007/978-1-4612-0559-3}.

\bibitem{leis2013initial}
{\sc R.~Leis}, {\em Initial-boundary value problems in mathematical physics},
  B. G. Teubner, Stuttgart; John Wiley \& Sons, Ltd., Chichester, 1986,
  \url{https://doi.org/10.1007/978-3-663-10649-4},
  \url{https://doi.org/10.1007/978-3-663-10649-4}.

\bibitem{leopardi2016abstract}
{\sc P.~Leopardi and A.~Stern}, {\em The abstract {H}odge-{D}irac operator and
  its stable discretization}, SIAM J. Numer. Anal., 54 (2016), pp.~3258--3279,
  \url{https://doi.org/10.1137/15M1047684},
  \url{https://doi.org/10.1137/15M1047684}.

\bibitem{MacCamy1984}
{\sc R.~C. MacCamy and E.~Stephan}, {\em Solution procedures for
  three-dimensional eddy current problems}, J. Math. Anal. Appl., 101 (1984),
  pp.~348--379, \url{https://doi.org/10.1016/0022-247X(84)90108-2},
  \url{https://doi.org/10.1016/0022-247X(84)90108-2}.

\bibitem{marmolejo2012transmission}
{\sc E.~Marmolejo-Olea, I.~Mitrea, M.~Mitrea, and Q.~Shi}, {\em Transmission
  boundary problems for {D}irac operators on {L}ipschitz domains and
  applications to {M}axwell's and {H}elmholtz's equations}, Trans. Amer. Math.
  Soc., 364 (2012), pp.~4369--4424,
  \url{https://doi.org/10.1090/S0002-9947-2012-05606-6},
  \url{https://doi.org/10.1090/S0002-9947-2012-05606-6}.

\bibitem{marmolejo2004harmonic}
{\sc E.~Marmolejo-Olea and M.~Mitrea}, {\em Harmonic analysis for general first
  order differential operators in {L}ipschitz domains}, in Clifford algebras
  ({C}ookeville, {TN}, 2002), vol.~34 of Prog. Math. Phys., Birkh\"{a}user
  Boston, Boston, MA, 2004, pp.~91--114.

\bibitem{mcintosh1999clifford}
{\sc A.~McIntosh and M.~Mitrea}, {\em Clifford algebras and {M}axwell's
  equations in {L}ipschitz domains}, Math. Methods Appl. Sci., 22 (1999),
  pp.~1599--1620,
  \url{https://doi.org/10.1002/(SICI)1099-1476(199912)22:18<1599::AID-MMA95>3.3.CO;2-D},
  \url{https://doi.org/10.1002/(SICI)1099-1476(199912)22:18<1599::AID-MMA95>3.3.CO;2-D}.

\bibitem{mcintosh2016hodge}
{\sc A.~McIntosh and S.~Monniaux}, {\em {H}odge-{D}irac, {H}odge-{L}aplacian
  and {H}odge-{S}tokes operators in {$L^p$} spaces on {L}ipschitz domains},
  Rev. Mat. Iberoam., 34 (2018), pp.~1711--1753,
  \url{https://doi.org/10.4171/rmi/1041},
  \url{https://doi.org/10.4171/rmi/1041}.

\bibitem{mclean2000strongly}
{\sc W.~McLean}, {\em Strongly elliptic systems and boundary integral
  equations}, Cambridge University Press, Cambridge, 2000.

\bibitem{monk2003finite}
{\sc P.~Monk}, {\em Finite element methods for {M}axwell's equations},
  Numerical Mathematics and Scientific Computation, Oxford University Press,
  New York, 2003,
  \url{https://doi.org/10.1093/acprof:oso/9780198508885.001.0001},
  \url{https://doi.org/10.1093/acprof:oso/9780198508885.001.0001}.

\bibitem{picard1984low}
{\sc R.~Picard}, {\em On the low frequency asymptotics in electromagnetic
  theory}, J. Reine Angew. Math., 354 (1984), pp.~50--73,
  \url{https://doi.org/10.1515/crll.1984.354.50},
  \url{https://doi.org/10.1515/crll.1984.354.50}.

\bibitem{picard1985structural}
{\sc R.~Picard}, {\em On a structural observation in generalized
  electromagnetic theory}, J. Math. Anal. Appl., 110 (1985), pp.~247--264,
  \url{https://doi.org/10.1016/0022-247X(85)90348-8},
  \url{https://doi.org/10.1016/0022-247X(85)90348-8}.

\bibitem{rosengeometric}
{\sc A.~Ros\'{e}n}, {\em Geometric multivector analysis}, [2019] \copyright
  2019, \url{https://doi.org/10.1007/978-3-030-31411-8},
  \url{https://doi.org/10.1007/978-3-030-31411-8}.
\newblock From Grassmann to Dirac.

\bibitem{sauter2010boundary}
{\sc S.~A. Sauter and C.~Schwab}, {\em Boundary element methods}, vol.~39 of
  Springer Series in Computational Mathematics, Springer-Verlag, Berlin, 2011,
  \url{https://doi.org/10.1007/978-3-540-68093-2},
  \url{https://doi.org/10.1007/978-3-540-68093-2}.
\newblock Translated and expanded from the 2004 German original.

\bibitem{schulz2020coupled}
{\sc E.~Schulz and R.~Hiptmair}, {\em Coupled domain-boundary variational
  formulations for {H}odge-{H}elmholtz operators}, arXiv preprint
  arXiv:2003.12644,  (2020).

\bibitem{schulz2020spurious}
{\sc E.~Schulz and R.~Hiptmair}, {\em Spurious resonances in coupled
  domain-boundary variational formulations of transmission problems in
  electromagnetism and acoustics}, arXiv preprint arXiv:2003.14357,  (2020).

\bibitem{steinbach2007numerical}
{\sc O.~Steinbach}, {\em Numerical approximation methods for elliptic boundary
  value problems}, Springer, New York, 2008,
  \url{https://doi.org/10.1007/978-0-387-68805-3},
  \url{https://doi.org/10.1007/978-0-387-68805-3}.
\newblock Finite and boundary elements, Translated from the 2003 German
  original.

\bibitem{taskinen2007current}
{\sc M.~Taskinen and S.~V\"{a}nsk\"{a}}, {\em Current and charge integral
  equation formulations and {P}icard's extended {M}axwell system}, IEEE Trans.
  Antennas and Propagation, 55 (2007), pp.~3495--3503,
  \url{https://doi.org/10.1109/TAP.2007.910363},
  \url{https://doi.org/10.1109/TAP.2007.910363}.

\bibitem{taskinen2007Picard}
{\sc M.~Taskinen and S.~Vanska}, {\em {P}icard’s extended {M}axwell system
  and frequency stable surface integral equations}, in 2007 Computational
  Electromagnetics Workshop, 2007, pp.~49--53,
  \url{https://doi.org/10.1109/CEM.2007.4387650}.

\bibitem{taskinen2006current}
{\sc M.~Taskinen and P.~Yl\"{a}-Oijala}, {\em Current and charge integral
  equation formulation}, IEEE Trans. Antennas and Propagation, 54 (2006),
  pp.~58--67, \url{https://doi.org/10.1109/TAP.2005.861580},
  \url{https://doi.org/10.1109/TAP.2005.861580}.

\bibitem{von1989boundary}
{\sc T.~von Petersdorff}, {\em Boundary integral equations for mixed
  {D}irichlet, {N}eumann and transmission problems}, Math. Methods Appl. Sci.,
  11 (1989), pp.~185--213, \url{https://doi.org/10.1002/mma.1670110203},
  \url{https://doi.org/10.1002/mma.1670110203}.

\end{thebibliography}
\end{document}